\documentclass{amsart}
\usepackage{amssymb}
\usepackage{mathrsfs}
\usepackage{stmaryrd}
\usepackage{bbm}
\usepackage{color}
\usepackage{oldgerm}
\usepackage[english]{babel}
\usepackage[T1]{fontenc}
\usepackage[latin1]{inputenc}
\usepackage[all]{xy}
\usepackage{hyperref}
\usepackage[all]{xy}
\usepackage{extarrows}

\newtheorem{theorem}[subsection]{Theorem}
\newtheorem{proposition}[subsection]{Proposition}

\newtheorem{lemma}[subsection]{Lemma}

\theoremstyle{definition}

\newtheorem{definition}[subsection]{Definition}
\newtheorem{question}[subsection]{Question}
\newtheorem{remark}[subsection]{Remark}

\numberwithin{equation}{subsection}

\usepackage{amsmath}
\usepackage{amsthm}
\usepackage{mathrsfs}
\usepackage{amsfonts}

\begin{document}

\title {$p$-adic GKZ hypergeometric complex}
\thanks{We would like to thank the referee for careful reading of the paper and for many 
suggestions improving the paper. The research of Lei Fu is supported by NSFC12171261 and 2021YFA 1000700.}

\author{Lei Fu, Peigen Li, Daqing Wan and Hao Zhang}
\address{Yau Mathematical Sciences Center, Tsinghua University, Beijing 100084, P. R. China}
\email{leifu@tsinghua.edu.cn}
\address{Department of Mathematics, Tsinghua University, Beijing 100084, P. R. China}
\email{lpg22@bimsa.cn}
\address{Department of Mathematics, University of California, Irvine, CA 92697}
\email{dwan@math.uci.edu}
\address{Yau Center, Southeast University, Nanjing 210096, P. R. China}
\email{hzhang21@seu.edu.cn}

\date{}
\maketitle

\begin{abstract}
To a torus action on a complex vector space, Gelfand, Kapranov and
Zelevinsky introduce a system of differential equations, which are now called the
GKZ hypergeometric system. Its solutions are GKZ hypergeometric
functions. We study the $p$-adic counterpart of the GKZ
hypergeometric system. The $p$-adic GKZ hypergeometric 
complex is a twisted relative de Rham complex of overconvergent
differential forms with logarithmic poles. It is an over-holonomic object in the derived category of 
arithmetic $\mathcal D$-modules with Frobenius structures. Traces
of Frobenius on fibers at Techm\"uller points of the GKZ hypergeometric 
complex
define the hypergeometric function over the finite
field introduced by Gelfand and Graev. 
Over the non-degenerate locus, the GKZ hypergeometric 
complex defines an overconvergent $F$-isocrystal. It is
the crystalline companion of the $\ell$-adic GKZ hypergeometric sheaf that we constructed before. Our method is 
a combination of Dwork's theory and the theory of arithmetic $\mathcal D$-modules of Berthelot. 



\end{abstract}

\section*{Introduction}

\subsection{The GKZ hypergeometric system}
Let $$A=\left(\begin{array}{ccc} w_{11}&\cdots&w_{1N}\\
\vdots&&\vdots\\
w_{n1}&\cdots&w_{nN}\end{array} \right)$$ be an $(n\times N)$-matrix
of rank $n$ with integer entries. Denote the column vectors of $A$
by ${\mathbf w}_1,\ldots, {\mathbf w}_N\in \mathbb Z^n$. It defines
an action of the $n$-dimensional torus $\mathbb T_{\mathbb
Z}^n=\mathrm{Spec}\, \mathbb Z[t_1^{\pm 1},\ldots, t_n^{\pm 1}]$ on
the $N$-dimensional affine space $\mathbb A_{\mathbb
Z}^N=\mathrm{Spec}\, \mathbb Z[x_1,\ldots, x_N]$:
$$
\mathbb T_{\mathbb Z}^n\times \mathbb A_{\mathbb Z}^N \to \mathbb
A_{\mathbb Z}^N, \quad \big((t_1,\ldots, t_n),(x_1,\ldots,
x_N)\big)\mapsto (t_1^{w_{11}}\cdots t_n^{w_{n1}}x_1,\ldots,
t_1^{w_{1N}}\cdots t_n^{w_{nN}}x_N).
$$
Let $\gamma_1,\ldots, \gamma_n\in \mathbb C$. 
In \cite{GKZ1}, Gelfand, Kapranov and Zelevinsky define
the \emph{$A$-hypergeometric system} to be the system of
differential equations
\begin{eqnarray}\label{GKZeqn}
\begin{array}{l}
\sum_{j=1}^N w_{ij} x_j\frac{\partial f}{\partial x_j}+\gamma_i
f=0 \quad (i=1,\ldots, n),\\
\prod_{\lambda_j>0} \left(\frac{\partial}{\partial x_j}\right)^{\lambda_j}f=
\prod_{\lambda_j<0} \left(\frac{\partial}{\partial x_j}\right)^{-\lambda_j}f,
\end{array}
\end{eqnarray}
where for the second system of equations, $(\lambda_1,\ldots,
\lambda_N)\in\mathbb Z^N$ goes over the family of integral linear
relations
$$\sum_{j=1}^N
\lambda_j{\mathbf w}_j=0$$ among ${\mathbf w}_1,\ldots, {\mathbf w}_N$. We now call
the $A$-hypergeometric system the \emph{GKZ hypergeometric system}. 
An  integral representation of a solution of the GKZ hypergeometric system is given by
\begin{eqnarray}\label{intrepgkz}
f(x_1,\ldots, x_N)=\int_{\Sigma} t_1^{\gamma_1}\cdots
t_n^{\gamma_n} e^{\sum_{j=1}^N x_jt_1^{w_{1j}}\cdots
t_n^{w_{nj}}}\frac{dt_1}{t_1}\cdots \frac{dt_n}{t_n}
\end{eqnarray}
where $\Sigma$ is a real $n$-dimensional cycle in $\mathbb T^n$. Confer 
\cite[equation (2.6)]{A1}, \cite[section 3]{Fu1} and
\cite[Corollary 2 in \S 4.2]{GGR1}. 

\subsection{The GKZ hypergeometric function over a finite field} 
Let $p$ be a prime number, $q$ a power of $p$, $\mathbb F_q$ the 
finite field with $q$ elements, 
$\psi:
\mathbb F_q\to\overline{\mathbb Q}^\ast$ a nontrivial additive character, and 
$\chi_1,\ldots, \chi_n:\mathbb F_q^\ast\to \overline {\mathbb Q}^\ast$
multiplicative characters. In \cite{GG}
and \cite{GGR}, Gelfand and Graev define the \emph{hypergeometric
function over the finite field} $\mathbb F_q$ to be the function defined by the family of twisted exponential sums
\begin{eqnarray}\label{finitegkz}
\mathrm{Hyp}(x_1,\ldots, x_N)
=\sum_{t_1,\ldots, t_n\in \mathbb F_q^\ast}\chi_1(t_1)\cdots
\chi_n(t_n)\psi\Big( \sum_{j=1}^N x_j t_1^{w_{1j}}\cdots
t_n^{w_{nj}}\Big),
\end{eqnarray} where $(x_1, \ldots, x_N)$ varies in $\mathbb A^N(\mathbb F_q)$. 
It is an arithmetic analogue of the expression (\ref{intrepgkz}).

In \cite{Fu2}, we introduce the $\ell$-adic GKZ hypergeometric sheaf  $\mathrm{Hyp}$
which is a perverse sheaf on $\mathbb A_{\mathbb F_q}^N$ such that for any rational point
$x=(x_1,\ldots, x_N)\in \mathbb A^N(\mathbb F_q)$, we have
\begin{eqnarray}\label{ladicsum}
\mathrm{Hyp}(x_1,\ldots, x_N)=
(-1)^{n+N}\mathrm{Tr}(\mathrm
{Frob}_x,\mathrm{Hyp}_{\bar x}),
\end{eqnarray} where
$\mathrm{Frob}_x$ is the geometric Frobenius at $x$. In this paper, we study the
crystalline companion of the $\ell$-adic GKZ hypergeometric sheaf. 

\subsection{The $p$-adic GKZ hypergeometric complex} 
For any $\mathbf v=(v_1, \ldots, v_N)\in\mathbb Z_{\geq 0}^N$ and $\mathbf w=(w_1, \ldots, w_n)\in \mathbb Z^n$, 
write $$\mathbf x^{\mathbf v}=x_1^{v_1}\cdots x_N^{v_N},\quad \mathbf t^{\mathbf w}=t_1^{w_1}\cdots t_n^{w_n}, \quad \vert\mathbf v\vert =v_1+\cdots+
 v_N,
\quad \vert\mathbf w\vert =|w_1|+\cdots+|w_n|.$$ 
Fix an algebraic closure $\overline{\mathbb Q}_p$ of $\mathbb Q_p$. Throughout this paper, all finite extensions of $\mathbb Q_p$ 
are taken inside $\overline{\mathbb Q}_p$. 
Denote by $\vert\cdot\vert$ the $p$-adic norm on $\overline{\mathbb Q}_p$ 
defined by $\vert a\vert=p^{-\mathrm{ord}_p(a)}$. Let $\pi\in\overline{\mathbb Q}_p$ be an element satisfying 
$$\pi^{p-1}+p=0,$$ $K$ a finite extension of $\mathbb Q_p$ containing $\pi$, $R$ the discrete valuation ring of $K$, and 
$\kappa=\mathbb F_q$ the residue field of $R$. To apply the theories of rigid cohomology theory and arithmetic $\mathcal D$-modules, we need to fix 
a lift to $K$ of the Frobenius substitution $x\mapsto x^q$ in $\kappa$. But on $\kappa=\mathbb F_q$, the Frobenius substitution is identity. We choose 
the identity on $K$ as the lift in this paper. 
For each real number $r>0$, consider the algebras 
\begin{eqnarray*}
K\{r^{-1}\mathbf x\}
&=&\{\sum_{\mathbf v\in\mathbb Z^N_{\geq 0}} a_{\mathbf v} \mathbf x^{\mathbf v}:\; a_{\mathbf v}\in K,\;
\vert a_{\mathbf v}\vert r^{\vert \mathbf v\vert} \hbox { are bounded}\},\\
K\langle r^{-1}\mathbf x\rangle 
&=&\{\sum_{\mathbf v\in\mathbb Z^N_{\geq 0}} a_{\mathbf v} \mathbf x^{\mathbf v}:\; a_{\mathbf v}\in K,\;
\lim_{\vert \mathbf v\vert\to \infty}\vert a_{\mathbf v}\vert r^{\vert \mathbf v\vert}=0\}. 
\end{eqnarray*} 
They are Banach $K$-algebras with respect to the norm
$$\Vert\sum_{\mathbf v\in\mathbb Z^N_{\geq 0}} a_{\mathbf v} \mathbf x^{\mathbf v}\Vert_r=\sup \vert a_{\mathbf v} \vert r^{\vert \mathbf v\vert}.$$ 
We have $K\langle r^{-1}\mathbf x\rangle\subset K\{r^{-1}\mathbf x\}.$  
Elements in $K\langle r^{-1}\mathbf x\rangle$ are exactly those power series  converging in the closed polydisc 
$$D(0, r^+)^N=\{(x_1, \ldots, x_N):\; x_i\in \overline{\mathbb Q}_p,\; \vert x_i\vert \leq r\}.$$ Moreover, for any $r<r'$, we have
$$K\{r'^{-1}\mathbf x\}\subset K\langle r^{-1}\mathbf x\rangle\subset K\{ r^{-1}\mathbf x\}.$$
Let 
$$ K\langle\mathbf x\rangle^\dagger=\bigcup_{r>1} K\{r^{-1}\mathbf x\}=\bigcup_{r>1}K\langle r^{-1}\mathbf x\rangle.$$  
$K\langle\mathbf x\rangle^\dagger$ is the ring of \emph{overconvergent}
power series, that is, series converging in closed polydiscs of radii $>1$.  The main goal of this paper is to construct the GKZ hypergeometric complex
$C^\cdot(L^\dagger)$ of $K\langle \mathbf x\rangle^\dagger$-modules with connections $\nabla$, and a Frobenius structure 
given by a horizontal morphism $$F: \mathrm{Fr}^*(C^\cdot(L^\dagger), \nabla)\to (C^\cdot(L^\dagger), \nabla),$$ where 
$\mathrm{Fr}:D(0, 1^+)^N
\to D(0, 1^+)^N$ is the morphism $(x_1, \ldots, x_N)\mapsto (x_1^q,\ldots, x_N^q)$. See Propositions \ref{estimation} and \ref{FG}. 
For any $\bar{\mathbf a}= 
(\bar a_1, \ldots, \bar a_N)\in \mathbb A^N(\mathbb F_q)$, 
let $\mathbf a=(a_1, \ldots, a_N)\in D(0,1^+)^N$ be its Techm\"uller lifting. In Theorem \ref{arithmetic}, we prove the following 
$p$-adic analogue of the formula (\ref{ladicsum}):
$$\mathrm{Hyp}(\bar a_1, \ldots, \bar a_N)=\mathrm{Tr}\Big(q^n F_{\mathbf a}^{-1}, C^\cdot(L^\dagger)\otimes^L_{K\langle \mathbf x\rangle^\dagger} K(\mathbf a)\Big),$$
where $K(\mathbf a)=K$ is regarded as a $K\langle \mathbf x\rangle^\dagger$-algebra via the homomorphism 
$$K\langle \mathbf x\rangle^\dagger\to K, \quad x_i\mapsto a_i.$$ 
Over the non-degenerate locus $U$ which we will make precise later, we prove in Theorem \ref{convergent}
that $H^i(C^\cdot(L^\dagger))|_U=0$ for $i\not =n$, and 
$(H^n(C^\cdot(L^\dagger))|_U,\nabla)$ is an overconvergent $F$-isocrystal.   
The connection $\nabla$ makes $C^\cdot(L^\dagger)$ a complex of arithmetic $\mathcal D$-modules
defined by Berthelot. In Proposition \ref{overholonomic}, we prove  $C^\cdot(L^\dagger)$ is over-holonomic in the sense of Caro.  

\medskip
We now give the detailed construction. 
Let $\delta$ be a rational convex polyhedral cone in $\mathbb R^n$ containing $\mathbf w_1, \ldots, \mathbf w_N$. 
For any real numbers $r>1$ and $s>1$, define
\begin{eqnarray*}
L(r,s) &=&\{\sum_{\mathbf w\in\mathbb Z^n\cap\delta} a_{\mathbf w}(\mathbf x) \mathbf t^{\mathbf w}:\;  a_{\mathbf w}(\mathbf x) \in K
\{r^{-1}\mathbf  x\},  
\; \Vert a_{\mathbf w}(\mathbf x)\Vert_r s^{\vert\mathbf w\vert} \hbox { are bounded}\}\\
&=&\{\sum_{\mathbf v\in\mathbb Z^N_{\geq 0},\; 
\mathbf w\in\mathbb Z^n\cap\delta} a_{\mathbf v\mathbf w} \mathbf x^{\mathbf v}\mathbf t^{\mathbf w}:\;  a_{\mathbf v\mathbf w}\in K,
\; \vert a_{\mathbf v\mathbf w}\vert r^{\vert \mathbf v\vert} s^{\vert \mathbf w\vert} \hbox { are bounded}\},\\
L^\dagger&=& \bigcup_{r>1,\;s>1} L(r,s)
\end{eqnarray*}
Note that $L(r,s)$ and 
$L^\dagger$ are rings. 
Let $\gamma=(\gamma_1, \ldots, \gamma_n)$ be an $n$-tuple of rational numbers, and let
$$F(\mathbf x, \mathbf t)=\sum_{j=1}^N x_j t_1^{w_{1j}}\cdots t_n^{w_{nj}},\quad
\mathbf t^\gamma=t_1^{\gamma_1}\cdots t_n^{\gamma_n}.$$ 
Consider the \emph{logarithmic twisted de Rham complex} $C^\cdot( L^\dagger)$ defined as follows: We set
$$C^k( L^\dagger)=
\{\sum_{1\leq i_1<\cdots < i_k\leq n}f_{i_1\ldots i_k}(\mathbf x, \mathbf t)
\frac{\mathrm dt_{i_1}}{t_{i_1}}\wedge \cdots\wedge \frac{\mathrm dt_{i_k}}{t_{i_k}}:\;f_{i_1\ldots i_k}(\mathbf x,\mathbf t)\in  L^\dagger\}$$ 
with differential $d: C^k( L^\dagger)\to C^{k+1}( L^\dagger)$ given by
\begin{eqnarray*}
d(\omega)&=&\big(\mathbf t^\gamma \exp(\pi F(\mathbf x,\mathbf t))\big)^{-1} \circ \mathrm d_{\mathbf t} \circ 
\big( \mathbf t^\gamma\exp(\pi F(\mathbf x,\mathbf t))\big)(\omega)
\\&=&\mathrm d_{\mathbf t}\omega + \sum_{i=1}^n \Big(\gamma_i+\sum_{j=1}^N \pi w_{ij}x_j\mathbf t^{\mathbf w_j}\Big)
\frac{\mathrm dt_i}{t_i}\wedge \omega
\end{eqnarray*}
for any $\omega\in C^k( L^\dagger)$, where $\mathrm d_{\mathbf t}$ is the exterior differentiation with respect to the $\mathbf t$ variable. 

Let $\Omega^{1}_{K\langle\mathbf x\rangle^\dagger}$
be the free $K\langle\mathbf x\rangle^\dagger$-module
with basis $\mathrm dx_1, \ldots, \mathrm dx_N$. We have an integrable connection
\begin{eqnarray}\label{nabla}
\nabla: C^\cdot ( L^\dagger)\to C^\cdot( L^\dagger) 
\otimes_{K\langle\mathbf x\rangle^\dagger}\Omega^1_{K\langle\mathbf x\rangle^\dagger}
\end{eqnarray}
defined by
\begin{eqnarray*}
\nabla(\omega) &=& \big(\mathbf t^\gamma \exp(\pi F(\mathbf x,\mathbf t))\big)^{-1} \circ \mathrm d_{\mathbf x} \circ 
\big( \mathbf t^\gamma\exp(\pi F(\mathbf x,\mathbf t))\big)(\omega)\\
&=& \mathrm d_{\mathbf x}\omega+\pi \mathrm d_{\mathbf x}F\wedge \omega\\
&=& \mathrm d_{\mathbf x}\omega+\sum_{j=1}^N\pi \mathbf t^{\mathbf w_j} \mathrm dx_j \wedge \omega,
\end{eqnarray*}
where $\mathrm d_{\mathbf x}$ is the exterior differentiation with respect to the $\mathbf x$ variable.
Since $\mathrm d_{\mathbf x}$ commutes with $\mathrm d_{\mathbf t}$,  
$\nabla$ commutes with $d: C^k( L^\dagger)\to C^{k+1}( L^\dagger)$, and hence $\nabla$ induces connections on cohomology
groups of $C^\cdot( L^\dagger)$.
 
Consider the Frobenius map in the variable $\mathbf t$ defined by 
$$\Phi(f({\mathbf x},{\mathbf t}))=f(\mathbf x, \mathbf t^q).$$
One verifies directly that
$\Phi(L(r, s))\subset L(r, \sqrt[q]{s})$ and hence $\Phi(L^\dagger)\subset L^\dagger$. 
It induces maps $\Phi: C^k( L^\dagger)\to C^k( L^\dagger)$ on differential forms commuting with $\mathrm d_{\mathbf t}$: 
$$\Phi\Big(\sum_{1\leq i_1<\cdots < i_k\leq n}f_{i_1\ldots i_k}(\mathbf x, \mathbf t)\frac{\mathrm dt_{i_1}}{t_{i_1}}\wedge \cdots\wedge \frac{\mathrm dt_{i_k}}
{t_{i_k}}\Big)
=\sum_{1\leq i_1<\cdots < i_k\leq n}q^kf_{i_1\ldots i_k}(\mathbf x, \mathbf t^q)\frac{\mathrm dt_{i_1}}{t_{i_1}}\wedge \cdots\wedge 
\frac{\mathrm dt_{i_k}}{t_{i_k}}.$$
Suppose furthermore that
$(q-1)\gamma\in \mathbb Z^n\cap \delta$.  Consider the maps $F: C^k( L^\dagger)\to C^k( L^\dagger)$ defined by
\begin{eqnarray}\label{FROB}
F&=&\big(\mathbf t^\gamma \exp (\pi F(\mathbf x, \mathbf t))\big)^{-1}
\circ\Phi\circ\big(\mathbf t^\gamma \exp (\pi F(\mathbf x^{q}, \mathbf t))\big)\\
\label{F} &=&\mathbf t^{(q-1)\gamma}\exp \big(\pi F(\mathbf x^q,\mathbf t^q)-
\pi F(\mathbf x, \mathbf t)\big)
\circ\Phi.
\end{eqnarray}
Even though $\mathbf t^\gamma\exp (\pi F(\mathbf x, \mathbf t))$ does not lie in 
$L^\dagger$ and multiplication by it does not define an endomorphism on $C^\cdot(L^\dagger)$, Proposition
\ref{estimation} (i) below
shows that $\exp \big(\pi F(\mathbf x^q,\mathbf t^q)-
\pi F(\mathbf x, \mathbf t)\big)$ lies in $L^\dagger$, and hence the expression (\ref{F}) shows that 
$F$ defines an endomorphism on each $C^k( L^\dagger)$. 

\begin{proposition}\label{estimation} ${}$

(i) Let $w=\max(\vert\mathbf w_1\vert, \ldots, \vert\mathbf w_N\vert)$. Both $\exp \big(\pi F(\mathbf x^q,\mathbf t^q)-
\pi F(\mathbf x, \mathbf t)\big)$ and $\exp\big(\pi F(\mathbf x,\mathbf t)- \pi F(\mathbf x^q,\mathbf t^{q})\big)$ 
lie in $L(r,r^{-\frac{1}{w}} p^{\frac{p-1}{pqw}})$ for any $1<r\leq p^{\frac{p-1}{pq}}$.

(ii) Suppose $(q-1)\gamma\in \mathbb Z^n\cap \delta$. Consider the $K$-algebra homomorphism 
$$\mathrm{Fr}^\natural: K\langle \mathbf x\rangle\to K\langle \mathbf x\rangle, \quad x_j\mapsto x_j^q.$$ 
Let $$\mathrm{Fr}^*(C^\cdot( L^\dagger), \nabla)=(C^\cdot( L^\dagger), \nabla)\otimes_{K\langle \mathbf x\rangle^\dagger, \mathrm{Fr}^\natural} 
K\langle \mathbf x\rangle^\dagger.$$ Then $F$ defines a horizontal morphism of complexes of $K\langle \mathbf x\rangle^\dagger$-modules with connections
$$F: \mathrm{Fr}^*(C^\cdot( L^\dagger), \nabla) \to (C^\cdot( L^\dagger), \nabla).$$
\end{proposition}
 
Let $\Psi'$ be the substitution $\mathbf t\mapsto \mathbf t^{\frac {1}{q}}$:
$$\Psi'\big(\sum_{\mathbf w}a_{\mathbf w}(\mathbf x)\mathbf t^{\mathbf w}\big)
=\sum_{\mathbf w}a_{\mathbf w}(\mathbf x)\mathbf t^{\frac{\mathbf w}{q}},$$
and let 
$$\mathrm{pr}:K\langle\mathbf x\rangle^\dagger\big[\big[\frac{1}{q}\mathbb Z^n\cap\delta\big]\big]\to 
K\langle\mathbf x\rangle^\dagger[[\mathbb Z^n\cap\delta]]
$$ be the $K\langle\mathbf x\rangle^\dagger$-linear 
map defined by 
$$\mathrm{pr}(\mathbf t^{\mathbf w})= \left\{ \begin{array}{cl}
\mathbf t^{\mathbf w}&\hbox{if } \mathbf w\in \mathbb Z^n\cap\delta, \\
0&\hbox{otherwise}.
\end{array}\right.$$ Consider the operator $\Psi: L^\dagger \to 
 L^\dagger$ defined by 
$$\Psi=\mathrm{pr}\circ \Psi'.$$
We extend $\Psi'$ and $\Psi$ to differential forms by
\begin{eqnarray*}
\Psi'\Big(\sum_{1\leq i_1<\cdots < i_k\leq n}f_{i_1\ldots i_k}(\mathbf x, \mathbf t)\frac{\mathrm dt_{i_1}}{t_{i_1}}\wedge \cdots\wedge \frac{\mathrm dt_{i_k}}
{t_{i_k}}\Big)
&=&\sum_{1\leq i_1<\cdots < i_k\leq n}q^{-k} f_{i_1\ldots i_j}(\mathbf x,\mathbf t^{\frac{1}{q}})\frac{\mathrm dt_{i_1}}{t_{i_1}}\wedge \cdots\wedge 
\frac{\mathrm dt_{i_k}}{t_{i_k}},\\
\Psi\Big(\sum_{1\leq i_1<\cdots < i_k\leq n}f_{i_1\ldots i_k}(\mathbf x, \mathbf t)\frac{\mathrm dt_{i_1}}{t_{i_1}}\wedge \cdots\wedge \frac{\mathrm dt_{i_k}}
{t_{i_k}}\Big)
&=&\sum_{1\leq i_1<\cdots < i_k\leq n}q^{-k} \Psi(f_{i_1\ldots i_j}(\mathbf x,\mathbf t))\frac{\mathrm dt_{i_1}}{t_{i_1}}\wedge \cdots\wedge 
\frac{\mathrm dt_{i_k}}{t_{i_k}}.
\end{eqnarray*}
They commute with $\mathrm d_{\mathbf t}$. Suppose $(1-q)\gamma\in \mathbb Z^n \cap \delta$. 
Define \emph{the Dwork opertor} $G: C^\cdot( L^\dagger)
\to C^\cdot( L^\dagger)$ by 
\begin{eqnarray*}
G&=&\mathrm{pr}\circ \big( \mathbf t^\gamma \exp(\pi F(\mathbf x^q,\mathbf t))\big)^{-1} \circ \Psi' \circ 
\big(\mathbf t^\gamma \exp(\pi F(\mathbf x,\mathbf t))\big)\\
&= & \Psi \circ \big(\mathbf t^{(1-q)\gamma}\exp\big(\pi F(\mathbf x,\mathbf t)- 
\pi F(\mathbf x^q, \mathbf t^{q})\big)\big).
\end{eqnarray*}
By Proposition \ref{estimation} (i), $\exp\big(\pi F(\mathbf x,\mathbf t)- 
\pi F(\mathbf x^q, \mathbf t^{q})\big)$ lies in $L^\dagger$ and hence $G$ defines an operator on $ C^\cdot(L^\dagger)$. 

\begin{proposition}\label{FG} ${}$

(i) Suppose $(1-q)\gamma\in \mathbb Z^n \cap \delta$. Then $G$ defines a horizontal morphism of complexes of 
$K\langle \mathbf x\rangle^\dagger$-modules with connections
$$G: (C^\cdot(  L^\dagger), \nabla)\to \mathrm{Fr}^*(C^\cdot(  L^\dagger), \nabla).$$

(ii) Suppose $\pm (1-q)\gamma\in \mathbb Z^n \cap \delta$. 
We have $G\circ  F=\mathrm{id}$. Moreover, $F$ and $G$ induce isomorphisms 
on cohomology groups inverse to each other.
\end{proposition} 

\begin{definition} Suppose $\pm (1-q)\gamma \in \mathbb Z^n\cap \delta$.
The \emph{$p$-adic GKZ hypergeometric complex} is defined to be the tuple
$(C^\cdot(  L^\dagger), \nabla, F)$ consisting of the complex $C^\cdot(  L^\dagger)$ of $K\langle\mathbf x\rangle^\dagger$-modules with the connection $\nabla$ and the \emph{Frobenius structure} defined by 
the horizontal morphism $F: \mathrm{Fr}^*(C^\cdot(  L^\dagger), \nabla)\to 
(C^\cdot(  L^\dagger), \nabla).$
\end{definition}

Our construction of $C^\cdot(L^\dagger)$ makes sense under the assumption that 
$\delta$ is a rational convex polyhedral cone containing $\mathbf w_1, \ldots, \mathbf w_N$ and 
$\pm (1-q)\gamma\in\mathbb Z^n\cap\delta$. But in order for 
$C^\cdot(L^\dagger)$ to have the required properties, from now on, unless we state otherwise,  we take 
$$\delta=\{\lambda_1\mathbf w_1+\cdots+\lambda_N \mathbf w_N:\; \lambda_j\geq 0\}$$
 and we assume either $\gamma=0$ or $0$ lies in the interior of the convex hull 
of $\{0, \mathbf w_1, \ldots, \mathbf w_N\}$. 

\subsection{The GKZ hypergeometric $D^\dagger$-module} 

We study the GKZ hypergeometric complex using the arithmetic $\mathcal D$-module theory.
Let $\mathscr P_R^N$ be the formal projective space obtained by taking the formal completion of the projective
space $\mathbb P^N_R$ along the special fiber. We often omit the subscript $R$ in $\mathscr P_R^N$ for convenience. 
Let $\mathcal D^\dagger_{\mathscr P^N,\mathbb Q}(\infty)$ be the sheaf of differential operators of on $\mathscr P^N$ with singularities 
overconvergent along $\infty$. Here we fix a homogeneous coordinate system $[x_0:\cdots:x_N]$ on $\mathscr P^N$ and $\infty$ is the 
divisor defined by $x_0=0$. For the definition of this sheaf, see  \cite{B} or  \ref{d-module}. By \cite{Hu}, we have 
$\Gamma(\mathscr P^N, \mathcal D^\dagger_{\mathscr P^N, \mathbb Q}(\infty))=D^\dagger$, where 
$$D^\dagger=\bigcup_{r>1,\;s>1}
\{\sum_{\mathbf v\in\mathbb Z_{\geq 0}^N} f_{\mathbf v}(\mathbf x)\frac{\partial^{\mathbf v}}{\mathbf v_!}:\;
f_{\mathbf v}(\mathbf x)\in K\{ r^{-1}\mathbf x\},\; 
\Vert f_{\mathbf v}(\mathbf x)\Vert_r s^{\vert \mathbf v\vert} \hbox { are bounded}\},$$ 
$\partial^{\mathbf v}= \frac{\partial^{v_1+\cdots+v_N}}{\partial x_1^{v_1}\cdots\partial x_N^{v_N}}$ and $\mathbf v!=v_1!\cdots v_N!$ for any 
$\mathbf v=(v_1, \ldots, v_N)\in \mathbb Z^N_{\geq 0}$. In Proposition \ref{two_defn_D}, we show that 
$$D^\dagger=\bigcup_{r>1,\;s>1}
\{\sum_{\mathbf v\in\mathbb Z_{\geq 0}^N} f_{\mathbf v}(\mathbf x)\frac{\partial^{\mathbf v}}{\pi^{\vert \mathbf v\vert}}:\;
f_{\mathbf v}(\mathbf x)\in K\{ r^{-1}\mathbf x\},\; 
\Vert f_{\mathbf v}(\mathbf x)\Vert_r s^{\vert \mathbf v\vert} \hbox { are bounded}\}.$$ 
By a result in \cite{Hu}, $D^\dagger$ is a coherent ring. 
Let $\frac{\partial}{\partial x_j}\in D^\dagger$ act on $L^\dagger$ via 
\begin{eqnarray*}
\nabla_{\frac{\partial}{\partial x_j}}=\big(\mathbf t^\gamma \exp (\pi F(\mathbf x,\mathbf t))\big)^{-1}
\circ  \frac{\partial}{\partial x_j} \circ  \big(\mathbf t^\gamma\exp (\pi F(\mathbf x,\mathbf t))\big)
=\frac{\partial}{\partial x_j}+\pi \mathrm t^{\mathbf w_j}.
\end{eqnarray*} 

\begin{proposition}\label{coherent} 
$L^{\dagger}$ is
a coherent left $D^\dagger$-modules.
\end{proposition}

Thus $C^\cdot( L^{\dagger})$ is
a complex of coherent $D^\dagger$-modules. 
Denote by $D^b_{coh}(\mathcal D^\dagger_{\mathscr P^N, \mathbb Q}(\infty))$ the derive category of complexes of 
$\mathcal D^\dagger_{\mathscr P^N, \mathbb Q}(\infty)$-modules with coherent cohomology. By \cite[Th\'eor\`eme 5.3.3]{Hu} and 
Proposition \ref{coherent},  $C^\cdot(L^\dagger)$ defines an object in $D^b_{coh}(\mathcal D^\dagger_{\mathscr P^N, \mathbb Q}(\infty))$, which we denote 
by $C^\cdot(\mathcal L^{\dagger})$. 
Caro \cite[D\'efinition 3.1]{Overhol} defines $D^b_{ovhol}(\mathcal D^\dagger_{\mathscr P^N, \mathbb Q}(\infty))$ as the subcategory of 
$D^b_{coh}(\mathcal D^\dagger_{\mathscr P^N, \mathbb Q}(\infty))$ consisting of over-holonomic complexes. Denote by 
$F\hbox{-}D^b_{ovhol}(\mathcal D^\dagger_{\mathscr P^N, \mathbb Q}(\infty))$ 
the subcategory of 
$D^b_{ovhol}(\mathcal D^\dagger_{\mathscr P^N, \mathbb Q}(\infty))$ consisting of those objects with Frobenius structures
and morphisms compatible with the Frobenius structures. (Confer \cite[D\'efinition 5.1.1]{B2})

\begin{proposition}\label{overholonomic} 
$C^\cdot(\mathcal L^{\dagger})$  is an object in $F\hbox{-}D^b_{ovhol}(\mathcal D^\dagger_{\mathscr P^N, \mathbb Q}(\infty)).$
\end{proposition}

\begin{remark} Let $\mathscr A^N$ and $\mathscr T^n$ be the formal affine space and the formal torus, let
$$\mathrm{pr}_1: \mathscr T^n\times \mathscr A^N\to \mathscr T^n,\quad 
\mathrm{pr}_2: \mathscr T^n\times \mathscr A^N\to \mathscr A^N$$ be the projections and let 
$F: \mathscr T^n\times \mathscr A^N \to\mathscr A^1$ be the morphism defined by $F(\mathbf t, \mathbf x)$. One can use the method in
section 4 to show that 
$$C^\cdot(\mathcal L^{\dagger})\cong \mathrm{pr}_{2, +}(F^+\mathcal L_\psi \otimes\mathrm{pr}_1^+\mathcal K_\chi)[2N-1],$$ 
where $\mathcal L_\psi$ is the Dwork isocrystal  associated to $\psi$ and $\mathcal K_\chi$ is the Kummer
isocrystal associated to $\gamma$. The six functor formalism of arithmetic $\mathcal D$-modules then implies Proposition 
\ref{overholonomic}. In this paper, we don't use this approach and prove Proposition \ref{overholonomic} 
directly. 
\end{remark}

\begin{remark} The left $D^\dagger$-module
$$D^\dagger/(\sum_{i=1}^n D^\dagger E_{i,\gamma}+\sum_{\lambda\in \Lambda}D^\dagger \Box_\lambda)$$
is the $p$-adic analogue of the (complex) hypergeometric $D$-module (\cite{A1}) associated to the GKZ hypergeometric system of differential equations
(\ref{GKZeqn}). However, this $D$-module is not directly related to $C^\cdot(L^\dagger)$. 
In \ref{GKZprime}, we introduce a subcomplex $C^\cdot(L^{\dagger\prime})$ of $C^\cdot(L^{\dagger})$, and prove that 
$$H^n(C^\cdot(L^{\dagger\prime}))\cong D^\dagger/(\sum_{i=1}^n D^\dagger E_{i,\gamma}+\sum_{\lambda\in \Lambda}D^\dagger \Box_\lambda).$$ 
See Proposition \ref{grobner}. 
In general, the action of $G$ on $C^\cdot(L^{\dagger})$ does not preserve $C^\cdot(L^{\dagger\prime})$, and the above $D^\dagger$-module does not seem to
have a Frobenius structure with interesting arithmetic properties. 
\end{remark}

\subsection{The GKZ hypergeometric $F$-isocrystal} 

Let 
$$F_{i, \gamma}=\big(\mathbf t^\gamma\exp(\pi F(\mathbf x, \mathbf t))\big)^{-1}\circ t_i\frac{\partial}{\partial t_i}\circ 
\big(\mathbf t^\gamma\exp(\pi F(\mathbf x, \mathbf t))\big)=t_i\frac{\partial}{\partial t_i}+\gamma_i+\pi\sum_{j=1}^N w_{ij}x_j\mathbf t^{\mathbf w_j}.$$
It follows from the definition of the logarithmic twisted de Rham complex that the homomorphism 
$$ L^{\dagger}\to C^n( L^{\dagger}),\quad f(\mathbf x, \mathbf t)\mapsto f(\mathbf x, \mathbf t)\frac{\mathrm dt_1}{t_1}\wedge \cdots\wedge \frac{\mathrm dt_n}{t_n}$$
induces an isomorphism
 $$ L^{\dagger}
/ \sum_{i=1}^n F_{i,\gamma} L^{\dagger}\cong H^n(C^\cdot( L^{\dagger})).
$$ The connection (\ref{nabla})  defines a connection $\nabla$ on $H^n(C^\cdot( L^\dagger))$, and the morphism (\ref{FROB}) 
defines a horizontal isomorphism 
$$F: \mathrm{Fr}^* (H^n(C^\cdot( L^\dagger)),\nabla)\to   (H^n(C^\cdot( L^\dagger)),\nabla).$$

\begin{definition}
Let $\Delta$ be the convex hull of $\{0,\mathbf w_1, \ldots,\mathbf w_N\}$ in $\mathbb R^n$ and $\kappa=\mathbb F_q$ the residue field of the integer ring
of $K$. The \emph{nondegenerate locus} $U_0$ is the Zariski open subset of $\mathbb A^n_\kappa$ consisting of 
those points $\bar{\mathbf a}=(\bar a_1, \ldots,\bar a_N)$ so that 
$F(\bar {\mathbf a}, \mathbf t)=\sum_{j=1}^N \bar a_j \mathbf  t^{\mathbf w_j}$ 
is \emph{non-degenerate} in the sense that for any face $\tau$ of $\Delta$ not containing the origin, the system of equations 
$$\frac{\partial}{\partial t_1}F_\tau(\bar{\mathbf a},\mathbf t)=\cdots =\frac{\partial}{\partial t_n}F_\tau(\bar{\mathbf a},\mathbf t)=0$$ has no solution 
in $(\bar \kappa^\ast)^n$, where $F_\tau(\bar{\mathbf a},\mathbf t)=\sum_{\mathbf w_j\in \tau}\bar a_j \mathbf t^{\mathbf w_j}$. 
\end{definition}

\subsection{} Let $U=\text{sp}^{-1}(U_0)$, where $$\text{sp}: D(0,1^+)^N\to \mathbb A_\kappa^N$$ is the specialization morphism. 
Let $X$ be an open subset of $\mathbb P_\kappa^N$ such that its complement $T$ is a divisor. Let $F\text{-isoc}^\dagger(X)$ be the category of 
$F$-isocrystals on $X$ overconvergent along $T$ as defined in  \cite[D\'efinition 2.3.6]{B3}, and let 
$F\text{-Coh} (\mathcal D^\dagger_{\mathscr P^N, \mathbb Q}(^\dagger T))$ be the category of coherent 
$\mathcal D^\dagger_{\mathscr P^N, \mathbb Q}(^\dagger T)$-modules with Frobenius structures. 
We have a fully faithful functor induced by the specialization morphism
$$\text{sp}_\ast: F\text{-isoc}^\dagger(X) \to F\text{-Coh} (\mathcal D^\dagger_{\mathscr P^N, \mathbb Q}(^\dagger T)).$$
By \cite[Th\'eor\`eme 2.2.12]{Caro} 
the essential image consists of coherent $\mathcal D^\dagger_{\mathscr P^N, \mathbb Q}(^\dagger T)$-modules with Frobenius structures 
whose restriction to $\mathscr P^N \backslash T$ is
$(\mathcal O_{\mathscr P^N ,\mathbb Q})|_{\mathscr P^N \backslash T}$-coherent, where  
$\mathcal O_{\mathscr P^N,\mathbb Q}=\mathcal O_{\mathscr P^N}\otimes_{\mathbb Z}\mathbb Q$. We also call an object in 
$F\text{-Coh} (\mathcal D^\dagger_{\mathscr P^N, \mathbb Q}(^\dagger T))$ which lies in the essential image of 
$\mathrm{sp}_*$ a convergent $F$-isocrystal on $X$ overconvergent 
along $T$, or simply an overconvergent $F$-isocrystal on $X$.

\begin{theorem}\label{convergent} ${}$

(i) $H^k(C^\cdot(\mathcal L^{\dagger}))|_{U_0}=0$ for $k\ne n$.

(ii) The cohomology $H^n(C^\cdot(\mathcal L^\dagger))$ defines an overconvergent $F$-isocrystal. More precisely, for any divisor $T$ containing 
$\mathbb P_k^N\backslash U_0$, $(^\dagger T)(H^n(C^\cdot(\mathcal L^{\dagger})))
:=\mathcal D^\dagger_{\mathscr P^N, \mathbb Q}(^\dagger T)\otimes_{\mathcal D^\dagger
_{\mathscr P^N, \mathbb Q}(\infty)} H^n(C^\cdot(\mathcal L^{\dagger})))$ is a convergent $F$-isocrystal on 
$\mathbb P_\kappa^N\backslash T$ overconvergent along $T$. 
\end{theorem}

\begin{definition} We define the \emph{GKZ hypergeometric F-isocrystal} to be $(^\dagger T)(H^n(C^\cdot (\mathcal L^\dagger)))$
with $T=\mathbb P_\kappa^N-U_0$, 
and denote it by ${\mathrm{Hyp}}_{U_0}$. 
\end{definition}

\subsection{Fibers of the GKZ hypergeometric complex}\label{aaa}

In number theory, we often need to study the twisted exponential sum (\ref{finitegkz}) for a rational point $\mathbf x=(x_1, \dots, x_N)$
in $\mathbb A^N_{\kappa}$ with coordinates $x_i$ lying in a finite extension of $\kappa$, and we want to study how the exponential sum
and its $L$-function vary if the rational point $x$ varies in $\mathbb A^N_{\kappa}$. Such kind of information can be read from the fibers of the 
GKZ hypergeometric complex. 
Let $\mathbf a=(a_1, \ldots, a_N)$ be a point in the closed unit polydisc $D(0,1^+)^N$, where $a_i\in K'$ for some finite extension $K'$ of $K$, and 
$a_i$ lie in the discrete valuation ring $R'$ of $K'$. 
This point defines a closed immersion 
$i_{\mathbf a}: \mathrm{Spf}\, R' \to \mathscr P^N$ of formal schemes.
Let
\begin{eqnarray*}
L^\dagger_0&=&\bigcup_{s>1} \{\sum_{\mathbf w\in\mathbb Z^n\cap\delta } a_{\mathbf w} t^{\mathbf w}:\;  a_{\mathbf w} \in K', \; 
\vert a_{\mathbf w}\vert s^{\vert \mathbf w\vert} \hbox { are bounded} \}.
\end{eqnarray*}
Consider the  logarithmic twisted de Rham complex $C^\cdot(  L_0^\dagger)$ such that 
$$C^k(  L^\dagger_0)=
\{\sum_{1\leq i_1<\cdots < i_k\leq n}f_{i_1\ldots i_k}(\mathbf t)\frac{\mathrm dt_{i_1}}{t_{i_1}}\wedge \cdots\wedge \frac{\mathrm dt_{i_k}}{t_{i_k}}:
\;f_{i_1\ldots i_k}(\mathbf t)\in   L^\dagger_0\}$$ and the differential $d: C^k(  L_0^\dagger)\to 
C^{k+1}(  L_0^\dagger)$ is given by
\begin{eqnarray*}
d(\omega)&=&\big(\mathbf t^\gamma \exp(\pi F(\mathbf a,\mathbf t))\big)^{-1} \circ \mathrm d_{\mathbf t} \circ 
\big( \mathbf t^\gamma\exp(\pi F(\mathbf a,\mathbf t))\big)(\omega)
\\&=&\mathrm d_{\mathbf t}\omega + \sum_{i=1}^n\Big(\gamma_i+\pi\sum_{j=1}^N w_{ij}a_j\mathbf t^{\mathbf w_j}\Big)
\frac{\mathrm dt_i}{t_i}\wedge \omega
\end{eqnarray*}
for any $\omega\in C^k(  L_0^\dagger)$. Even though the notation $C^\cdot(L_0^\dagger)$ does not involve $\mathbf a$, the complex 
$C^\cdot(L_0^\dagger)$ depends on $\mathbf a$ since the differential $d$ does. 

\begin{proposition}\label{newfiber} ${}$ 

(i) 
For any $\mathbf a=(a_1, \ldots, a_N)\in D(0,1^+)^N$,  $H^k(C^{\cdot}(L_0^\dagger))$ are finite dimensional. If $\mathbf a=(a_1, \ldots, a_N)$ lies in $U$, then
$$\mathrm{dim}\, H^k(C^\cdot(L^\dagger_0))\cong \left\{\begin{array}{cc}
0&\hbox{if } k\not =n ,\\
n!\mathrm{vol}(\Delta) &\hbox{if }k=n.
\end{array}
\right.$$ 

(ii) For any $\mathbf a$ in $D(0, 1^+)^N$, we have isomorphisms 
$$i^!_{\mathbf a}C^\cdot(\mathcal L^\dagger)[N] \cong C^\cdot(  L^\dagger)\otimes^L_{K\langle \mathbf x\rangle^\dagger}
K'\cong C^\cdot (  L^\dagger_0),$$ where $i_{\mathbf a}^!$ is the extraordinary pullback operator in the derived category 
of arithmetic $\mathcal D$-modules,
$K'$ is regarded as a $K\langle \mathbf x\rangle^\dagger$-algebra
via the homomorphism $$K\langle \mathbf x\rangle^\dagger\to K',\quad x_j\mapsto a_j.$$
\end{proposition}

\subsection{} In this subsection, we just assume that $\delta$ is a rational convex polyhedral cone containing $\mathbf w_1, \ldots, \mathbf w_N$
such that $(1-q)\gamma\in\mathbb Z^n\cap \delta$. Define $\Psi_{\mathbf a}:  L^\dagger_0\to   L^\dagger_0$ by $$\Psi_{\mathbf a} ( \sum_{\mathbf w\in\mathbb Z^n\cap\delta}
c_{\mathbf w}\mathbf t^{\mathbf w})= \sum_{\mathbf w\in\mathbb Z^n}
c_{q\mathbf w}\mathbf t^{\mathbf w}.$$ 
On differential forms, define
$$\Psi_{\mathbf a}\Big(\sum_{1\leq i_1<\cdots < i_k\leq n}f_{i_1\ldots i_k}(\mathbf t)\frac{\mathrm dt_{i_1}}{t_{i_1}}\wedge \cdots\wedge \frac{\mathrm dt_{i_k}}
{t_{i_k}}\Big)
=\sum_{1\leq i_1<\cdots < i_k\leq n}q^{-k} \Psi_{\mathbf a}(f_{i_1\ldots i_j}(\mathbf t))\frac{\mathrm dt_{i_1}}{t_{i_1}}\wedge \cdots\wedge 
\frac{\mathrm dt_{i_k}}{t_{i_k}}.$$
Define
\begin{eqnarray*}
G_{\mathbf a}&=&\big(\mathbf t^\gamma \exp(\pi F(\mathbf a^q,\mathbf t))\big)^{-1} \circ \Psi_{\mathbf a} \circ 
\big(\mathbf t^\gamma\exp(\pi F(\mathbf a,\mathbf t))\big)\\
&= & \Psi_{\mathbf a} \circ \big(\mathbf t^{(1-q)\gamma}\exp\big(\pi F(\mathbf a,\mathbf t)- 
\pi F(\mathbf a^q, \mathbf t^{q})\big)\big).
\end{eqnarray*}
By Proposition \ref{estimation} (i), $\exp\big(\pi F(\mathbf a,\mathbf t)- 
\pi F(\mathbf a^q, \mathbf t^{q})\big)$ lies in $L_0^\dagger$ and hence $G_{\mathbf a}$ defines an operator on $C^\cdot( L^\dagger_0)$.

From now on, we assume that $\mathbf a$ is a Techm\"uller point, that is, $a_j^q=a_j$ $(j=1, \ldots, N)$. Then $\mathbf a$ is a fixed point of 
$$\mathrm{Fr}:D(0,1^+)^N\to D(0,1^+)^N,\quad (x_1, \ldots, x_N)\mapsto (x_1^q,\ldots, x_N^q).$$ 
In this case $G_{\mathbf a}$ commutes with $d: C^k(  L^\dagger_0)
\to C^{k+1}(  L^\dagger_0)$ and hence is a chain map. 
We will show that each $G_{\mathbf a}:C^k( L^\dagger_0)\to C^k( L^\dagger_0)$ is a nuclear operator and hence 
the homomorphism on each $H^k(C^\cdot( L^\dagger_0))$ induced by $G_{\mathbf a}$ is also nuclear. We can talk about 
their traces and characteristic power series. Let 
\begin{eqnarray*}
\mathrm{Tr}\big(G_{\mathbf a}, C^\cdot( L^\dagger_0)\big)&=&\sum_{k=0}^n (-1)^k 
\mathrm{Tr}\big( G_{\mathbf a}, C^k( L^\dagger_0)\big)\\
&=&\sum_{k=0}^n (-1)^k \mathrm{Tr}\big( G_{\mathbf a} ,H^k( C^\cdot(L^\dagger_0))\big),\\
\mathrm{det}\big(I-T G_{\mathbf a}, C^\cdot( L^\dagger_0)\big)&=&\prod_{k=0}^n 
\mathrm{det}\big(I-TG_{\mathbf a}, C^k( L^\dagger_0)\big)^{(-1)^k}\\
&=& \prod_{k=0}^n 
\mathrm{det}\big(I-T G_{\mathbf a}, H^k(C^\cdot( L^\dagger_0))\big)^{(-1)^k}.
\end{eqnarray*}

Let $\chi:\mathbb F_q^*\to \overline{\mathbb Q}_p^*$ be the Techm\"uller character which maps each $u$ in $\mathbb F_q^*$ to its Techm\"uller 
lifting. By \cite[Theorems 4.1 and 4.3]{M1}, the formal power series $\theta(z)=\exp(\pi z-\pi z^p)$ converges in a disc of radius $>1$, and its value $\theta(1)$ at $z=1$ 
is a primitive $p$-th root of unity in $K$. Let $\psi:\mathbb F_q\to K^*$ be the additive character defined by 
$$\psi(\bar a)=\theta(1)^{\mathrm{Tr}_{\mathbb F_q/\mathbb F_p}(\bar a)}$$ for any $\bar a\in \mathbb F_q$. 
Let $\bar a_j\in\mathbb F_q$ be the residue class of
$a_j$, let
\begin{eqnarray*}
&&S_m(F(\bar{\mathbf a},\mathbf t))\\
&=&\sum_{\bar u_1,\ldots, \bar u_n\in \mathbb F_{q^m}^*}\chi_1 (\mathrm{Norm}_{\mathbb F_q^m/\mathbb F_q}(\bar u_1))
\cdots\chi_n(\mathrm{Norm}_{\mathbb F_q^m/\mathbb F_q}(\bar u_n)) \psi\Big 
(\mathrm{Tr}_{\mathbb F_{q^m}/\mathbb F_q}\Big(\sum_{j=1}^N \bar a_j \bar u_1^{w_{1j}}\cdots \bar u_n^{w_{nj}}\Big)\Big)
\end{eqnarray*}
be the twisted exponential sum for the multiplicative characters $\chi_i=\chi^{(1-q)\gamma_i}$, the nontrivial additive 
character $\psi:\mathbb F_q\to \overline{\mathbb Q}_p^*$, and the polynomial $F(\bar{\mathbf a},\mathbf t)$, and let
$$L(F(\mathbf {\bar a}, \mathbf t),T)=\exp\Big(\sum_{m=1}^\infty S_m(F(\bar{\mathbf a},\mathbf t))\frac{T^m}{m}\Big)$$ be the $L$-function
for the twisted exponential sums. Note that for $m=1$, we have 
$$S_m(F(\bar{\mathbf a}, \mathbf t))=\mathrm{Hyp}(\bar a_1, \ldots, \bar a_n),$$ where $\mathrm{Hyp}(x_1, \ldots, x_N)$ is given by 
(\ref{finitegkz}).

\begin{theorem} \label{arithmetic}
Suppose $\delta$ is a rational convex polyhedral cone containing $\mathbf w_1, \ldots, \mathbf w_N$ such that 
$(1-q)\gamma\in \mathbb Z^n\cap \delta$, and suppose 
$K'$ contains all $(q-1)$-th root of unity. Let $\mathbf a=(a_1,\ldots, a_n)$ be a Techm\"uller point, that is, $a_j^q=a_j$. Then each 
$G_{\mathbf a}: C^k( L^\dagger_0)\to C^k( L^\dagger_0)$ is nuclear. Moreover, we have 
\begin{eqnarray*}
S_m(F(\bar{\mathbf a},\mathbf t))&=&\mathrm{Tr}\big( (q^nG_{\mathbf a})^m, C^\cdot( L^\dagger_0)\big),\\
L(F(\bar{\mathbf a},\mathbf t),T)&=&\mathrm{det}\big(I-T q^nG_{\mathbf a}, C^\cdot( L^\dagger_0)\big)^{-1},
\end{eqnarray*}
\end{theorem}

The paper is organized as follows. In section 1, we prove Propositions \ref{estimation} and \ref{FG} on basic properties of the GKZ hypergeometric complex. In Section 2, 
we prove Proposition \ref{coherent} about the $D^\dagger$-module structure on $C^\cdot(L^\dagger)$. 
In Section 3, we prove the Dwork trace formula and Theorem \ref{arithmetic}
relating the twisted exponential sum with the trace of the Frobenius  on $C^\cdot(L^\dagger_0)$. Section 4 is the most technical part of the paper. 
By relating $C^\cdot(L^\dagger_0)$ with the rigid cohomology of a crystal, we prove Proposition \ref{newfiber} saying that the cohomology groups of 
$C^\cdot(L^\dagger_0)$ are finite dimensional.  
Combining this finiteness result with a theorem of Caro, we deduce Proposition \ref{overholonomic} which says $C^\cdot(L^\dagger)$ is 
over-holonomic. Finally we use Berthelot's theory of arithmetic $\mathcal D$-modules and results of Caro and Ogus 
to prove Theorems \ref{convergent} which claims that $H^n(C^\cdot(L^\dagger))$ defines an overconvergent $F$-isocrystal on the non-degenerate locus $U$. This proof is
inspired
by \cite{Miy} which treats the one-variable hypergeometric systems. In Section 5, we state some open questions. 

\section{Proof of some basic propositions}  

\subsection{Proof of Proposition \ref{estimation}} 
(i) Write $\exp(\pi z-\pi z^q)=1+ \sum_{i=1}^\infty c_i z^i$. We have $\vert c_i \vert\leq p^{-\frac{p-1}{pq}i}$ by \cite[Theorem 4.1]{M1}.
Write 
\begin{eqnarray*}
\exp(\pi z^q-\pi z)&=&1-(\sum_{i=1}^\infty c_i z^i)+(\sum_{i=1}^\infty c_i z^i)^2-\cdots\\
&=& \sum_{i=0}^\infty c'_i z^i.
\end{eqnarray*}
Then we also have $\vert c'_i \vert\leq p^{-\frac{p-1}{pq}i}$.
For the monomial $x_j\mathbf t^{\mathbf w_j}$, we have 
\begin{eqnarray*}
&&\exp\big(\pi (x_j\mathbf t^{\mathbf w_j})^q-\pi x_j\mathbf t^{\mathbf w_j}\big)=\sum_{i=0}^\infty c'_i x_j^i\mathbf t^{i\mathbf w_j},\\
&&\Vert c'_i x_j^i\Vert_r \leq p^{-\frac{p-1}{pq}i} r^i 
= \Big(r^{-1} p^{\frac{p-1}{pq}}\Big)^{-i} \leq  \Big(r^{-\frac{1}{w}} p^{\frac{p-1}{pqw}}\Big)^{-\vert i\mathbf w_j\vert}.
\end{eqnarray*}
Here for the last inequality, we use the fact that $\frac{\vert i\mathbf w_j\vert}{w} \leq i $ and the assumption that $r\leq p^{\frac{p-1}{pq}}$.
So we have $\exp\big(\pi (x_j\mathbf t^{\mathbf w_j})^q-\pi x_j\mathbf t^{\mathbf w_j}\big)\in L(r,r^{-\frac{1}{w}} p^{\frac{p-1}{pqw}})$. 
We have $$\exp\big(\pi F(\mathbf x^q,\mathbf t^q)- \pi F(\mathbf x,\mathbf t)\big)=\prod_{j=1}^N 
\exp\big(\pi (x_j\mathbf t^{\mathbf w_j})^q-\pi x_j\mathbf t^{\mathbf w_j}\big).$$ 
Since $r^{-\frac{1}{w}} p^{\frac{p-1}{pqw}} \geq 1,$ the space
$L(r,r^{-\frac{1}{w}} p^{\frac{p-1}{pqw}})$ is a ring. So 
$\exp\big (\pi F(\mathbf x^q,\mathbf t^q)-
\pi F(\mathbf x, \mathbf t)\big)$ lies  in $L(r,r^{-\frac{1}{w}} p^{\frac{p-1}{pqw}})$.
Similarly
$\exp\big(\pi F(\mathbf x,\mathbf t)- \pi F(\mathbf x^q,\mathbf t^{q})\big)$ 
lies in $L(r,r^{-\frac{1}{w}} p^{\frac{p-1}{pqw}})$. 

(ii)  Let $C^{(1)\cdot}( L^\dagger)$ be the logarithmic twisted de Rham complex 
so that $C^{(1) ,k}( L^\dagger)=C^k( L^\dagger)$ for each $k$, and 
$d^{(1)}:C^{(1),k}( L^\dagger)\to C^{(1),k+1}( L^\dagger)$ is given by 
\begin{eqnarray*}
d^{(1)}&=&\big ( \mathbf t^\gamma\exp(\pi F(\mathbf x^q,\mathbf t))\big)^{-1} \circ \mathrm d_{\mathbf t} \circ \big(\mathbf t^\gamma
\exp(\pi F(\mathbf x^q,\mathbf t)\big)\big)\\
&=&\mathrm d_{\mathbf t}+ \sum_{i=1}^n \Big(\gamma_i+\sum_{j=1}^N \pi w_{ij}x_j^q\mathbf t^{\mathbf w_j}\Big)
\frac{\mathrm dt_i}{t_i}.
\end{eqnarray*}
Let $\nabla^{(1)}$ be the connection on $C^{(1)\cdot}( L^\dagger)$ defined by 
\begin{eqnarray*}
\nabla^{(1)}(\omega)&=&\big(\mathbf t^\gamma \exp (\pi F(\mathbf x^{q}, \mathbf t))\big)^{-1}
\circ \mathrm d_{\mathbf x} \circ \big(\mathbf t^\gamma\exp (\pi F(\mathbf x^{q}, \mathbf t))\big)(\omega)\\
&=&
\mathrm d_{\mathbf x}\omega+\sum_{j=1}^Nq\pi x_j^{q-1}\mathbf t^{\mathbf w_j} \mathrm dx_j \wedge \omega,
\end{eqnarray*}
We first prove $$\mathrm{Fr}^*(C^\cdot( L^{\dagger}),\nabla)\cong (C^{(1)\cdot}( L^\dagger), \nabla^{(1)}).$$
Consider the $K$-algebra homomorphism 
$$K\langle y_1, \ldots, y_N\rangle^\dagger \to K\langle x_1, \ldots, x_N\rangle^\dagger,\quad y_j\mapsto x_j^{q}.$$  
This makes $K\langle\mathbf x\rangle^\dagger$ a finite $K\langle\mathbf y\rangle^\dagger$-algebra. 
We have a canonical isomorphism 
\begin{eqnarray*}
 \tilde L^\dagger  \otimes_{K\langle\mathbf y\rangle^\dagger}K\langle\mathbf x\rangle^\dagger
\stackrel \cong\to   L^{\dagger},
\end{eqnarray*}
where $\tilde L^{\dagger}$ is defined in the same way as $L^\dagger$ except that we change the variables from $x_j$ to $y_j$. 
The connection $\nabla$ on $ \tilde L^\dagger$ defines a connection on  
 $ \tilde L^\dagger\otimes_{K\langle\mathbf y\rangle^\dagger}K\langle\mathbf x\rangle^\dagger$ 
 by the Leibniz rule. Via the above isomorphism, it defines the connection $\mathrm{Fr}^*\nabla$ on  
 $L^{\dagger}$. Let's verify that 
 it coincides with the connection $\nabla^{(1)}$ on $ L^\dagger$. Any element in 
 $ L^{\dagger}$ can be written as a finite sum
 of elements of the form
 $f(\mathbf x)g(\mathbf x^q, \mathbf t)$ with $f(\mathbf x)\in K [\mathbf x]$ and $g(\mathbf y,\mathbf t)\in\tilde L^\dagger$. By 
 the Leibniz rule, we have 
 \begin{eqnarray*}
 (\mathrm{Fr}^*\nabla)(f(\mathbf x)g(\mathbf x^q,\mathbf t))
 &=&\mathrm d_{\mathbf x}( f(\mathbf x)) g(\mathbf x^q,\mathbf t)+ f(\mathbf x)(\nabla g(\mathbf y,\mathbf t))|_{\mathbf y=\mathbf x^q}\\
 &=&\mathrm d_{\mathbf x} (f(\mathbf x)) g(\mathbf x^q,\mathbf t)+ f(\mathbf x)\Big (\mathrm d_{\mathbf y}g(\mathbf y,\mathbf t)
 +\pi g(\mathbf y,\mathbf t) d_{\mathbf y} F(\mathbf y, \mathbf t)\Big)|_{\mathbf y=\mathbf x^q}\\
 &=& \mathrm d_{\mathbf x}\Big( f(\mathbf x) g(\mathbf x^q,\mathbf t)\Big) +\pi f(\mathbf x)g(\mathbf x^q,\mathbf t) d_{\mathbf x} F(\mathbf x^q, \mathbf t).
 \end{eqnarray*}
 This proves our assertion. Similarly, one verifies that the connection $\mathrm{Fr}^*\nabla$ on  
 $\mathrm{Fr}^*C^k( \tilde L^{\dagger})$ can be identified with the connection $\nabla^{(1)}$ on 
 $C^{(1),k}( L^{\dagger})$.  
Using the fact that $\Phi \circ \mathrm d_{\mathbf t}=\mathrm d_{\mathbf t} \circ \Phi$ and 
$\Phi \circ 
\mathrm d_{\mathbf x}=\mathrm d_{\mathbf x}\circ \Phi$, one checks that 
$F \circ d^{(1)}=d\circ F$ and $F\circ\nabla^{(1)}=\nabla\circ F.$ 
So $F$ defines a horizontal morphism of complexes of $K\langle \mathbf x\rangle^\dagger$-modules with connections
$$F: (C^{(1)\cdot}( L^\dagger), \nabla^{(1)})\to (C^\cdot( L^\dagger), \nabla).$$
 
\subsection{Proof of Proposition \ref{FG}}

That $G:(C^\cdot(L^\dagger), \nabla)\to (C^{(1)\cdot}(L^\dagger), \nabla^{(1)})$ is a horizontal morphism 
can be proved in the same way as the proof of Proposition \ref{estimation} (ii). To prove the 
assertions in \ref{FG} (ii),
we first work with the usual logarithmic de Rham complex and later with the twisted logarithmic de Rham complex.
We have
$$\Psi\circ\Phi=\mathrm{id}$$ on $C^\cdot(L^\dagger)$. 
By enlarging $K$, we may assume $K$ contains all $q$-th roots of unity. 
Let $\mu_q$ be the group of $q$-th roots of unity in $K$. For any $\zeta=(\zeta_1, \ldots, \zeta_n)\in \mu_q^n$, we use the notation
$$\zeta\mathbf t=(\zeta_1t_1, \ldots, \zeta_nt_n),\quad  \zeta^{\mathbf w}=\zeta_1^{w_1}\cdots\zeta_n^{w_n}.$$  
We have 
$$\sum_{\zeta\in \mu_q^n} \zeta^{\mathbf w}=\left\{\begin{array}{cl}
q^n &\hbox{if } q|\mathbf w,\\
0&\hbox{otherwise}. 
\end{array}
\right.$$
So we have
\begin{eqnarray*}
\Phi\circ \Psi(\sum_{\mathbf w}a_{\mathbf w}(\mathbf x)\mathbf t^{\mathbf w})=\sum_{\mathbf w}
a_{q\mathbf w}(\mathbf x)\mathbf t^{q\mathbf w}=\frac{1}{q^n} \sum_{\zeta\in \mu_q^n} \sum_{\mathbf w}
a_{\mathbf w}(\mathbf x)(\zeta \mathbf t)^{\mathbf w}.
\end{eqnarray*}
Let $\Theta_\zeta$ be the endomorphism on differential forms induced by the substitution $\mathbf t\mapsto \zeta\mathbf t$: 
$$\Theta_\zeta\Big(\sum_{1\leq i_1<\cdots < i_k\leq n}f_{i_1\ldots i_k}(\mathbf x,\mathbf t)\frac{\mathrm dt_{i_1}}{t_{i_1}}\wedge \cdots\wedge \frac{\mathrm dt_{i_k}}
{t_{i_k}}\Big)=
\sum_{1\leq i_1<\cdots < i_k\leq n}f_{i_1\ldots i_k}(\mathbf x, \zeta\mathbf t)\frac{\mathrm dt_{i_1}}{t_{i_1}}\wedge \cdots\wedge \frac{\mathrm dt_{i_k}}
{t_{i_k}}.$$ It commutes with $\mathrm d_{\mathbf t}$. We have 
$$\Phi\circ \Psi=\frac{1}{q^n}\sum_{\zeta\in \mu_q^n}  \Theta_\zeta.$$ 
Let's show that $\Phi \circ \Psi$ is homotopic to $\mathrm{id}$. It suffices to that $\Theta_\zeta$ is homotopic to 
$\mathrm{id}$ for each $\zeta\in \mu_q^n$. 
Let $$L_T^\dagger
=\bigcup_{r>1,\;s>1}\{\sum_{k\in \mathbb Z_{\geq 0},\; 
\mathbf w\in\mathbb Z^n\cap\delta} a_{(k,\mathbf w)}(\mathbf x) T^k \mathbf t^{\mathbf w}:\;  a_{(k,\mathbf w)}(\mathbf x) \in K
\{r^{-1}\mathbf x\},  
\; \Vert a_{(k,\mathbf w)}(\mathbf x)\Vert_r s^{k+ \vert \mathbf w\vert} \hbox { are bounded}\}.$$
Consider the de Rham complex
$(C^\cdot(L_T^\dagger), \mathrm d_{(T,\mathbf t)})$ so that $C^k(L_T^\dagger)$ is the space of $k$-forms 
which can be written as sums of products of $\mathrm dT, \frac{\mathrm dt_1}{t_1}, \ldots, \frac{\mathrm dt_n}{t_n}$ and functions in $L_T^\dagger$, and
$\mathrm d_{(T,\mathbf t)}:C^k(L_T^\dagger)\to C^{k+1}(L_T^\dagger)$ is the usual exterior diffferentiation of differential forms in the variables $T,\mathbf t$. 
The substitution 
$$t_i\to (1+(\zeta_i-1)T) t_i\quad (i=1, \ldots, n)$$
induces a chain map $$\iota: (C^\cdot (L^\dagger),\mathrm d_t) \to (C^\cdot(L_T^\dagger),\mathrm d_{(T,\mathbf t)}).$$ 
Here we use the fact that $\zeta_i\equiv 1\mod p$ so that  each $1+(\zeta_i-1)T$  is a unit in $L_T^\dagger$. 
In particular, $\frac{\mathrm d\big((1+(\zeta_i-1)T)t_i\big)}{(1+(\zeta_i-1)T)t_i}$ lies in $C^\cdot(L_T^\dagger)$. 
The evaluation at $T=0$ (resp. $T=1$)
induces a chain map 
$$\mathrm{ev}_0: (C^\cdot(L_T^\dagger), \mathrm d) \to (C^\cdot(L^\dagger),\mathrm d_t) \quad 
(\hbox{resp. }\mathrm{ev}_1: (C^\cdot(L_T^\dagger), \mathrm d) \to (C^\cdot(L^\dagger),\mathrm d_t)).
$$ We have 
$$\mathrm{ev}_1\circ \iota=\Theta_\zeta,\quad \mathrm{ev}_0\circ \iota=\mathrm{id}.$$ 
To prove  $\Theta_\zeta$ is homotopic to identity, it suffices to show $\mathrm{ev}_1$ is homotopic to $\mathrm{ev}_0$. 
Note that $\int_0^1 g(\mathbf x,T,\mathbf t)\mathrm dT$ lies in $L^\dagger$ for any $g(\mathbf x, T,\mathbf t)\in L^\dagger_T$. 
Define $\Xi: C^k(L^\dagger_T)\to C^{k-1}(L^\dagger)$ by
\begin{eqnarray*}
&& \Xi \Big(f(\mathbf x, T,\mathbf t)  \frac{\mathrm dt_{i_1}}{t_{i_1}}\wedge \cdots\wedge \frac{\mathrm dt_{i_k}}{t_{i_k}}\Big)=0,\\
&& \Xi \Big(g(\mathbf x,T,\mathbf t) \mathrm dT\wedge \frac{\mathrm dt_{j_1}}{t_{j_1}}\wedge \cdots\wedge \frac{\mathrm dt_{j_{k-1}}}{t_{j_{k-1}}}\Big)
=\Big(\int_0^1 g(\mathbf x,T,\mathbf t) \mathrm dT\Big) \frac{\mathrm dt_{j_1}}{t_{j_1}}\wedge \cdots\wedge \frac{\mathrm dt_{j_{k-1}}}{t_{j_{k-1}}}.
\end{eqnarray*}
Then we have $$\mathrm d_{\mathbf t}\Xi +\Xi \mathrm d_{(T, \mathbf t)}=\mathrm{ev}_1-\mathrm{ev}_0.$$

We now consider the logarithmic twisted de Rham complexes. Let $$T_{\zeta}, L, E_0, E_1, H$$ be the conjugates of 
$$\Theta_\zeta, \iota, \mathrm{ev}_0, \mathrm{ev}_1, \Xi$$ by $\mathbf t^\gamma \exp(\pi F(\mathbf x,\mathbf t))$,
respectively. One verifies that they are defined on 
$C^\cdot(L^\dagger)$ or $C^\cdot(L^\dagger_T)$. Actually we have 
\begin{eqnarray*}
&&E_0=\mathrm{ev}_0, \quad E_1=\mathrm{ev}_1,\quad H=\Xi,\\
&& T_\zeta=\zeta^\gamma \exp\big(\pi F(\mathbf x, \zeta\mathbf t)-\pi F(\mathbf x,\mathbf t)\big) \Theta_\zeta,\\
&&L=\big(1+(\zeta-1)T\big)^\gamma\exp \Big(\pi F\big(\mathbf x, (1+(\zeta-1)T)\mathbf t\big)-
\pi F(\mathbf x,\mathbf t)\Big)\iota,
\end{eqnarray*}
where for any $q$-th root of unity $\zeta_i$ and any $\gamma_i$ with $(1-q)\gamma_i\in\mathbb Z$, we define 
$\zeta_i^{\gamma_i}$ to be the unique $q$-th root of unity such that 
$$(\zeta_i^{\gamma_i})^{1-q}=\zeta_i^{(1-q)\gamma_i}.$$
This makes sense since taking $(1-q)$-th power defines an isomorphism on the group of $q$-th roots of unity. But taking $(1-q)$-th 
power on the group of $q$-th roots of unity is the identity map. So we actually have $\zeta_i^{\gamma_i}=\zeta_i^{(1-q)\gamma_i}$.  
Define $\big(1+(\zeta_i-1)T\big)^{\gamma_i}$ by 
$$\big(1+(\zeta_i-1)T\big)^{\gamma_i}=1+\gamma_i  (\zeta_i-1)T+ \frac{\gamma_i(\gamma_i-1)}{2!}\big((\zeta_i-1)T\big)^2+\cdots.$$
Then we have
$$
\big(1+(\zeta_i-1)T\big)^{\gamma_i}|_{T=1}=\zeta_i^{\gamma_i}.$$
One can verifty that
\begin{eqnarray*}
&&GF=\mathrm{id},\quad FG=\frac{1}{q^n}\sum_{\zeta\in \mu_q^n} T_\zeta,\\
&& E_1\circ L=T_\zeta,\quad E_0\circ L=\mathrm{id},\quad dH+Hd=E_1-E_1.
\end{eqnarray*}
It follows that each $T_\zeta$ is homotopic to identity and hence $FG$ is also homotopic to identity.

\section{$D^\dagger$-modules}

Recall that 
$$D^\dagger=\Gamma(\mathbb P^N, \mathcal D^\dagger_{\mathscr P^N, \mathbb Q}(\infty))=\bigcup_{r>1,\;s>1}
\{\sum_{\mathbf v\in\mathbb Z_{\geq 0}^N} f_{\mathbf v}(\mathbf x)\frac{\partial^{\mathbf v}}{\mathbf v_!}:\;
f_{\mathbf v}(\mathbf x)\in K\{ r^{-1}\mathbf x\},\; 
\Vert f_{\mathbf v}(\mathbf x)\Vert_r s^{\vert \mathbf v\vert} \hbox { are bounded}\}.$$ 
In this section, we mainly use the following description for $D^\dagger$. 

\begin{proposition}\label{two_defn_D}
We have
$$D^\dagger=\bigcup_{r>1,\;s>1}
\{\sum_{\mathbf v\in\mathbb Z_{\geq 0}^N} f_{\mathbf v}(\mathbf x)\frac{\partial^{\mathbf v}}{\pi^{\vert \mathbf v\vert}}:\;
f_{\mathbf v}(\mathbf x)\in K\{ r^{-1}\mathbf x\},\; 
\Vert f_{\mathbf v}(\mathbf x)\Vert_r s^{\vert \mathbf v\vert} \hbox { are bounded}\}.$$ 
\end{proposition}

\begin{proof} Let $m$ be a positive integer and let $$m=a_0+a_1p+a_2p^2+\cdots+ a_lp^l$$ be its $p$-expansion, where $0\leq a_i\leq p-1$ for all $i$
and $a_l\not=0$.  
Set $\sigma(m)=a_0+a_1+a_2+\cdots+a_l.$ We have
$$\sigma(m)\leq (p-1)(l+1)\leq (p-1)(\log_p m+1).$$ 
It is a standard fact that  
$$\mathrm{ord}_p(m!)=\frac{m-\sigma(m)}{p-1}.$$
So we have 
$$\mathrm{ord}_p\Big(\frac{\pi^m}{m!}\Big)=\frac{\sigma(m)}{p-1},\quad 0\leq \mathrm{ord}_p\Big(\frac{\pi^m}{m!}\Big)
\leq \log_p m+1.$$
For any nonzero $\mathbf v=(v_1, \ldots, v_N)\in\mathbb Z_{\geq 0}^N$, let $v_{i_1},\ldots, v_{i_m}$ be those nonzero components. We have 
\begin{eqnarray*}
0\leq\mathrm{ord}_p\Big(\frac{\pi^{\vert \mathbf v\vert}}{\mathbf v!}\Big)&\leq& \sum_{j=1}^m (\log_p v_{i_j}+1)\\
&=& \log_p\Big(\prod_{j=1}^m v_{i_j}\Big) +m\\
&\leq&\log_p \left(\frac{ \sum_{j=1}^m v_{i_j}}{m}\right)^m+m\\
&\leq& N\log_p\vert \mathbf v\vert +N,
\end{eqnarray*}
that is, $$0\leq\mathrm{ord}_p\Big(\frac{\pi^{\vert \mathbf v\vert}}{\mathbf v!}\Big)\leq N\log_p\vert \mathbf v\vert +N.$$
Set 
$$B^\dagger=\bigcup_{r>1,\;s>1}
\{\sum_{\mathbf v\in\mathbb Z_{\geq 0}^N} f_{\mathbf v}(\mathbf x)\frac{\partial^{\mathbf v}}{\mathbf v_!}:\;
f_{\mathbf v}(\mathbf x)\in K\{ r^{-1}\mathbf x\},\; 
\Vert f_{\mathbf v}(\mathbf x)\Vert_r s^{\vert \mathbf v\vert} \hbox { are bounded}\}.$$ 
Let's prove $B^\dagger=D^\dagger$. Given 
$\sum_{\mathbf v\in\mathbb Z_{\geq 0}^N} f_{\mathbf v}(\mathbf x)\frac{\partial^{\mathbf v}}{\mathbf v_!}$ 
in $B^\dagger$, choose real numbers $r>1, s>1$ and $C>0$ such that  
$$\Vert f_{\mathbf v}(\mathbf x)\Vert_r s^{\vert \mathbf v\vert}\leq C.$$
We have $$\sum_{\mathbf v\in\mathbb Z_{\geq 0}^N} f_{\mathbf v}(\mathbf x)\frac{\partial^{\mathbf v}}{\mathbf v_!}
=\sum_{\mathbf v\in\mathbb Z_{\geq 0}^N} \Big(f_{\mathbf v}(\mathbf x)\frac{\pi^{\vert\mathbf v\vert}}{\mathbf v!}\Big)\frac{\partial^{\mathbf v}}{\pi^{\vert \mathbf v\vert}}.$$
For any nonzero $\mathbf v \in \mathbb Z_{\geq 0}^N$, we have 
$\mathrm{ord}_p \Big(\frac{\pi^{\vert \mathbf v\vert}}{\mathbf v!}\Big)\geq 0.$ So we have
$$\Big\Vert f_{\mathbf v}(\mathbf x)\frac{\pi^{\vert\mathbf v\vert}}{\mathbf v!}\Big\Vert_r s^{\vert \mathbf v\vert}\leq 
\Vert f_{\mathbf v}(\mathbf x)\Vert_r s^{\vert \mathbf v\vert}\leq C.$$
Hence $\sum_{\mathbf v\in\mathbb Z_{\geq 0}^N} f_{\mathbf v}(\mathbf x)\frac{\partial^{\mathbf v}}{\mathbf v_!}$ lies in 
$D^\dagger$. 

Conversely, given 
$\sum_{\mathbf v\in\mathbb Z_{\geq 0}^N} f_{\mathbf v}(\mathbf x)\frac{\partial^{\mathbf v}}{\pi^{\vert \mathbf v\vert}}$ 
in $D^\dagger$, choose real numbers $r>1, s>1$ and $C'>0$ such that  
$$\Vert f_{\mathbf v}(\mathbf x)\Vert_r s^{\vert \mathbf v\vert}\leq C'.$$
We have $$\sum_{\mathbf v\in\mathbb Z_{\geq 0}^N} f_{\mathbf v}(\mathbf x)\frac{\partial^{\mathbf v}}{\pi^{\vert \mathbf v\vert}}
=\sum_{\mathbf v\in\mathbb Z_{\geq 0}^N} \Big(f_{\mathbf v}(\mathbf x)\frac{\mathbf v!}{\pi^{\vert \mathbf v\vert}}\Big)\frac{\partial^{\mathbf v}}{\mathbf v_!}.$$
For any nonzero $\mathbf v=(v_1, \ldots, v_N)\in \mathbb Z_{\geq 0}^N$, we have 
$\mathrm{ord}_p \Big(\frac{\pi^{\vert \mathbf v\vert}}{\mathbf v!}\Big)\leq N\log_p\vert\mathbf v\vert+N$. So we have
$$\Big \vert \frac{\mathbf v!}{\pi^{\vert\mathbf v\vert}}\Big\vert \leq p^N \vert \mathbf v \vert^N$$
For any $1<s'<s$,  we have  
\begin{eqnarray*}
\Big\Vert f_{\mathbf v}(\mathbf x)\frac{\mathbf v!}{\pi^{\vert\mathbf v\vert}}\Big\Vert_r s'^{\vert \mathbf v\vert}&\leq& 
\Big(\Vert f_{\mathbf v}(\mathbf x)\Vert_r s^{\vert\mathbf v\vert} \Big)\Big(p^N \vert \mathbf v\vert^N\Big) (s'/s)^{\vert \mathbf v\vert}\\
&\leq& p^N C' \vert \mathbf v\vert^N(s'/s)^{\vert \mathbf v\vert}.
\end{eqnarray*}
Note that $\vert \mathbf v\vert^N(s'/s)^{\vert \mathbf v\vert}$ are bounded since $0<s'/s<1$. So 
 $\sum_{\mathbf v\in\mathbb Z_{\geq 0}^N} f_{\mathbf v}(\mathbf x)\frac{\partial^{\mathbf v}}{\mathbf v_!}$ lies in 
$B^\dagger$. 
\end{proof}

\begin{lemma}\label{comb} Let $S$ be any subset of $\mathbb Z^n_{\geq 0}$. There exists a finite subset $S_0$ of $S$ such that 
$S\subset \bigcup_{\mathbf v\in S_0}(\mathbf v+\mathbb Z_{\geq 0}^n)$. 
\end{lemma} 
 
\begin{proof} We use induction on $n$. When $n=1$, we have $S\subset v+\mathbb Z_{\geq 0}$, where $v$ is the minimal element in $S\subset \mathbb 
Z_{\geq 0}$. Suppose the assertion holds for any subset of $\mathbb Z^n_{\geq 0}$, and let $S$ be a subset of $\mathbb Z^{n+1}_{\geq 0}$. If $S$ is empty, 
our assertion holds trivially. Otherwise, we fix an element 
$\mathbf a=(a_1, \ldots, a_{n+1})$ in $S$. For any $i\in\{1,\ldots, n+1\}$ and any $0\leq b_i\leq a_i$, let 
$$S_{i, b_i}=\{(c_1, \ldots, c_{n+1})\in S:\;  c_i=b_i\}.$$ By 
the induction hypothesis, there exists a finite subset $T_{i, b_i}\subset S_{i, b_i}$ such that 
$$S_{i,b_i}\subset \bigcup_{\mathbf v\in T_{i,b_i}} (\mathbf v+\mathbb Z_{\geq 0}^{n+1}).$$ We have 
\begin{eqnarray*}
S&\subset& \Big(\bigcup_{1\leq i\leq n+1} \bigcup_{0\leq b_i\leq a_i} S_{i,b_i}\Big)\bigcup (\mathbf a+\mathbb Z^{n+1}_{\geq 0})\\
&\subset&  \Big(\bigcup_{1\leq i\leq n+1} \bigcup_{0\leq b_i\leq a_i} \bigcup_{\mathbf v\in T_{i,b_i}} (\mathbf v+\mathbb Z_{\geq 0}^{n+1})\Big)
\bigcup (\mathbf a+\mathbb Z^{n+1}_{\geq 0}).
\end{eqnarray*}
We can take $S_0=\Big(\bigcup_{1\leq i\leq n+1} \bigcup_{0\leq b_i\leq a_i} T_{i, b_i}\Big)\bigcup \{\mathbf a\}.$
\end{proof}

Let 
\begin{eqnarray*}
&&C(A)=\{k_1\mathbf w_1+\cdots +k_N \mathbf w_N:\; k_i\in\mathbb Z_{\geq 0}\},\\
&&L^{\dagger\prime}=\bigcup_{r>1,\;s>1}
\{\sum_{\mathbf w\in C(A)} a_{\mathbf w}(\mathbf x)\mathbf t^{\mathbf w}:\; a_{\mathbf w}(\mathbf x)\in K\{r^{-1}\mathbf x\}, \; 
\Vert a_{\mathbf w}(\mathbf x)\Vert_r s^{\vert \mathbf w\vert} \hbox{ are bounded}\}.
\end{eqnarray*}
$C(A)$ is a submonoid of $\mathbb Z^n\cap \delta$, and $L^{\dagger\prime}$ is a subring of $L^\dagger$. 
 
\begin{lemma}\label{LL} ${}$

(i) The ring homomorphism $$\phi: K\langle \mathbf x, \mathbf y\rangle^\dagger\to L^{\dagger\prime}, \quad  
\sum_{\mathbf v\in \mathbb Z^N_{\geq 0}}f_{\mathbf v}(\mathbf x) 
\mathbf y^{\mathbf v}
\mapsto\sum_{\mathbf v\in \mathbb Z^N_{\geq 0}}f_{\mathbf v}(\mathbf x)  \mathbf t^{v_1\mathbf w_1+\cdots +v_N\mathbf w_N}$$ is surjective,
where $\mathbf y=(y_1, \ldots, y_N)$ and 
$$K\langle \mathbf x, \mathbf y\rangle^\dagger=\bigcup_{r>1,\;s>1}\{ \sum_{\mathbf v\in \mathbb Z^N_{\geq 0}} f_{\mathbf v}(\mathbf x){\mathbf y}^{\mathbf v}:\;
f_{\mathbf v}(\mathbf x)\in K\{r^{-1}x\} \hbox { and }
\Vert f_{\mathbf v}(x)\Vert_r s^{\vert \mathbf v\vert} \hbox { are bounded}\}.$$ 

(ii) Let $\frac{\partial}{\partial x_j}$ acts on $L^{\dagger\prime}$ via 
$$\nabla_{\frac{\partial}{\partial x_j}}=\big(\exp(\pi F(\mathbf x, \mathbf t))\big)^{-1}\circ \frac{\partial}{\partial x_j}\circ \big(\exp(\pi F(\mathbf x, \mathbf t))\big)
=\frac{\partial}{\partial x_j}+\pi \mathbf t^{\mathbf w_j}.$$
Then $L^{\dagger\prime}$ is a $D^\dagger$-module, and the map 
$$\varphi: D^\dagger\to L^{\dagger\prime}, \quad \sum_{\mathbf v\in \mathbb Z^N_{\geq 0}}f_{\mathbf v}(\mathbf x) 
\frac{\partial^{\mathbf v}}{\pi^{\vert\mathbf v\vert}}
\mapsto(\sum_{\mathbf v\in \mathbb Z^N_{\geq 0}}f_{\mathbf v}(\mathbf x) \frac{\partial^{\mathbf v}}{\pi^{\vert\mathbf v\vert}}
)\cdot 1=\sum_{\mathbf v\in \mathbb Z^N_{\geq 0}}f_{\mathbf v}(\mathbf x)   \mathbf t^{v_1\mathbf w_1+\cdots +v_N\mathbf w_N}$$ 
is an epimorphism of left $D^\dagger$-modules.
\end{lemma}

\begin{proof} ${}$

(i) Let $\Delta$ be the convex hull of $\{0, \mathbf w_1, \ldots, \mathbf w_N\}$. 
Decompose $\Delta$ into a finite union $\bigcup_\tau \tau$
so that each $\tau$ is a  
simplex of dimension $n$ with vertices $\{0,\mathbf w_{i_1},\ldots, 
\mathbf w_{i_n}\}$ for some subset $\{i_1,\ldots, i_n\}\subset \{1,\ldots, N\}$. For each $\tau$, let $\delta(\tau)$ be the cone 
generated by $\tau$, and let 
\begin{eqnarray*}
B(\tau)&=&\mathbb Z^n\cap \{c_1\mathbf w_{i_1}+\cdots +c_n\mathbf w_{i_n}:\; 0\leq c_i< 1\},\\
C(\tau)&=& \{k_1\mathbf w_{i_1}+\cdots +k_n\mathbf w_{i_n}:\; k_i\in\mathbb Z_{\geq 0}\}.
\end{eqnarray*}
Being a discrete bounded set, $B(\tau)$ is finite. Every element $\mathbf w\in \mathbb Z^n\cap \delta(\tau)$ can be written uniquely as 
$$\mathbf w=b(\mathbf w)+c(\mathbf w)$$ with $b(\mathbf w)\in B(\tau)$ and $c(\mathbf w)\in C(\tau)$.
So we have $\mathbb Z^n\cap\delta(\tau)=\bigcup_{\mathbf w\in B(\tau)}(\mathbf w+C(\tau))$, and hence
$$C(A)=\bigcup_\tau(C(A)\cap \delta(\tau))=\bigcup_\tau \bigcup_{\mathbf w\in B(\tau)}\big (C(A)\cap (\mathbf w+C(\tau))\big ).$$
For each $C(A)\cap (\mathbf w+C(\tau))$, the map 
$$\mathbb Z^n_{\geq 0}\to \mathbf w+C(\tau),\quad (k_1, \ldots, k_n)\mapsto \mathbf w+k_1\mathbf w_{i_1}+\cdots +k_n\mathbf w_{i_n}$$
is a bijection. Applying Lemma \ref{comb} to the inverse image of $C(A)\cap (\mathbf w+C(\tau))$, we can find finitely many 
$\mathbf u_1, \ldots, \mathbf u_m\in C(A)\cap (\mathbf w+C(\tau))$ such that 
$$C(A)\cap (\mathbf w+C(\tau))= \bigcup_{i=1}^m (\mathbf u_i+C(\tau)). $$
We thus decompose $C(A)$ into a finite union of subsets of the form $\mathbf u+C(\tau)$ such that
$\tau$ is a simplex 
of dimension $n$ with vertices $\{0,\mathbf w_{i_1},\ldots, 
\mathbf w_{i_n}\}$ for some subset $\{i_1,\ldots, i_n\}\subset \{1,\ldots, N\}$, and 
$\mathbf u\in 
C(A)\cap (\mathbf w+C(\tau))$ for some $\mathbf w\in B(\tau)$. Elements in $L^{\dagger\prime}$ is a sum of elements 
of the form $\sum_{\mathbf w\in \mathbf u+C(\tau)} a_{\mathbf w}(\mathbf x)\mathbf t^{\mathbf w}$, where 
$a_{\mathbf w}(\mathbf x)\in K\{r^{-1}\mathbf x\}$ and  
$\Vert a_{\mathbf w}(\mathbf x)\Vert_r  s^{\vert \mathbf w\vert}$ are bounded for some $r,s>1$. 
To prove
$\phi: K\langle \mathbf x, \mathbf y\rangle ^\dagger\to L^{\dagger\prime}$ 
is surjective, it suffices to show $\sum_{\mathbf w\in \mathbf u+C(\tau)} a_{\mathbf w}(\mathbf x)\mathbf t^{\mathbf w}$
lies in the image of $\phi$.
Write $\mathbf u=c_1\mathbf w_1+\cdots+c_N\mathbf w_N$, where $c_i\in\mathbb Z_{\geq 0}$. A preimage for 
$\sum_{\mathbf w\in \mathbf u+C(\tau)} a_{\mathbf w}(\mathbf x)\mathbf t^{\mathbf w}$ is
\begin{eqnarray*}
\sum_{v_1, \ldots, v_n\geq 0} a_{\mathbf u+v_1 \mathbf w_{i_1}+\cdots +v_n\mathbf w_{i_n}} (\mathbf x)
y_{i_1}^{c_{i_1}+v_{1}}\cdots y_{i_n}^{c_{i_n}+v_{n}}\prod_{j\in\{1, \ldots,N\}-\{i_1,\ldots,i_n\}}
y_j^{c_j}.
\end{eqnarray*}
Here to verify this element lies in $K\langle \mathbf x, 
\mathbf y\rangle^\dagger$, we use the fact that 
$$\vert \mathbf v\vert \leq C \vert v_1\mathbf w_{i_1}+\cdots+v_n\mathbf w_{i_n}\vert$$ for some constant $C$
since $\mathbf w_{i_1},\ldots, \mathbf w_{i_n}$ form a basis of $\mathbb R^n$. So we have 
\begin{eqnarray*}
&& \Vert a_{\mathbf u+v_1 \mathbf w_{i_1}+\cdots +v_n\mathbf w_{i_n}} (\mathbf x)\Vert_r (\sqrt[C]{s})^{c_1+\cdots+c_N+v_1+\cdots+v_n}\\
&\leq&(\sqrt[C]{s})^{c_1+\cdots+c_N} 
\Vert a_{\mathbf u+v_1 \mathbf w_{i_1}+\cdots +v_n\mathbf w_{i_n}} (\mathbf x)\Vert_r s^{\vert v_1\mathbf w_{i_1}+\cdots+v_n\mathbf w_{i_n}\vert}\\
&\leq &s^{\vert -\mathbf u\vert} (\sqrt[C]{s})^{c_1+\cdots+c_N} 
\Vert a_{\mathbf u+v_1 \mathbf w_{i_1}+\cdots +v_n\mathbf w_{i_n}} (\mathbf x)\Vert_r s^{\vert \mathbf u+ v_1\mathbf w_{i_1}+\cdots+v_n\mathbf w_{i_n}\vert}.
\end{eqnarray*}
Hence $\Vert a_{\mathbf u+v_1 \mathbf w_{i_1}+\cdots +v_n\mathbf w_{i_n}} (\mathbf x)\Vert_r (\sqrt[C]{s})^{c_1+\cdots+c_N+v_1+\cdots+v_n}$ are bounded. 
This proves that $\phi: K\langle \mathbf x, \mathbf y\rangle ^\dagger\to L^{\dagger\prime}$  
is surjective. 

(ii) By (i), the map $\varphi: D^\dagger\to L^{\dagger\prime}$ is surjective. For any $f\in L^{\dagger\prime}$, choose $Q\in D^\dagger$ such that 
$\varphi(Q)=f$. For any $P\in D^\dagger$, we have $$Pf=P\varphi(Q)=\varphi(PQ)\in L^{\dagger\prime}.$$ 
So $L^{\dagger\prime}$ is a left $D^\dagger$-module. Then $\varphi$ is a homomorphism of left $D^\dagger$-modules. 
\end{proof} 

\subsection{}\label{GKZprime} Let $$C^k( L^{\dagger\prime})=
\{\sum_{1\leq i_1<\cdots < i_k\leq n}f_{i_1\ldots i_k}(\mathbf x,\mathbf t)\frac{\mathrm dt_{i_1}}{t_{i_1}}\wedge \cdots\wedge \frac{\mathrm dt_{i_k}}{t_{i_k}}:
\;f_{i_1\ldots i_k}(\mathbf x,\mathbf t)
\in  L^{\dagger\prime}\}.$$ Note that $d: C^k( L^\dagger)\to C^{k+1}( L^\dagger)$
(resp. $\nabla_{\frac{\partial}{\partial x_j}}$) maps $C^k( L^{\dagger\prime})$ to $C^{k+1}( L^{\dagger\prime})$ 
(resp. $C^k( L^{\dagger\prime})$), where 
\begin{eqnarray*}
d&=&\big(\mathbf t^\gamma \exp(\pi F(\mathbf x,\mathbf t))\big)^{-1} \circ \mathrm d_{\mathbf t} \circ 
\big( \mathbf t^\gamma\exp(\pi F(\mathbf x,\mathbf t))\big)=
\mathrm d_{\mathbf t}+ \sum_{i=1}^n \Big(\gamma_i+\sum_{j=1}^N \pi w_{ij}x_j\mathbf t^{\mathbf w_j}\Big)\
\frac{\mathrm dt_i}{t_i}\wedge,\\
\nabla_{\frac{\partial}{\partial x_j}}&=& \big(\mathbf t^\gamma \exp(\pi F(\mathbf x,\mathbf t))\big)^{-1} \circ\frac{\partial}{\partial x_j}\circ 
\big( \mathbf t^\gamma\exp(\pi F(\mathbf x,\mathbf t))\big)
=\frac{\partial}{\partial x_j}+\pi  \mathbf t^{\mathbf w_j}. 
\end{eqnarray*}
So $C^\cdot ( L^{\dagger\prime})$ is a subcomplex 
of $C^\cdot( L^\dagger)$.
Let 
\begin{eqnarray*}
F_{i,\gamma}=\big(\mathbf t^\gamma \exp (\pi F(\mathbf x,\mathbf t))\big)^{-1}
\circ t_i \frac{\partial}{\partial t_i} \circ  \big(\mathbf t^\gamma\exp (\pi F(\mathbf x,\mathbf t))\big)
=t_i\frac{\partial}{\partial t_i}+\gamma_i+\pi\sum_{j=1}^N w_{ij}x_j \mathbf t^{\mathbf w_j}.
\end{eqnarray*}
It follows from the definition of the logarithmic twisted de Rham complex that the homomorphism 
$$ L^{\dagger\prime}\to C^n(  L^{\dagger\prime}),\quad f(\mathbf x,\mathbf t)\mapsto f(\mathbf x,\mathbf t)\frac{\mathrm dt_1}{t_1}\wedge \cdots\wedge \frac{\mathrm dt_n}{t_n}$$
induces an isomorphism
 $$   L^{\dagger\prime}
/ \sum_{i=1}^n F_{i,\gamma} L^{\dagger\prime}\cong H^n(C^\cdot( L^{\dagger\prime})).
$$  
Let 
\begin{eqnarray*}
\Lambda&=&\{\lambda=(\lambda_1, \ldots, \lambda_N)\in\mathbb Z^N:\;\sum_{j=1}^N
\lambda_j{\mathbf w}_j=0\},\\
\Box_{\lambda}&=&\prod_{\lambda_j>0} 
\left(\frac{1}{\pi}\frac{\partial}{\partial x_j}\right)^{\lambda_j}-
\prod_{\lambda_j<0} \left(\frac{1}{\pi}\frac{\partial}{\partial x_j}\right)^{-\lambda_j}\quad (\lambda\in \Lambda),\\
E_{i,\gamma}&=&\sum_{j=1}^N w_{ij} x_j\frac{\partial}{\partial x_j}+\gamma_i\quad (i=1,\ldots, n),
\end{eqnarray*}

\begin{proposition}\label{grobner} The homomorphism of $D^\dagger$-modules
$$\varphi: D^\dagger\to  L^{\dagger\prime}, \quad \sum_{\mathbf v\in \mathbb Z^N_{\geq 0}}f_{\mathbf v}(\mathbf x) 
\frac{\partial^{\mathbf v}}{\pi^{\vert \mathbf v\vert}}
\mapsto\Big(\sum_{\mathbf v\in \mathbb Z^N_{\geq 0}}f_{\mathbf v}(\mathbf x)\frac{\partial^{\mathbf v}}{\pi^{\vert \mathbf v\vert}}
\Big)\cdot 1= \sum_{\mathbf v\in \mathbb Z^N_{\geq 0}}f_{\mathbf v}(\mathbf x)   \mathbf t^{v_1\mathbf w_1+\cdots +v_N\mathbf w_N}$$  induces isomorphisms
\begin{eqnarray*}
D^\dagger/\sum_{\lambda\in \Lambda}D^\dagger \Box_\lambda&\stackrel \cong\to&  L^{\dagger\prime},\\
D^\dagger/(\sum_{i=1}^n D^\dagger E_{i,\gamma}+\sum_{\lambda\in \Lambda}D^\dagger \Box_\lambda)
&\stackrel \cong\to&  L^{\dagger\prime}
/ \sum_{i=1}^n F_{i,\gamma}  L^{\dagger\prime}\cong H^n(C^\cdot( L^{\dagger\prime})).
\end{eqnarray*}
Moreover, there exist finitely many $\mu_1, \ldots, \mu_m \in\Lambda$ such that 
$$\sum_{i=1}^mD^\dagger \Box_{\mu_m}=\sum_{\lambda\in \Lambda}D^\dagger \Box_\lambda.$$
\end{proposition}

\begin{proof} We have shown that $\varphi$ is surjective in Lemma \ref{LL}. 
One can verify that $\varphi(\Box_\lambda)=0$ for all $\lambda\in\Lambda$. 
Suppose $\sum_{\mathbf v} f_{\mathbf v}(\mathbf x) \frac{\partial^{\mathbf v}}{\pi^{\vert\mathbf v\vert}}$ lies in the kernel of $\varphi$, that is, 
$$\sum_{\mathbf v} f_{\mathbf v}(\mathbf x)\mathbf t^{v_1\mathbf w_1+\cdots+v_N\mathbf w_N}=0,$$
where $f_{\mathbf v}(\mathbf x)\in K\{r^{-1}\mathbf x\}$ and $\Vert f_{\mathbf v}(\mathbf x)\Vert_r s^{\vert\mathbf v\vert}$ are bounded for some $r,s>1$.  
For each $\mathbf w\in C(A)$, let $$S_{\mathbf w}=\{\mathbf v\in \mathbb Z_{\geq 0}^N:\; \mathbf w=v_1\mathbf w_1+\cdots+v_N\mathbf w_N\}.$$ Then we have 
$$\sum_{\mathbf v\in S_{\mathbf w}} f_{\mathbf v}(\mathbf x)=0.$$ 
For each nonempty $S_{\mathbf w}$, choose an element $\mathbf v^{(0)}=(v^{(0)}_1, \ldots, v^{(0)}_N)\in S_{\mathbf w}$ such that 
$\vert\mathbf v^{(0)}\vert $ is minimal in $S_{\mathbf w}$.
For any $\mathbf v\in S_{\mathbf w}$, let $\lambda_{\mathbf v}=\mathbf v-\mathbf v^{(0)}$. 
We have $\lambda_{\mathbf v}\in\Lambda$ and 
\begin{eqnarray*}
\frac{\partial^{\mathbf v}}{\pi^{\vert\mathbf v\vert}}-\frac{\partial^{\mathbf v^{(0)}}}{\pi^{\vert \mathbf v^{(0)}\vert}}
&=&\frac{\partial^{\min(\mathbf v,\mathbf v^{(0)})}}{\pi^{\sum_j \min(v_j, v_j^{(0)})}}\Big(
\prod_{v_j>v^{(0)}_j}\Big(\frac{1}{\pi}\frac{\partial}{\partial x_j}\Big)^{v_j-v^{(0)}_j} -
\prod_{v_j<v^{(0)}_j}\Big(\frac{1}{\pi}\frac{\partial}{\partial x_j}\Big)^{v^{(0)}_j-v_j}\Big)\\
&=&\frac{\partial^{\min(\mathbf v,\mathbf v^{(0)})}}{\pi^{\sum_j \min(v_j, v_j^{(0)})}}\Box_{\lambda_{\mathbf v}}.
\end{eqnarray*}
So $\frac{\partial^{\mathbf v}}{\pi^{\vert\mathbf v\vert}}-\frac{\partial^{\mathbf v^{(0)}}}{\pi^{\vert \mathbf v^{(0)}\vert}}$ lies 
in the ideal $\sum_{\lambda\in\Lambda} D^\dagger\Box_\lambda$.

Let 
$D=K\Big[\frac{1}{\pi}\frac{\partial}{\partial x_1},\ldots, \frac{1}{\pi} \frac{\partial}{\partial x_N}\Big]$. 
We have an isomorphism of rings $$\phi: D\to K[z_1,\ldots,z_N],\quad \frac{1}{\pi}\frac{\partial}{\partial x_i}\mapsto z_i\;(i=1,\ldots,N).$$ Let 
$\tilde{\Box}_\lambda=\phi(\Box_\lambda)$ ($\lambda\in \Lambda$), and 
let $f_{\mathbf v}=\phi\big(\frac{\partial^{\mathbf v}}{\pi^{\vert\mathbf v\vert}}-\frac{\partial^{\mathbf v^{(0)}}}
{\pi^{\vert \mathbf v^{(0)}\vert}}\big)=\mathbf z^{\mathbf v}-\mathbf z^{\mathbf v^{(0)}}.$ 
Then we have
$$f_{\mathbf v}=\mathbf z^{\min(\mathbf v,\mathbf v^{(0)})}\tilde \Box_{\lambda_{\mathbf v}}.$$ 
Let $\mu_1, \ldots, \mu_m$ be as in Lemma \ref{annoy} below.  We can write
$$\tilde \Box_{\lambda_{\mathbf v}}=b^{(1)} \tilde\Box_{\mu_1}+\cdots +b^{(m)}\tilde\Box_{\mu_m}$$
where each $b^{(i)}$ is a polynomial with integer coefficients whose total degree does not exceed that of 
$\tilde \Box_{\lambda_{\mathbf v}}$. We can thus write 
$$f_{\mathbf v}=a^{(1)}_{\mathbf v} \tilde\Box_{\mu_1}+\cdots +a^{(m)}_{\mathbf v}\tilde\Box_{\mu_m},$$
where each $a^{(i)}$ is a polynomials with integer coefficients whose total degree does not exceed 
$\mathrm{totdeg}\,f_{\mathbf v}=\vert \mathbf v\vert$. So we can write $a^{(i)}_{\mathbf v}=\sum_{\mathbf u}a_{\mathbf{vu}}^{(i)}\mathbf z^{\mathbf u}$
such that $a_{\mathbf{vu}}^{(i)}$ are integers and $a_{\mathbf{vu}}^{(i)}=0$ for $\vert \mathbf u\vert>\vert \mathbf v\vert$. 
We then have
$$\frac{\partial^{\mathbf v}}{\pi^{\vert\mathbf v\vert}}-\frac{\partial^{\mathbf v^{(0)}}}{\pi^{\vert \mathbf v^{(0)}\vert}}
=\Big(\sum_{\mathbf u}
a^{(1)}_{\mathbf v\mathbf u}\frac{\partial^{\mathbf u}}{\pi^{\vert\mathbf u\vert}}\Big)\Box_{\mu_1}+\cdots+
\Big(\sum_{\mathbf u}
a^{(m)}_{\mathbf v\mathbf u}\frac{\partial^{\mathbf u}}{\pi^{\vert\mathbf u\vert}}\Big)\Box_{\mu_m}.$$ 
So
\begin{eqnarray*}
\sum_{\mathbf v} f_{\mathbf v}(\mathbf x) \frac{\partial^{\mathbf v}}{\pi^{\vert\mathbf v\vert}}
&=&\sum_{\mathbf w\in C(A)}
\sum_{\mathbf v\in S_{\mathbf w}}  f_{\mathbf v}(\mathbf x)\frac{\partial^{\mathbf v}}{\pi^{\vert\mathbf v\vert}}\\
&=& \sum_{\mathbf w\in C(A)}\sum_{\mathbf v\in S_{\mathbf w}}  f_{\mathbf v}(\mathbf x) \Big(\frac{\partial^{\mathbf v}}{\pi^{\vert\mathbf v\vert}}
-\frac{\partial^{\mathbf v^{(0)}}}{\pi^{\vert\mathbf v^{(0)}\vert}}\Big)\\
&=& \sum_{k=1}^m\Big( \sum_{\mathbf u}\sum_{\mathbf w\in C(A), \mathbf v\in S_{\mathbf w}}
a^{(k)}_{\mathbf v\mathbf u}  f_{\mathbf v}(\mathbf x)  \frac{\partial^{\mathbf u}}{\pi^{\vert\mathbf u\vert}}\Big) \Box_{\mu_k}.
\end{eqnarray*}
So $\sum_{\mathbf v} f_{\mathbf v}(\mathbf x) \frac{\partial^{\mathbf v}}{\pi^{\vert\mathbf v\vert}}$ lies in the left ideal of
$D^\dagger$ generated by 
$\Box_{\mu_k}$, and hence we have 
$$\mathrm{ker}\,\varphi=\sum_{k=1}^mD^\dagger\Box_{\mu_k}=\sum_{\lambda\in\Lambda} D^\dagger\Box_{\lambda}.$$ 
Here we need to check $\sum_{\mathbf u}\sum_{\mathbf w\in C(A), \mathbf v\in S_{\mathbf w}}
a^{(k)}_{\mathbf v\mathbf u}  f_{\mathbf v}(\mathbf x)  \frac{\partial^{\mathbf u}}{\pi^{\vert\mathbf u\vert}}$ is a well-defined element in 
$D^\dagger$. First note that $\sum_{\mathbf w\in C(A), \mathbf v\in S_{\mathbf w}}
a^{(k)}_{\mathbf v\mathbf u}  f_{\mathbf v}(\mathbf x)$ converges in $K\langle r^{-1} \mathbf x\rangle$ since 
$\Vert f_{\mathbf v}(\mathbf x)\Vert_r s^{\vert\mathbf v\vert}$ are bounded and $\vert a^{(k)}_{\mathbf v\mathbf u} \vert \leq 1$. Moreover,
since $a_{\mathbf v\mathbf u}^{(i)}=0$ for $\vert\mathbf v\vert<\vert \mathbf u\vert$, we have 
\begin{eqnarray*}
\Vert \sum_{\mathbf w\in C(A), \mathbf v\in S_{\mathbf w}}
a^{(k)}_{\mathbf v\mathbf u}  f_{\mathbf v}(\mathbf x)  \Vert_r s^{\vert\mathbf u\vert}&\leq &
\max_{\mathbf w\in C(A), \mathbf v\in S_{\mathbf w},\; \vert\mathbf v\vert\geq \vert \mathbf u\vert} 
\Vert  a^{(k)}_{\mathbf v\mathbf u}  f_{\mathbf v}(\mathbf x)\Vert_r s^{\vert\mathbf u\vert}\\
&\leq&\max_{\mathbf v} 
\Vert f_{\mathbf v}(\mathbf x)\Vert_r s^{\vert\mathbf v\vert}.
\end{eqnarray*}
So $\sum_{\mathbf u}\sum_{\mathbf w\in C(A), \mathbf v\in S_{\mathbf w}}
a^{(k)}_{\mathbf v\mathbf u}  f_{\mathbf v}(\mathbf x)  \frac{\partial^{\mathbf u}}{\pi^{\vert\mathbf u\vert}}$ does define an element in 
$D^\dagger$.

For any $g_i \in L^{\dagger\prime}$ $(i=1, \ldots, n)$, choose $P_i\in D^\dagger$ such that $\varphi(P_i)=g_i$. 
One can check directly that $E_{i, \gamma}(1)=F_{i,\gamma}(1)$. Moreover, $F_{i,\gamma}$ commutes with each 
$\nabla_{\frac{\partial }{\partial x_j}}$ and hence with $P_i$. So
we have 
\begin{eqnarray*}
\varphi(\sum_i P_iE_{i,\gamma})= \sum_i P_i E_{i,\gamma}(1)
=\sum_i P_i F_{i,\gamma}(1)
= \sum_i F_{i,\gamma}P_i(1)
= \sum_i F_{i,\gamma} \varphi(P_i)
= \sum_i F_{i,\gamma} g_i.
\end{eqnarray*}
Therefore
$$\varphi(\sum_i D^\dagger E_{i,\gamma})=\sum_i F_{i,\gamma} L^{\dagger\prime}.$$
Together with the fact that $\varphi$ is surjective and $\mathrm{ker}\,\varphi=\sum_{\lambda\in\Lambda} D^\dagger\Box_\lambda$, we get
$$D^\dagger/\sum_{\lambda\in \Lambda}D^\dagger \Box_\lambda\cong  
L^{\dagger\prime},\quad
D^\dagger/(\sum_{i=1}^nD^\dagger E_{i,\gamma}+\sum_{\lambda\in \Lambda}D^\dagger \Box_\lambda)\cong  
L^{\dagger\prime}/ \sum_{i=1}^n F_{i,\gamma}  L^{\dagger\prime}.$$
\end{proof}

\begin{lemma}\label{annoy} For any $\lambda=(\lambda_1, \ldots, \lambda_N)\in\Lambda$, let 
$$\lambda^+=(\max(\lambda_1, 0), \ldots, \max(\lambda_N,0)),\quad 
\lambda^-=(\max(-\lambda_1, 0), \ldots, \max(-\lambda_N,0)),\quad 
\tilde \Box_\lambda=\mathbf z^{\lambda^+}-\mathbf z^{\lambda^-}.$$ 
There exist finitely many $\mu_1, \ldots, \mu_m\in \Lambda$ such that for any $\lambda\in\Lambda$, we can write 
$$\tilde \Box_\lambda=a^{(1)}\tilde\Box_{\mu_1}+\cdots + a^{(m)}\tilde\Box_{\mu_m},$$
where $a^{(i)}$ are polynomials with integer coefficients
such that their total degrees do not exceed the total degree $\max(\vert\lambda^+\vert, \vert\lambda^-\vert)$ of $\tilde\Box_\lambda$. 
\end{lemma}

\begin{proof} Consider the set $S=\{(\lambda^+, \lambda^-):\, \lambda\in \Lambda-\{0\}\}\subset \mathbf Z^{2N}_{\geq 0}$. By Lemma \ref{comb}, 
there exists finitely many $\mu_1, \ldots, \mu_m\in \Lambda-\{0\}$ such that 
$$S\subset \bigcup_{i=1}^m \Big((\mu_i^+, \mu_i^-)+ \mathbb Z^{2N}_{\geq 0}\Big).$$ 
We prove the lemma by induction on $\vert\lambda^+\vert+ \vert\lambda^-\vert$. For any $\lambda\in\Lambda$, we can find some 
$\mu_i$ such that $(\lambda^+,\lambda^-)\in (\mu_i^+, \mu_i^-)+ \mathbb Z^{2N}_{\geq 0}$. Then $\lambda-\mu_i\in\Lambda$ and 
$$(\lambda-\mu_i)^+=\lambda^+-\mu_i^+, \quad (\lambda-\mu_i)^-=\lambda^--\mu_i^-.$$
In the case where 
$\vert\mu_i^+\vert\geq \vert\mu_i^-\vert$, we have 
\begin{eqnarray*}
\tilde \Box_\lambda&=&\mathbf z^{\lambda^+}-\mathbf z^{\lambda^-}\\
&=& \mathbf z^{\lambda^+-\mu_i^+}(\mathbf z^{\mu_i^+}-\mathbf z^{\mu_i^-})+ 
\mathbf z^{\mu_i^-}(\mathbf z^{\lambda^+-\mu_i^+}-\mathbf z^{\lambda^--\mu_i^-})\\
&=&  \mathbf z^{\lambda^+-\mu_i^+}\tilde\Box_{\mu_i}+\mathbf z^{\mu_i^-}\tilde\Box_{\lambda-\mu_i}.
\end{eqnarray*}
Note that $\vert (\lambda-\mu_i)^+\vert+ \vert(\lambda-\mu_i)^-\vert<\vert\lambda^+\vert+\vert\lambda^-\vert$. By the induction hypothesis, we can write 
$$\tilde\Box_{\lambda-\mu_i}=a^{(1)}\tilde\Box_{\mu_1}+\cdots + a^{(m)}\tilde\Box_{\mu_m}$$ such that 
$a^{(j)}$ are polynomials with integer coefficients
of total degree $\leq \max(\vert\lambda^+-\mu_i^+\vert, \vert\lambda^--\mu_i^-\vert)$.
Then we have 
$$\tilde \Box_\lambda= \mathbf z^{\lambda^+-\mu_i^+}\tilde\Box_{\mu_i}+\sum_{j=1}^m \mathbf z^{\mu_i^-} a^{(j)}\tilde\Box_{\mu_j}.$$
Since we are in the case $\vert\mu_i^+|\geq \vert\mu_i^-|$, the total degrees of $\mathbf z^{\mu_i^-} a^{(j)}$ are $\leq\max(\vert \lambda^+\vert, \vert\lambda^-\vert)$.
This gives the required expression for $\tilde \Box_\lambda$. 
In the case where 
$\vert\mu_i^+\vert\leq \vert\mu_i^-\vert$, we have 
\begin{eqnarray*}
\tilde \Box_\lambda&=&\mathbf z^{\lambda^+}-\mathbf z^{\lambda^-}\\
&=& \mathbf z^{\lambda^--\mu_i^-}(\mathbf z^{\mu_i^+}-\mathbf z^{\mu_i^-})+ 
\mathbf z^{\mu_i^+}(\mathbf z^{\lambda^+-\mu_i^+}-\mathbf z^{\lambda^--\mu_i^-})\\
&=&  \mathbf z^{\lambda^--\mu_i^+}\tilde\Box_{\mu_i}+\mathbf z^{\mu_i^+}\tilde\Box_{\lambda-\mu_i}.
\end{eqnarray*}
We conclude as before using the induction hypothesis. 
\end{proof}

\subsection{Proof of Proposition \ref{coherent}}\label{co}  
We have shown that $L^{\dagger\prime}$ is a $D^\dagger$-module in 
Lemma \ref{LL}. It is known that $D^\dagger$ is coherent (\cite{Hu}). 
So by Proposition \ref{grobner}, $L^{\dagger\prime}$ is a coherent $D^\dagger$-module. 
 
Keep the notation in the proof of Lemma \ref{LL}. Decompose $\Delta$ into a finite union $\bigcup_\tau \tau$
so that each $\tau$ is a  
simplex of dimension $n$ with vertices $\{0,\mathbf w_{i_1},\ldots, 
\mathbf w_{i_n}\}$ for some subset $\{i_1,\ldots, i_n\}\subset \{1,\ldots, N\}$.
Let $B$ be the finite set $\bigcup_{\tau}B(\tau)$. Consider the map $$\psi:
\bigoplus_{\beta\in B} L^{\dagger\prime}\to L^\dagger,\quad (f_\beta)\mapsto \sum_{\beta\in B}f_\beta {\mathbf t}^\beta.
$$
We will show $\psi$ is surjective. As in the proof of Lemma \ref{LL} (ii), this implies that $L^\dagger$ is a left $D^\dagger$-module
and $\psi$ is a homomorphism of $D^\dagger$-modules. We will prove
$\mathrm{ker}\,\psi$ is a finitely generated $D^\dagger$-module. Combined with the fact that 
$L^{\dagger\prime}$ is a coherent $D^\dagger$-module, this implies that $L^{\dagger}$ is a coherent 
$D^\dagger$-module. 

We have 
$\mathbb Z^n\cap \delta=\bigcup_\tau (\mathbb Z^n\cap \delta(\tau)).$ To prove $\psi$ is surjective, it suffices to show every element in $L^\dagger$ of the form
$\sum_{\mathbf w\in \mathbb Z^n\cap \delta(\tau)} a_{\mathbf w}(\mathbf x) t^{\mathbf w}$ lies in the image of $\psi$,
where 
$a_{\mathbf w}(\mathbf x)\in K\{r^{-1}\mathbf x\}$ and  
$\Vert a_{\mathbf w}(\mathbf x)\Vert_r  s^{\vert \mathbf w\vert}$ are bounded for some $r,s>1$. 
Every element $\mathbf w\in \mathbb Z^n\cap \delta(\tau)$ can be written uniquely as 
$$\mathbf w=b(\mathbf w)+c(\mathbf w)$$ with $b(\mathbf w)\in B(\tau)$ and $c(\mathbf w)\in C(\tau)$. We have 
$$\sum_{\mathbf w\in \mathbb Z^n\cap \delta(\tau)} a_{\mathbf w}(x)\mathbf t^{\mathbf w}=\sum_{\beta\in B(\tau)}\Big(\sum_{\mathbf w\in \mathbb Z^n\cap \delta(\tau),\;
b(\mathbf w)=\beta} a_{\mathbf w}(x)\mathbf t^{c(\mathbf w)}\Big) {\mathbf t}^\beta.$$
Note that $\sum\limits_{\mathbf w\in \mathbb Z^n\cap \delta(\tau),\;
b(\mathbf w)=\beta} a_{\mathbf w}(x) t^{c(\mathbf w)}$ lie in $L^{\dagger\prime}$ since
$$\Vert a_{\mathbf w}(\mathbf x)\Vert_r  s^{\vert c(\mathbf w)\vert} 
\leq (\Vert a_{\mathbf w}(\mathbf x)\Vert_r  s^{\vert \mathbf w\vert}) s^{\vert-b(\mathbf w)\vert}$$
are bounded. Thus $\psi$ is surjective. 

Given $\beta',\beta''\in B$, set 
\begin{eqnarray*}
L_{\beta',\beta''}&=&\{f\in L^{\dagger\prime}:\; f\mathbf t^{\beta'-\beta''}\in L^{\dagger \prime}\}, \\
S_{\beta',\beta''}&=&\{\mathbf w\in C(A): \mathbf w+\beta'-\beta''\in C(A)\}.
\end{eqnarray*}
We have $$L_{\beta',\beta''}=\bigcup_{r>1,\; s>1} \{\sum_{\mathbf w\in S_{\beta',\beta''}} a_{\mathbf w}(\mathbf x)\mathbf t^{\mathbf w}:
\; a_{\mathbf w}(\mathbf x)\in  K\{r^{-1}\mathbf x\},\; \Vert a_{\mathbf w}(\mathbf x)\Vert_r s^{\vert \mathbf w\vert} \hbox{ are bounded for some }
r, s>1\}.$$ 
We have
$S_{\beta',\beta''}+\mathbf w_j\subset S_{\beta',\beta''}$ for all $j$, and $L_{\beta',\beta''}$ is 
a $\mathcal D^\dagger$-submodule of $L^{\dagger\prime}$. For any $f\in L_{\beta',\beta''}$ and $\beta\in B$, let 
$$\iota_{\beta',\beta''}(f)_\beta=\left\{
\begin{array}{cl}
f &\hbox{if } \beta=\beta', \\
-ft^{\beta'-\beta''}&\hbox{if }\beta=\beta'', \\
0&\hbox{if } \beta\in B\backslash\{\beta',\beta''\}.
\end{array}
\right.$$ 
Then the map 
$$\iota_{\beta', \beta''}: L_{\beta', \beta''}\to \bigoplus_{\beta\in B} L^{\dagger\prime},\quad f\mapsto (\iota_{\beta',\beta''}(f)_\beta)_{\beta\in B}$$
is a homomorphism of $\mathcal D^\dagger$-modules and its image is contained in $\mathrm{ker}\, \psi$. We will prove 
each $L_{\beta',\beta''}$ is a finitely generated $\mathcal D^\dagger$-module, and $$\mathrm{ker}\, \psi=\sum_{\beta',\beta''}
\iota_{\beta', \beta''}(L_{\beta', \beta''}).$$
It follows that $\mathrm{ker}\,\psi$ is a finitely generated $D^\dagger$-module. 

We have
$$S_{\beta',\beta''}=\bigcup_\tau(S_{\beta',\beta''}\cap \delta(\tau))=
\bigcup_\tau \bigcup_{\mathbf w\in B(\tau)}(S_{\beta',\beta''}\cap (\mathbf w+C(\tau))).$$
By Lemma \ref{comb}, for each $S_{\beta',\beta''}\cap (\mathbf w+C(\tau))$, we can find finitely many 
$\mathbf u_1, \ldots, \mathbf u_m\in S_{\beta',\beta''}\cap (\mathbf w+C(\tau))$ such that 
$$S_{\beta',\beta''}\cap (\mathbf w+C(\tau))= \bigcup_{i=1}^m (\mathbf u_i+C(\tau)). $$
We thus decompose $S_{\beta',\beta''}$ into a finite union of subsets of the form $\mathbf u+C(\tau)$ such that
$\tau$ is a simplicial complex 
of dimension $n$ with vertices $\{0,\mathbf w_{i_1},\ldots, 
\mathbf w_{i_n}\}$ for some subset $\{i_1,\ldots, i_n\}\subset \{1,\ldots, N\}$, and 
$\mathbf u\in 
S_{\beta',\beta''}\cap (\mathbf w+C(\tau))$ for some $\mathbf w\in B(\tau)$. Then $L_{\beta',\beta''}$ is generated 
by these $\mathbf t^{\mathbf u}$ as a $D^\dagger$-module. Indeed, every element in $L_{\beta',\beta''}$ is a sum of elements 
of the form $\sum_{\mathbf w\in \mathbf u+C(\tau)} a_{\mathbf w}(\mathbf x)\mathbf t^{\mathbf w}$. 
We have 
\begin{eqnarray*}
\sum_{\mathbf w\in \mathbf u+C(\tau)} a_{\mathbf w}(\mathbf x)\mathbf t^{\mathbf w}
=\sum_{v_1, \ldots, v_n\geq 0} a_{\mathbf u+v_1 \mathbf w_{i_1}+\cdots +v_n\mathbf w_{i_n}} (\mathbf x)
\Big(\frac{1}{\pi} \frac{\partial}{\partial x_{i_1}}\Big)^{v_{1}}\cdots \Big(\frac{1}{\pi}\frac {\partial}{\partial x_{i_m}}\Big)^{v_{n}} \cdot \mathbf t^{\mathbf u}.
\end{eqnarray*}

Suppose $(f^{(0)}_\beta)\in\bigoplus_{\beta\in B}L^{\dagger\prime}$ is an element in 
$\mathrm{ker}\,\psi$. We then have
$$\sum_{\beta\in B} f^{(0)}_\beta \mathbf t^\beta=0.$$ Suppose $B=\{\beta_1, \ldots, \beta_k\}$, and write 
$$f^{(0)}_\beta=\sum_{\mathbf w\in C(A)} a_{\beta\mathbf w}(\mathbf x) \mathbf t^{\mathbf w}.$$ 
Define 
\begin{eqnarray*}
f_\beta^{(1)}&=&\sum_{\mathbf w\in C(A), \; \mathbf w+(\beta-\beta_1)\not\in C(A)} a_{\beta\mathbf w}(\mathbf x)t^{\mathbf w},\\
g_\beta^{(1)}&=&\sum_{\mathbf w\in C(A), \; \mathbf w+(\beta-\beta_1)\in C(A)} a_{\beta\mathbf w}(\mathbf x)t^{\mathbf w}.
\end{eqnarray*}
In particular, $f_{\beta_1}^{(1)}$ is $0$ since it is a sum over the empty set. 
We have $g_\beta^{(1)}\in L_{\beta,\beta_1}$ and
\begin{eqnarray}\label{fg}
(f^{(0)}_\beta)-\sum_{\alpha\in B\backslash\{\beta_1\}}\iota_{\alpha,\beta_1}(g^{(1)}_\alpha)=(f^{(1)}_\beta).
\end{eqnarray}
To verify this equation, we show it holds componentwisely. The equation clearly holds for those component $\beta\not=\beta_1$. 
Collecting the linear combination of $\mathbf t^{\mathbf w}$ with $\mathbf w\in C(A)$ on both sides of  
the equation $$\sum_{\beta\in B} f^{(0)}_\beta \mathbf t^{\beta-\beta_1}=0,$$
we get 
$$f_{\beta_1}^{(0)}+\sum_{\beta\in B\backslash\{\beta_1\}} g_\beta^{(1)} \mathbf t^{\beta-\beta_1}=0.$$
This is exactly the $\beta_1$ component of the equation (\ref{fg}).

In general, for $i=1, \ldots, k$, we define 
\begin{eqnarray*}
f_\beta^{(i)}&=&\sum_{\mathbf w\in C(A), \; \mathbf w+(\beta-\beta_1)\not\in C(A),\ldots,\; \mathbf w+(\beta-\beta_i)\not\in C(A)} a_{\beta\mathbf w}(\mathbf x)t^{\mathbf w},\\
g_\beta^{(i)}&=&\sum_{\mathbf w\in C(A), \; \mathbf w+(\beta-\beta_1)\not \in C(A),\ldots, \;  \mathbf w+(\beta-\beta_{i-1})\not\in C(A), 
\mathbf w+(\beta-\beta_i)\in C(A)} a_{\beta\mathbf w}(\mathbf x)t^{\mathbf w}.
\end{eqnarray*}
We have $g_\beta^{(i)}\in L_{\beta,\beta_i}$ and 
$$(f^{(i-1)}_\beta)-\sum_{\alpha\in B\backslash\{\beta_i\}}\iota_{\alpha,\beta_i}(g^{(i)}_\alpha)=(f^{(i)}_\beta).$$
We have $f^{(k)}_\beta=0$ for all $\beta\in B=\{\beta_1, \ldots, \beta_k\}$. So we have 
$$(f_\beta^{(0)})=\sum_{i=1}^k\sum_{\alpha\in B\backslash\{\beta_i\}}\iota_{\alpha,\beta_i}(g^{(i)}_\alpha).$$ 
Hence $\mathrm{ker}\, \psi=\sum_{\beta',\beta''}
\iota_{\beta', \beta''}(L_{\beta', \beta''}).$ This finishes the proof of Proposition \ref{coherent}. 

\medskip
The results in this section can be used to prove the following. 

\begin{lemma}\label{flat} ${}$ 

(i) $L^\dagger$ is flat over $K\langle \mathbf x\rangle^\dagger$. 

(ii) Let $\mathbf a=(a_1, \ldots, a_N)$ be a 
point in the closed polydisc $D(0, 1^+)^N$ such that $a_i\in K'$, where $K'$ is a finite extension of $K$. Regard $K'$ as a $K\langle \mathbf x\rangle^\dagger$-algebra
via the homomorphism $$K\langle \mathbf x\rangle^\dagger\to K',\quad x_j\mapsto a_j.$$ We have 
$L^\dagger\otimes_{K\langle \mathbf x\rangle^\dagger} K'\cong L_0^\dagger.$
\end{lemma}

\begin{proof} Let $R$ be the integer ring of $K$, and let 
\begin{eqnarray*}
R\langle \mathbf x\rangle^\dagger&=&\bigcup_{r>1}\{ \sum_{\mathbf v\in\mathbb Z^N_{\geq 0}} 
a_{\mathbf v}\mathbf x^{\mathbf v}:\;
a_{\mathbf v}\in R, \; \vert a_{\mathbf v}\vert r^{\vert \mathbf v\vert} \hbox{ are bounded }\}, \\
R\langle \mathbf x,\mathbf y\rangle^\dagger&=&\bigcup_{r>1, \; s>1}\{ \sum_{\mathbf u,\mathbf v\in\mathbb Z^N_{\geq 0}} 
a_{\mathbf u\mathbf v}\mathbf x^{\mathbf u}\mathbf y^{\mathbf v}:\;
a_{\mathbf u\mathbf v}\in R, \; \vert a_{\mathbf u\mathbf v}\vert r^{\vert \mathbf u\vert}s^{\vert\mathbf v\vert} \hbox{ are bounded }\}, \\
L_R^\dagger&=&\bigcup_{r>1, s>1} \{\sum_{\mathbf v\in\mathbb Z^N_{\geq 0},\; 
\mathbf w\in\mathbb Z^n\cap\delta} a_{\mathbf v\mathbf w} \mathbf x^{\mathbf v}\mathbf t^{\mathbf w}:\;  a_{\mathbf v\mathbf w}\in R,
\; \vert a_{\mathbf v\mathbf w}\vert r^{\vert \mathbf v\vert} s^{\vert \mathbf w\vert} \hbox { are bounded}\},\\
L_R^{\dagger\prime}&=&\bigcup_{r>1, s>1} \{\sum_{\mathbf v\in\mathbb Z^N_{\geq 0},\; 
\mathbf w\in C(A)} a_{\mathbf v\mathbf w} \mathbf x^{\mathbf v}\mathbf t^{\mathbf w}:\;  a_{\mathbf v\mathbf w}\in R,
\; \vert a_{\mathbf v\mathbf w}\vert r^{\vert \mathbf v\vert} s^{\vert \mathbf w\vert} \hbox { are bounded}\}.
\end{eqnarray*}
We have 
$$K\langle \mathbf x\rangle^\dagger\cong R\langle \mathbf x\rangle^\dagger\otimes_R K, \quad 
L^\dagger\cong L_R^\dagger\otimes_R K.$$ 
To prove $L^\dagger$ is flat over $K\langle \mathbf x\rangle^\dagger$, it suffices to show 
$L_R^\dagger$ is flat over $R\langle \mathbf x\rangle^\dagger$. 

Keep the notation in the proof of Lemma \ref{LL} and \ref{co}. The same proof shows that the following homomorphisms 
\begin{eqnarray*}
\bigoplus_{\beta\in B} L_R^{\dagger\prime}\to L_R^\dagger, && (f_\beta)\mapsto \sum_{\beta\in B}f_\beta {\mathbf t}^\beta,\\
R\langle \mathbf x, \mathbf y\rangle^\dagger\to L_R^{\dagger\prime}, &&
\sum_{\mathbf v\in \mathbb Z^N_{\geq 0}}f_{\mathbf v}(\mathbf x) 
\mathbf y^{\mathbf v}
\mapsto\sum_{\mathbf v\in \mathbb Z^N_{\geq 0}}f_{\mathbf v}(\mathbf x)  t^{v_1\mathbf w_N+\cdots +v_N\mathbf w_N}
\end{eqnarray*}
are surjective. By \cite{Fulton}, $R\langle\mathbf x, \mathbf y\rangle^\dagger$ is a Noetherian ring.
It follows that $L_R^\dagger$ is also Noetherian. We have 
\begin{eqnarray*}
L_R^\dagger /\pi^k L_R^\dagger \cong (R/\pi^k)[\mathbf x][\mathbb Z^n\cap \delta],\quad 
R\langle \mathbf x\rangle^\dagger/\pi^k R\langle x\rangle^\dagger\cong (R/\pi^k)[\mathbf x].
\end{eqnarray*}
So $L_R^\dagger /\pi^k L_R^\dagger $ is flat over $R\langle \mathbf x\rangle^\dagger /\pi^kR
\langle \mathbf x\rangle^\dagger $ for all $k$. 
By \cite[IV Th\'eor\`eme 5.6]{SGA1}, $L_R^\dagger$ is flat over $R\langle \mathbf x\rangle^\dagger$. 

Finally let's prove 
$L^\dagger\otimes_{K\langle \mathbf x\rangle^\dagger}K'\cong L^\dagger_0.$
One can verify directly that in the case where $K'=K$, the homomorphism 
$$L^\dagger \to L_0^\dagger,\quad \sum_{\mathbf w\in \mathbb Z^n\cap\delta}a_{\mathbf w}(\mathbf x)\mathbf t^{\mathbf w}
\mapsto  \sum_{\mathbf w\in \mathbb Z^n\cap\delta}a_{\mathbf w}(0)\mathbf t^{\mathbf w}$$ 
is surjective with kernel $(x_1, \ldots, x_N)L^\dagger$. This proves our assertion in the case where 
$K=K'$ and $\mathbf a=(0,\ldots, 0)$. In general, we have an isomorphism 
$L^\dagger\otimes_K K'\cong L_{K'}^\dagger$, 
where
$$L_{K'}^\dagger=\bigcup_{r>1, s>1} \{\sum_{\mathbf v\in\mathbb Z^N_{\geq 0},\; 
\mathbf w\in\mathbb Z^n\cap\delta} a_{\mathbf v\mathbf w} \mathbf x^{\mathbf v}\mathbf t^{\mathbf w}:\;  a_{\mathbf v\mathbf w}\in K',
\; \vert a_{\mathbf v\mathbf w}\vert r^{\vert \mathbf v\vert} s^{\vert \mathbf w\vert} \hbox { are bounded}\}.$$
By base change from $K$ to $K'$ and using this isomorphism,
we can reduce to the case where $K'=K$. Then using the automorphism 
$$K'\langle \mathbf x\rangle^\dagger \to K'\langle \mathbf x\rangle^\dagger,
\quad x_i\mapsto x_i-a_i, $$ we can reduce to the case where $\mathbf a=(0,\ldots, 0)$. 
\end{proof}

\section{Dwork's theory}

\subsection{}\label{trace}  Let 
\begin{eqnarray*}
\theta(z)=\exp(\pi z-\pi z^p),\quad \theta_m(z)=\exp(\pi z-\pi z^{p^m})=\prod_{i=0}^{m-1}\theta(z^{p^i}).
\end{eqnarray*}
Then $\theta_m (z)$ converges in a disc of radius $>1$, and the value $\theta(1)=\theta(z)|_{z=1}$ 
of the power series $\theta(z)$ at $z=1$ is a primitive $p$-th root of unity in $K$ (\cite[Theorems 4.1 and 4.3]{M1}). 
Let $\bar u\in\mathbb F_{p^m}$ and let $u\in \overline{\mathbb Q}_p$ be its Techm\"uller lifting so that $u^{p^m}=u$.  
Then we have (\cite[Theorem 4.4]{M1})
$$\theta_m(z)|_{z=u}=\theta(1)^{\mathrm{Tr}_{\mathbb F_{p^m}/\mathbb F_p}(\bar u)}.$$
From now on, we denote elements in finite fields by letters with bars such as $\bar u, \bar a_j, \bar u_i$ ... and denote their 
Techm\"uller liftings by the same letters without
bars such as $u,a_j, u_i$ ...  
Let $\psi_m:\mathbb F_{q^m}\to K^\ast$ be the additive character defined by 
$$\psi_m(\bar u)=\theta(1)^{\mathrm{Tr}_{\mathbb F_{q^m}/\mathbb F_p}(\bar u)}.$$
Then we have
$$\psi_m(\bar u)=\exp(\pi z-\pi z^{q^m})|_{z=u}.$$ 
Denote $\psi_1$ by $\psi$. We have $\psi_m=\psi\circ \mathrm{Tr}_{\mathbb F_{q^m}/\mathbb F_q}.$
Let $\bar a_1,\ldots, \bar a_N\in \mathbb F_q$. 
For any $\bar u_1,\ldots, \bar u_n \in \mathbb F_{q^m}^*$, we have 
\begin{eqnarray}\label{add}
\begin{array}{ccl}
\psi \Big(\mathrm{Tr}_{\mathbb F_{q^m}/\mathbb F_q}(\sum_{j=1}^N \bar a_j \bar u_1^{w_{1j}}\cdots \bar u_n^{w_{nj}})\Big)
&=&\prod_{j=1}^N \psi_m (\bar a_j \bar u_1^{w_{1j}}\cdots \bar u_n^{w_{nj}})  \\
&= &\prod_{j=1}^N\exp(\pi z-\pi z^{q^m})|_{z=a_ju_1^{w_{1j}}\cdots u_n^{w_{nj}}}.
\end{array}
\end{eqnarray}

Let $\chi: \mathbb F_q^*\to \overline{\mathbb Q}_p^*$ be the Techm\"uller character, that is, $\chi(\bar u)=u$ is the Techm\"uller lifting of 
$\bar u\in\mathbb F_q$.  It is a generator for the group of multiplicative characters on $\mathbb F_q$. Any multiplicative character $\mathbb F_q^*\to 
\overline{\mathbb Q}_p^*$ is of the form $\chi_{\gamma}=\chi^{\gamma (1-q)}$ for some rational number $\gamma\in \frac{1}{1-q}\mathbb Z$.  Moreover, for any 
$\bar u\in \mathbb F_{q^m}$, we have
\begin{eqnarray}\label{multi}
\chi_\gamma (\mathrm{Norm}_{\mathbb F_q^m/\mathbb F_q}(\bar u))= (u^{1+q+\cdots+q^{m-1}})^{\gamma(1-q)}=u^{\gamma(1-q^m)},
\end{eqnarray}
Let $\gamma_1, \ldots, \gamma_n\in \frac{1}{1-q}\mathbb Z$. Set $\chi_i=\chi^{\gamma_i(1-q)}$ $(i=1,\ldots, n)$. 

Consider the twisted exponential sum 
$$S_m (F(\bar{\mathbf a},\mathbf t))=\sum_{\bar u_1,\ldots, \bar u_n\in \mathbb F_{q^m}^*}
\chi_1 (\mathrm{Norm}_{\mathbb F_q^m/\mathbb F_q}(\bar u_1))
\cdots\chi_n(\mathrm{Norm}_{\mathbb F_q^m/\mathbb F_q}(\bar u_n))\psi \Big(\mathrm{Tr}_{\mathbb F_{q^m}/\mathbb F_q}\Big(\sum_{j=1}^N \bar a_j \bar u_1^{w_{1j}}\cdots \bar u_n^{w_{nj}}\Big)\Big).$$ Write $\exp(\pi z-\pi z^{q^m})=\sum_{i=1}^\infty c_i z^i$. 
By the equations (\ref{add}) and (\ref{multi}), we have 
\begin{eqnarray*}
&&S_m (F(\bar{\mathbf a},\mathbf t))\\
&=&\sum_{u_i^{q^m-1}=1} u_1^{\gamma_1(1-q^m)}\cdots u_n^{\gamma_n(1-q^m)} 
\prod_{j=1}^N \exp(\pi z-\pi z^{q^m})|_{z=a_ju_1^{w_{1j}}\cdots u_n^{w_{nj}}}\\
&=& \sum_{u_i^{q^m-1}=1} u_1^{\gamma_1(1-q^m)}\cdots u_n^{\gamma_n(1-q^m)} 
\prod_{j=1}^N \Big(\sum_{i=1}^\infty c_i (a_ju_1^{w_{1j}}\cdots u_n^{w_{nj}})^i\Big)\\
&=& \sum_{u_i^{q^m-1}=1} \left(t_1^{\gamma_1(1-q^m)}\cdots t_n^{\gamma_n(1-q^m)} 
\prod_{j=1}^N \Big(\sum_{i=1}^\infty c_i (a_jt_1^{w_{1j}}\cdots t_n^{w_{nj}})^i\Big)\right)|_{t_i=u_i}\\
&=& \sum_{u_i^{q^m-1}=1} \left(t_1^{\gamma_1(1-q^m)}\cdots t_n^{\gamma_n(1-q^m)}\prod_{j=1}^N \exp\big(\pi a_jt_1^{w_{1j}}\cdots t_n^{w_{nj}}- 
\pi a_jt_1^{q^mw_{1j}}\cdots t_n^{q^mw_{nj}}\big)
\right)|_{t_i=u_i} \\
&=& \sum_{u_i^{q^m-1}=1} \left(t_1^{\gamma_1(1-q^m)}\cdots t_n^{\gamma_n(1-q^m)}\exp\big(\pi F(\mathbf a,\mathbf t)- \pi F(\mathbf a,\mathbf t^{q^m})\big)
\right)|_{t_i=u_i}.
\end{eqnarray*}
We thus have 
\begin{eqnarray}\label{addagain}
\qquad \qquad S_m (F(\bar{\mathbf a},\mathbf t))
=\sum_{u_i^{q^m-1}=1} \left(t_1^{\gamma_1(1-q^m)}\cdots t_n^{\gamma_n(1-q^m)}\exp\big(\pi F(\mathbf a,\mathbf t)- \pi F(\mathbf a,\mathbf t^{q^m})\big)
\right)|_{t_i=u_i}.
\end{eqnarray}

\subsection{} Let $K'$ be a finite extension of $K$ containing all $q$-th roots of unity. Set
\begin{eqnarray*}
L(s)_0&=&\{\sum_{\mathbf w\in\mathbb Z^n\cap\delta} a_{\mathbf w} t^{\mathbf w}:\; a_{\mathbf w}\in K', \; 
\vert a_{\mathbf w}\vert s^{\vert \mathbf w\vert} \hbox { are bounded}\}.
\end{eqnarray*}
We have $L^\dagger_0=\bigcup_{s>1}L(s)_0.$ 
Note that $L(s)_0$ $(s> 1)$ and $L^\dagger_0$ are rings. Each $L(s)_0$ is a Banach space with respect to the norm
$$\Vert \sum_{\mathbf w\in\mathbb Z^n\cap\delta} a_{\mathbf w} t^{\mathbf w}\Vert=\sup_{\mathbf w\in\mathbb Z^n\cap \delta} 
\vert a_{\mathbf w}\vert s^{\vert \mathbf w\vert}.$$  

\begin{theorem}[Dwork trace formula] \label{dwork1} The operator 
$G_{\mathbf a}:\;  L^\dagger_0\to  L^\dagger_0$ is nuclear, and we have 
\begin{eqnarray*}
(q^m-1)^n \mathrm{Tr}(G_{\mathbf a}^m, { L^\dagger_0})= 
\sum_{u_i^{q^m-1}=1} \left(t_1^{\gamma_1(1-q^m)}\cdots t_n^{\gamma_n(1-q^m)}\exp\big(\pi F(\mathbf a,\mathbf t)- \pi F(\mathbf a,\mathbf t^{q^m})\big)
\right)|_{t_i=u_i}.
\end{eqnarray*}
\end{theorem}

\begin{proof} For any real number $s>1$, define
$$\tilde L(s)_0=\{\sum_{\mathbf w\in\mathbb Z^n\cap\delta} a_{\mathbf w} t^{\mathbf w}:\; a_{\mathbf w}\in K', \; 
\lim_{\vert \mathbf w\vert\to \infty} \vert a_{\mathbf w}\vert s^{\vert \mathbf w\vert} =0\}.$$
For any $s<s'$, we have $$L(s')_0\subset \tilde L(s)_0\subset L(s)_0,$$  
and $L^\dagger_0=\bigcup_{s>1}\tilde L(s)_0$. 
Endow $\tilde L(s)_0$ with the norm
$$\Vert \sum_{\mathbf w\in\mathbb Z^n\cap\delta} a_{\mathbf w}t^{\mathbf w}\Vert=\sup_{\mathbf w\in\mathbb Z^n\cap \delta} 
\vert a_{\mathbf w}\vert s^{\vert \mathbf w\vert}.$$ 
Then $\tilde L(s)_0$ is a Banach space. 
The inclusion $L(s')_0\hookrightarrow \tilde L(s)_0$ is completely continuous. Indeed, choose $s<s''<s'$. We can factorize this inclusion as the composite 
$$L(s')_0\hookrightarrow \tilde L(s'')_0\hookrightarrow \tilde L(s)_0.$$ It suffices to verify the inclusion $i: \tilde L(s'')_0\hookrightarrow\tilde L(s)_0$ is completely continuous.
Let $L_S$ be the finite dimensional $K'$-vector space spanned by a finite subset $S$ of   
$\{t^{\mathbf w}\}_{\mathbf w\in\mathbb Z^n\cap \delta}$, and let $$i_S: \tilde L(s'')_0\to\tilde L(s)_0$$ be the composite of the projection 
$\tilde L(s'')_0\to L_S$ and the inclusion $L_S\hookrightarrow\tilde L(s)_0$. One can verify that 
$$\Vert i_S -i\Vert\leq \sup_{\mathbf w\not\in S}\Big(\frac{s}{s''}\Big)^{\vert \mathbf w\vert}.$$
So $i_S$ converges to 
$i$ as $S$ goes over all finite subsets of  $\{t^{\mathbf w}\}_{\mathbf w\in\mathbb Z^n\cap \delta}$. Moreover $i_S$ has finite 
ranks. So $i$ is completely continuous. 

Let 
$H(\mathbf t) =\mathbf t^{(1-q)\gamma}\exp\big(\pi F(\mathbf a,\mathbf t)- \pi F(\mathbf a,\mathbf t^{q})\big).$
By Proposition \ref{estimation}, we have $H_q(\mathbf t)\in L(p^{\frac{p-1}{pqw}})_0$. For any $s> 1$, we have 
$\Psi_{\mathbf a}( L(s)_0)\subset  L(s^q)_0$.
Consider the operator
\begin{eqnarray*}
G_{\mathbf a}= \big( \mathbf t^\gamma \exp(\pi F(\mathbf a,\mathbf t))\big)^{-1} \circ \Psi_{\mathbf a}\circ 
\big(\mathbf t^\gamma\exp(\pi F(\mathbf a,\mathbf t))\big) =\Psi_{\mathbf a}\circ H(\mathbf t).
\end{eqnarray*}
If $1< s< p^{\frac{p-1}{pw}}$, then $G_{\mathbf a}$ induces a map $G_{\mathbf a}:\;  \tilde L(s)_0\to  
\tilde L(s)_0$. Indeed it is the composite
$$ \tilde L(s)_0\hookrightarrow  
L(s)_0\stackrel{H(\mathbf t)}\to   L\Big(\min\Big(s,p^{\frac{p-1}{pqw}}\Big)\Big)_0\stackrel {\Psi_{\mathbf a}}\to 
  L\Big(\min\Big(s^q,p^{\frac{p-1}{pw}}\Big)\Big)_0\hookrightarrow   \tilde L(s)_0.$$ 
$G_{\mathbf a}:\;  \tilde L(s)_0\to  \tilde L(s)_0$ 
is completely continuous since the last inclusion in the above composite is completely continuous. In particular, it is nuclear (\cite[Theorem 6.9]{M1}). 
Write $$\mathbf t^{(1-q)\gamma}
\exp\big(\pi F(\mathbf a,\mathbf t)- \pi F(\mathbf a,\mathbf t^{q})\big)=\sum_{\mathbf w\in\mathbb Z^n\cap\delta} c_{\mathbf w} t^{\mathbf w}.$$ 
For any $\mathbf u\in \mathbb Z^n\cap\delta$, we have 
\begin{eqnarray*}
G_{\mathbf a}(\mathbf t^{\mathbf u})=
\Psi_{\mathbf a}( \sum_{\mathbf w\in \mathbb Z^n\cap\delta} c_{\mathbf w}t^{\mathbf w+\mathbf u})=
\Psi_{\mathbf a}( \sum_{\mathbf w\in \mathbf u+\mathbb Z^n\cap\delta} c_{\mathbf w-\mathbf u}t^{\mathbf w})
=\sum_{\mathbf w\in \mathbb Z^n\cap\delta,\; q\mathbf w\in \mathbf u+\mathbb Z^n\cap\delta} 
c_{q\mathbf w-\mathbf u}\mathbf t^{\mathbf w}.
\end{eqnarray*} 
The coefficient of $ \mathbf t^{\mathbf u}$ 
on the righthand side is $c_{(q-1)\mathbf u}$. Taking the sum of diagonal entries of the matrix of $G_{\mathbf a}$ on $\tilde L(s)_0$ 
with respect to $\{\mathbf t^{\mathbf w}\}_{\mathbf w\in\mathbb Z^n\cap\delta}$, we get
\begin{eqnarray*}
\mathrm{Tr}(G_{\mathbf a},  {\tilde L(s)_0})&=&\sum_{\mathbf u\in \mathbb Z^n\cap \delta}c_{(q-1)\mathbf u}.
\end{eqnarray*}
In particular, $\mathrm{Tr}(G_{\mathbf a},  {\tilde L(s)_0})$ is independent of $s$. Similarly,
$\mathrm{Tr}(G^m_{\mathbf a},  {\tilde L(s)_0})$ and $$\mathrm{det}(I-TG_{\mathbf a},  {\tilde L(s)_0})
=\exp\Big(-\sum_{m=1}^\infty
\frac {\mathrm{Tr}(G^m_{\mathbf a},  {\tilde L(s)_0})}{m}T^m\Big)$$ are independent 
of $s$. For any monic irreducible polynomial $f(T)\in K'[T]$ with nonzero constant term, write (\cite[Theorem 6.9]{M1})
$$ \tilde L(s)_0=N(s)_f\bigoplus W(s)_f,$$
where $N(s)_f$ and $W(s)_f$ are $G_{\mathbf a}$-invariant spaces, $N(s)_f$ is finite dimensional over $K'$, $f(G_{\mathbf a})$ is nilpotent on 
$N(s)_f$ and bijective on $W(s)_f$. 
We have 
$$N(s)_f=\bigcup_{m=1}^\infty \mathrm{ker}\, (f(G_{\mathbf a}))^m,\quad 
W(s)_f=\bigcap_{m=1}^\infty \mathrm{im}\, (f(G_{\mathbf a}))^m.$$
For any pair $s<s'$, we have 
$$\tilde L(s')_0\subset \tilde L(s)_0,\quad N(s')_f\subset N(s)_f,\quad W(s')_f\subset W(s)_f.$$ 
Let $N_f=\bigcup_{1< s< p^{\frac{p-1}{pw}}} N(s)_f$ and $W_f=\bigcup_{1< s< p^{\frac{p-1}{pw}}}W(s)_f$. Then
$$ L_0^\dagger=N_f\bigoplus W_f,$$
$N_f$ and $W_f$ are $G_{\mathbf a}$-invariant, $f(G_{\mathbf a})$ is nilpotent on $N_f$ and bijective on $W_f$. Since 
$\mathrm{det}(I-TG_{\mathbf a}, \tilde L(s)_0)$ is independent of $s$, all $N(s)_f$ have the same dimension, and hence 
we have $N_f=N(s)_f$ for all $1< s< p^{\frac{p-1}{pw}}$. 
This shows that $G_{\mathbf a}:  L^\dagger_0\to   L^\dagger_0$ is nuclear and 
$$\mathrm{Tr}(G_{\mathbf a},  {L^\dagger_0})=\sum_{\mathbf u\in \mathbb Z^n\cap \delta}c_{(q-1)\mathbf u}.$$
On the other hand, we have
$$\sum_{u^{q-1}=1} u^{w}=\left\{\begin{array}{cl}
q-1 &\hbox{if } {q-1}|w,\\
0&\hbox{otherwise}. 
\end{array}
\right.$$
So we have 
\begin{eqnarray*}
&& \sum_{u_i^{q-1}=1} \Big (t_1^{\gamma_1(1-q)}\cdots t_n^{\gamma_n(1-q)}\exp\big(\pi F(\mathbf a,\mathbf t)- \pi F(\mathbf a,\mathbf t^{q})\big)
\Big)|_{t_i=u_i}\\
&=&\sum_{\mathbf w\in\mathbb Z^n\cap\delta} \sum_{u_i^{q-1}=1}  c_{\mathbf w} u_1^{w_1}\cdots u_n^{w_n}\\
&=&(q-1)^n \sum_{\mathbf u\in \mathbb Z^n\cap \delta}c_{(q-1)\mathbf u}.
\end{eqnarray*}
We thus get $$(q-1)^n \mathrm{Tr}(G_{\mathbf a},  L^\dagger_0)= 
\sum_{u_i^{q-1}=1} \left(t_1^{\gamma_1(1-q)}\cdots t_n^{\gamma_n(1-q)}\exp\big(\pi F(\mathbf a,\mathbf t)- \pi F(\mathbf a,\mathbf t^{q})\big)
\right)|_{t_i=u_i}.$$ This proves the 
theorem for $m=1$. 
We have 
\begin{eqnarray*}
G_{\mathbf a}^m
&=&\Big( \exp(\pi F(\mathbf a,\mathbf t))\Big)^{-1} \circ \Psi_{\mathbf a}^m\circ 
\Big( \exp(\pi F(\mathbf a,\mathbf t))\Big)\\
&=&\Psi_{\mathbf a}^m\circ\exp\big(\pi F(\mathbf a, \mathbf t)- \pi F(\mathbf a,\mathbf t^{q^m})\big).
\end{eqnarray*}
So the assertion for general $m$ follows from the case $m=1$. 
\end{proof}

\subsection{Proof of Theorem \ref{arithmetic}} By the equation (\ref{addagain}) and
the Dwork trace formula \ref{dwork1}, we have 
\begin{eqnarray*}
S_m(F(\bar{\mathbf a},\mathbf t))&=&(q^m-1)^n \mathrm{Tr}(G_{\mathbf a}^m,  {L^\dagger_0})\\
&=& \sum_{k=0}^n {n\choose k} (-1)^k (q^m)^{n-k} \mathrm{Tr}(G_{\mathbf a}^m,  L^\dagger_0)\\
&=& \sum_{k=0}^n (-1)^k \mathrm{Tr}\Big((q^{n-k}G_{\mathbf a})^m,  L_0^{\dagger{n\choose k}} \Big). 
\end{eqnarray*}
For the $L$-function, we have 
\begin{eqnarray*}
L(F(\bar{\mathbf a},\mathbf t),T)&=&\exp\Big(\sum_{m=1}^\infty S_m(F(\bar{\mathbf a},\mathbf t))\frac{T^m}{m}\Big)\\
&=&\exp\Big(\sum_{m=1}^\infty \sum_{k=0}^n (-1)^k\mathrm{Tr}\Big((q^{n-k}G_{\mathbf a})^m,  L_0^{\dagger{n\choose k}}\Big)\frac{T^m}{m}\Big)\\
&=&\prod_{k=0}^n \exp\Big((-1)^k\sum_{m=1}^\infty \mathrm{Tr}\Big((q^{n-k}G_{\mathbf a})^m, L_0^{\dagger{n\choose k}}\Big)\frac{T^m}{m}\Big)\\
&=& \prod_{k=0}^n \mathrm{det}\Big(I-T q^{n-k}G_{\mathbf a}, L_0^{\dagger{n\choose k}}\Big)^{(-1)^{k+1}}
\end{eqnarray*}
This proves Theorem \ref{arithmetic}. 

\section{Over-holonomic $D^\dagger$-Modules and $F$-isocrystals}

\subsection{} 
In this section, we take $\delta$ to be the cone generated by $\mathbf w_1, \ldots, \mathbf w_N$.
We first compare our twisted logarithmic de Rham complex with the rigid cohomology of a certain isocrystal.
As in the proof of Lemma \ref{LL}, we decompose $\Delta$ into a finite union $\bigcup_\tau \tau$
so that each $\tau$ is a  
simplex of dimension $n$ with vertices $\{0,\mathbf w_{i_1},\ldots, 
\mathbf w_{i_n}\}$ for some subset $\{i_1,\ldots, i_n\}\subset \{1,\ldots, N\}$. For each $\tau$, let $\delta(\tau)$ be the cone 
generated by $\tau$, and let 
\begin{eqnarray*}
B(\tau)&=&\mathbb Z^n\cap \{c_1\mathbf w_{i_1}+\cdots +c_n\mathbf w_{i_n}:\; 0\leq c_i< 1\},\\
C(\tau)&=& \{k_1\mathbf w_{i_1}+\cdots +k_n\mathbf w_{i_n}:\; k_i\in\mathbb Z_{\geq 0}\}.
\end{eqnarray*}
$B(\tau)$ is finite. Every element $\mathbf w\in \mathbb Z^n\cap \delta(\tau)$ can be written uniquely as 
$\mathbf w=b(\mathbf w)+c(\mathbf w)$ with $b(\mathbf w)\in B(\tau)$ and $c(\mathbf w)\in C(\tau)$.
We have $$\mathbb Z^n\cap \delta=\bigcup_{\tau} (\mathbb Z^n\cap\delta(\tau))=\bigcup_{\tau} \bigcup_{\mathbf w\in B(\tau)}(\mathbf w+C(\tau)).$$
The finite set $B=\{\mathbf w_1, \ldots, \mathbf w_N\}\bigcup\Big (\bigcup_\tau B(\tau)\Big)$ 
generates the monoid $\mathbb Z^n\cap\delta$. 
Denote elements in $B$ by $\mathbf b_1, \ldots, \mathbf b_M$ with $\mathbf b_i=\mathbf w_i$ for $1\leq i\leq N$. 
Let $\Sigma$ be the fan of all faces of dual cone $\check{\delta}$ of $\delta$, and let
$$X_R=X_R(\Sigma)=\text{Spec}\, R[\mathbb Z^n \cap \delta]$$ be the affine toric $R$-scheme defined by $\Sigma$.
We have a closed immersion 
$X_R\to\mathbb A^M_R$ defined by the $R$-algebra homomorphism 
$$R[y_1, \ldots, y_M]\to R[\mathbb Z^n\cap \delta],\quad y_i\mapsto \mathbf t^{\mathbf b_i}.$$ 
Let $\overline{X}_R$ be the closure of $X_R$ in $\mathbb P^M_R$, let $\widehat{\overline{X}}_R$ be the formal completion of $\overline{X}_R$
with respect to the adic topology defined by the maximal ideal of $R$, and let $X_\kappa=X_R\otimes_{R} \kappa$ (resp. $\overline{X}_\kappa
=\overline{X}_R\otimes_{R} \kappa$) be the special fiber of $X_R$ (resp. $\overline{X}_R$). 
The analytification $\overline{X}_K^{an}$ of the generic fiber of $\overline X_R$ coincides with the rigid analytic space associate 
to the formal scheme $\widehat{\overline{X}}_R$. We have a specialization map 
$\mathrm{sp}: \overline X_K^{an}\to \widehat{\overline{X}}_R$. The tube $]X_\kappa[$ is defined to be 
$]X_\kappa[=\mathrm{sp}^{-1}(X_\kappa)$. We have a commutative diagram
$$ \xymatrix{
 ]X_\kappa[\ar[r]\ar[d] & \overline{X}_K^{an}\ar[r]^{=} \ar[d] & \overline{X}_K^{an}\ar[d]^{\text{sp}} \\
X_\kappa \ar[r] & \overline{X}_\kappa \ar[r]  & \widehat{\overline{X}}_R.}$$
The analytification $X_K^{an}$ of the generic fiber of $X_R$ is a strict neighborhood of $]X_\kappa[$ in $\overline{X}_K^{an}$. 
Denote by $j: ]X_\kappa[ \to X_K^{an}$ the inclusion.

\begin{lemma}\label{global}
We have an isomorphism $\Gamma(X_K^{an},  j^\dagger \mathcal O_{X_K^{an}}) \xrightarrow{\sim} L_0^\dagger$.
\end{lemma}

\begin{proof} By \cite[1.2.4 (iii)]{B3}, 
$]X_\kappa[$ is the intersection of $X_K^{an}$ with the closed unit polydisc $D(0, 1^+)^M$ in $\mathbb A_K^{M, an}$, 
and the intersections $V_\lambda$ of $X_K^{an}$ with the closed polydiscs $D(0, \lambda^{+})^M$ of 
radii $\lambda$ form a fundamental system of strict neighborhoods 
when $\lambda \to 1^+$. We have
$$\Gamma(X_K^{an},  j^\dagger \mathcal O_{X_K^{an}})=\underset{\longrightarrow}{\text{lim}}_{\lambda\to1^+}\Gamma(V_\lambda, \mathcal O_{V_\lambda})
=K\langle y_1,\ldots,y_M\rangle^\dagger/IK\langle y_1,\ldots,y_M\rangle^\dagger,$$
where $I$ is the ideal of $K[y_1, \ldots, y_M]$ defining the closed subscheme $X_K$ of $\mathbb A_K^M$. 
Let 
\begin{eqnarray*}
A&=&\{\lambda=(\lambda_1, \ldots, \lambda_M)\in\mathbb Z^M:\, \sum_{i=1}^M \lambda_i \mathbf b_i=0\},\\
\tilde\Box_\lambda&=&\prod_{\lambda_i>0}y_i^{\lambda_i}-\prod_{\lambda_i<0} y_i^{\lambda_i} \quad (\lambda\in A), 
\end{eqnarray*} 
and let $J$ be the ideal of $K[y_1, \ldots, y_M]$ generated by $\tilde\Box_\lambda$ $(\lambda\in A)$. 
We have $J\subset I$. Choose finitely many $\mu_1,\ldots,\mu_m\in A$ so that 
for any $\lambda\in A$, we have $\tilde\Box_\lambda=a_1\tilde \Box_{\mu_1}+\cdots+a_m \tilde\Box_{\mu_m}$, where $a_i$ 
are polynomials with integer coefficients, and the total degrees of $a_i$ do not exceed that of $\tilde \Box_\lambda$. 
Such a choice is possible by Lemma \ref{annoy}. 
Consider the homomorphism 
$$\phi: K\langle y_1,\ldots,y_M\rangle^\dagger\to L_0^\dagger,\quad  y_i\mapsto \mathbf t^{\mathbf b_i}.$$
It is well-defined. In fact, 
given $\sum_{\mathbf v}a_{\mathbf v}\mathbf y^{\mathbf v}\in K\langle y_1,\ldots,y_M\rangle^\dagger$ such that $\vert a_{\mathbf v}\vert s^{|\mathbf v|}$ are bounded
for some $s>1$, 
$\sum_{\mathbf v}a_{\mathbf v}\mathbf y^{\mathbf v}$ is mapped to $\sum_{\mathbf u}(\sum_{\sum_i v_i\mathbf b_i =\mathbf u} a_{\mathbf v})\mathbf t^{\mathbf u}.$
Let $c=\max_{1\leq i\leq M}\vert \mathbf b_i\vert.$  We have $\vert \mathbf u\vert \le c\vert \mathbf v\vert$ whenever $\mathbf u=\sum v_i \mathbf b_i$. Thus
$$\Big\vert\sum_{\sum_i  v_i\mathbf b_i=\mathbf u}a_{\mathbf v}\Big\vert s^{c^{-1}\vert\mathbf u\vert}\le \max_{\mathbf v}\vert a_{\mathbf v}\vert s^{\vert\mathbf v\vert}.$$
Suppose $\sum_{\mathbf v}a_{\mathbf v}\mathbf y^{\mathbf v}$ lies in $\mathrm{ker}\,\phi$. 
For any $\mathbf u$, let $S_{\mathbf u}=\{\mathbf v\in \mathbb Z_{\geq 0}^M :  \sum_i  v_i \mathbf b_i=\mathbf u\}$. Then we have
$\sum_{\mathbf v\in \mathbf S_{\mathbf u}} a_{\mathbf v}=0$ for any $\mathbf u$. For any nonempty $S_{\mathbf u}$, 
choose $\mathbf v^{(0)}=(v^{(0)}_1,\ldots,v^{(0)}_M)\in S_{\mathbf u}$ such that $\vert\mathbf v^{(0)}\vert$ is minimal in $S_{\mathbf u}$. 
Then $\mathbf v-\mathbf v^{(0)}\in A$. We have
$$\mathbf y^{\mathbf v}-\mathbf y^{\mathbf v^{(0)}}=\mathbf y^{\min(\mathbf v,\mathbf v^{(0)})}
\Big(\prod_{v_j>v_j^{(0)}} y_j^{v_j-v_j^{(0)}}-\prod_{v_j<v_j^{(0)}} y_j^{v_j^{(0)}-v_j}\Big)=\mathbf y^{\min(\mathbf v,\mathbf v^{(0)})}\tilde \Box_{\mathbf v-\mathbf v^{(0)}}\in J.$$
As in the proof of Proposition \ref{grobner}, we can write
$$\mathbf y^{\mathbf v}-\mathbf y^{\mathbf v^{(0)}}=\Big(\sum_{\mathbf w}a_{\mathbf{vw}}^{(1)}\mathbf y^{\mathbf w}\Big)\tilde \Box_{\mu_1}+\cdots
+\Big(\sum_{\mathbf w}a_{\mathbf{vw}}^{(s)}\mathbf y^{\mathbf w}\Big)\tilde \Box_{\mu_m}$$
for some $a^{(k)}_{\mathbf{vw}}\in \mathbb Z$ such that $a^{(k)}_{\mathbf{vw}}=0$ if $|\mathbf w|>|\mathbf v|$. We have
\begin{eqnarray*}
\sum_{\mathbf v}a_{\mathbf v}\mathbf y^{\mathbf v}&=&\sum_{\mathbf u\in \mathbb Z^n\cap\delta}\sum_{\mathbf v\in S_{\mathbf u}}a_{\mathbf v}\mathbf y^{\mathbf v}\\
&=&\sum_{\mathbf u\in \mathbb Z^n\cap\delta}\sum_{\mathbf v\in S_{\mathbf u}}a_{\mathbf v}(\mathbf y^{\mathbf v}-\mathbf y^{\mathbf v^{(0)}})\\
&=&\sum_{k=1}^m\big(\sum_{\mathbf w}\sum_{\mathbf v}a_{\mathbf{vw}}^{(k)}a_{\mathbf v}\mathbf y^{\mathbf w}\big)\tilde \Box_{\mu_k}
\end{eqnarray*}
As $a_{\mathbf{vw}}^{(k)}\in \mathbb Z$ and $\vert a_{\mathbf v}\vert s^{|\mathbf v|}$ are bounded for some $s>1$, 
$\sum_{\mathbf v}a_{\mathbf{vw}}^{(k)}a_{\mathbf v}$ converges in $R$. 
We have
\begin{eqnarray*}
\Big\vert \sum_{\mathbf v}a_{\mathbf{vw}}^{(k)}a_{\mathbf v}\Big\vert s^{|\mathbf w|}&\le& 
\max_{\mathbf v,\; |\mathbf w|\le |\mathbf v|}\vert a_{\mathbf{vw}}^{(k)}a_{\mathbf v}\vert s^{|\mathbf w|}\\
&\le&\max_{\mathbf v}\vert a_{\mathbf v} \vert s^{|\mathbf v|}.
\end{eqnarray*}
So 
$\sum_{\mathbf w}\sum_{\mathbf v}a_{\mathbf{vw}}^{(k)}a_{\mathbf v}\mathbf y^{\mathbf w} \in K\langle y_1,\ldots,y_M\rangle^\dagger,$ 
and $\sum_{\mathbf v}a_{\mathbf v}\mathbf y^{\mathbf v}\in JK\langle y_1,\ldots,y_M\rangle^\dagger $.
Hence $$\mathrm{ker}\,\phi\subset JK\langle y_1,\ldots,y_M\rangle^\dagger \subset IK\langle y_1,\ldots,y_M\rangle^\dagger.$$ 
It is clear that  $IK\langle y_1,\ldots,y_M\rangle^\dagger\subset 
\mathrm{ker}\phi$. We thus have $\mathrm{ker}\,\phi=IK\langle y_1,\ldots,y_M\rangle^\dagger.$ 

Let's prove $\phi$ is surjective. It suffices to show $\sum_{\mathbf w\in \mathbf b_i +C(\tau)} a_{\mathbf w}\mathbf t^{\mathbf w}\in L_0^\dagger$
lies in the image of $\phi$, where $\tau$ is a simplex with vertices $\{0, w_{i_1},\ldots, w_{i_n}\}$ and $\mathbf b_i \in B(\tau)$. A preimage for 
$\sum_{\mathbf w\in \mathbf b_i +C(\tau)} a_{\mathbf w}(\mathbf x)\mathbf t^{\mathbf w}$ is 
\begin{eqnarray*}
\sum_{v_1, \ldots, v_n\geq 0} a_{\mathbf b_i+v_1 \mathbf w_{i_1}+\cdots +v_n\mathbf w_{i_n}} 
y_i y_{i_1}^{v_{1}}\cdots y_{i_n}^{v_{n}}.
\end{eqnarray*} Let's prove it lies in $K\langle \mathbf y\rangle^\dagger$. 
Suppose $\vert a_{\mathbf w}\vert s^{\vert \mathbf w\vert}$ are bounded for some $s>1$. Since 
$\mathbf w_{i_1},\ldots, \mathbf w_{i_n}$ are linearly independent,  
there exist $c_1,  c_2>0$ such that $$c_1|\mathbf v|\le |v_1\mathbf w_{i_1}+\cdots+v_n\mathbf w_{i_n}|\le c_2|\mathbf v|$$
for any $\mathbf v=(v_1,  \ldots, v_n)$. 
We have 
\begin{eqnarray*}
&&\vert a_{\mathbf b_i+v_1 \mathbf w_{i_1}+\cdots +v_n\mathbf w_{i_n}} \vert \cdot (s^{c_1}) ^{1+ v_1+\cdots+v_n}\\
&\leq& \vert a_{\mathbf b_i+v_1 \mathbf w_{i_1}+\cdots +v_n\mathbf w_{i_n}} \vert\cdot s^{c_1+ \vert v_1 \mathbf w_{i_1}+\cdots +v_n\mathbf w_{i_n}\vert}\\
&\leq& \vert a_{\mathbf b_i+v_1 \mathbf w_{i_1}+\cdots +v_n\mathbf w_{i_n}} \vert \cdot s^{\vert \mathbf b_i+v_1 \mathbf w_{i_1}+\cdots +v_n\mathbf w_{i_n}\vert}
s^{c_1+\vert -\mathbf b_i\vert}.
\end{eqnarray*}
Thus $\vert a_{\mathbf b_i+v_1 \mathbf w_{i_1}+\cdots +v_n\mathbf w_{i_n}} \vert\cdot (s ^{c_1})^{1+ v_1+\cdots+v_n}$ are bounded and 
$\sum_{v_1, \ldots, v_n\geq 0} a_{\mathbf b_i+v_1 \mathbf w_{i_1}+\cdots +v_n\mathbf w_{i_n}} 
y_i y_{i_1}^{v_{1}}\cdots y_{i_n}^{v_{n}}$ lies in $K\langle \mathbf y\rangle^\dagger$. 
Therefore $\phi$ induces an isomorphism $\Gamma(X_K^{an},  j^\dagger \mathcal O_{X_K^{an}}) \cong L_0^\dagger$.
\end{proof}

Let $\Sigma'$ be a regular refinement of $\Sigma$, $X'_R$ the toric $R$-scheme associate to $\Sigma'$,  and
$\overline X'_R$ a compactification of $X'_R$. Replacing $\overline X'_R$ by the scheme theoretic image of the immersion 
$X'_R\to \overline X_R\times_R \overline X'_R,$ we may assume that we have a commutative diagram 
$$\begin{array}{ccc}
X'_R&\to& \overline X'_R\\
\downarrow&&\downarrow \\
X_R&\to&\overline X_R.
\end{array}$$
It gives rise to a commutative diagram
$$\begin{array}{ccccc} 
X'_R&\to& X_R\times_{\overline{X}_R}\overline X'_R&\to& \overline X'_R\\
&\searrow&\downarrow&&\downarrow\\
&& X_R&\to& \overline X_R.
\end{array}$$
The morphism $X'_R\to X_R\times_{\overline{X}_R}\overline X'_R$ is an open immersion since 
$X'_R\to \overline X'_R$ and the projection $X_R\times_{\overline{X}_R}\overline X'_R\to \overline X'_R$ are open immersions. 
It is proper since $X'_R\to X_R$ is proper and the projection $X_R\times_{\overline{X}_R}\overline X'_R\to X_R$ is separated. Since $X'_R$ is dense in 
$\overline X'_R $, 
we have $$X'_R\cong X_R\times_{\overline{X}_R}\overline{X}'_R.$$ 
Consider the Cartesian squares
\begin{eqnarray}\label{diagram}
\begin{array}{ccccc} 
X'_\kappa&\to&\overline X'_\kappa&\to &\widehat{\overline X}'_R\\
\downarrow&&\downarrow&&\downarrow\\
X_\kappa&\to&\overline{X}_\kappa&\to &\widehat{\overline{X}}_R,
\end{array}
\end{eqnarray}
where $\overline X'_\kappa$ (resp. $X'_\kappa$) is the special fibre of $\overline X'_R$
(resp. $X'_R$), and $\widehat{\overline X}'_R$ is the formal completion of $\overline X'_R.$ 
Let $]X'_\kappa[$ be the tube of $X'_\kappa$ in the analytification ${\overline X}'^{an}_K$ of generic fibre $\overline X'_K$ of $\overline X'$,
let $j':]X'_\kappa[\to X'^{an}_K$ be the inclusion, and let $\phi: X_K^{\prime an}\to X_K^{an}$ be the canonical morphism. 
We have a commutative diagram
\begin{eqnarray}\label{dia}
\begin{array}{ccccc}
]X'_\kappa[&\stackrel{j'}\to &X'^{an}_K&\to &\overline X'^{an}_K\\
\downarrow&&\downarrow{\scriptstyle \phi}&&\downarrow\\
]X_\kappa[&\stackrel j\to& X^{an}_K&\to &\overline{X}^{an}_K.
\end{array}
\end{eqnarray}

\begin{lemma}\label{comparison}
We have 
\begin{eqnarray*}
R^k\phi_\ast {j'}^\dagger \mathcal O_{X'^{an}_K}&\cong& \left\{\begin{array}{cl}
j^\dagger \mathcal O_{X_K^{an}} &\hbox{if }k=0,\\
0&\hbox{if }k\geq 1,
\end{array}\right. \\
H^k(X_K^{an}, j^\dagger \mathcal O_{X_K^{an}})
&\cong&
\left\{\begin{array}{cl}
L_0^\dagger &\hbox{if }k=0,\\
0&\hbox{if }k\geq 1.
\end{array}\right.
\end{eqnarray*}
\end{lemma}

\begin{proof}
For any $\lambda>1$, let $V_\lambda$ be the intersection of $X_K^{an}$ with the closed polydisc $D(0, \lambda^+)^M$ in
 $\mathbb P_K^{M, an}$. Then $V_\lambda$ $(\lambda>1)$ form a fundamental system of affinoid strict neighborhoods of $]X_\kappa[$ in $X^{an}_K$
 and $j_\lambda: V_\lambda\to X^{an}_K$ are affine morphisms. Since 
squares in (\ref{diagram}) are Cartesian, 
 $V'_\lambda=\phi^{-1}(V_\lambda)$ $(\lambda>1)$ form a fundamental system of strict neighborhoods of $]X'_\kappa[$ in $X'^{an}_K$, and 
 $j'_\lambda: V'_\lambda\to X'^{an}_K$ are affine. Consider the  commutative diagram
$$\begin{array}{cclcc}
]X'_\kappa[&\to& V'_\lambda&\stackrel{j'_\lambda}\to& X'^{an}_K\\
\downarrow&&\downarrow{\scriptstyle \phi_\lambda}&&\downarrow{\scriptstyle \phi}\\
]X_\kappa[ &\to&  V_\lambda&\stackrel {j_\lambda}\to & X^{an}_K.
\end{array}$$
We have 
\begin{eqnarray}
R^k\phi_\ast(j^\dagger \mathcal O_{X_K^{'an}})&=&R^k\phi_\ast \varinjlim\nolimits_{\lambda\to 1^+}j'_{\lambda \ast}\mathcal O_{V'_\lambda}\nonumber\\
\label{1}&\cong& \varinjlim\nolimits_{\lambda\to 1^+}R^k\phi_\ast j'_{\lambda \ast}\mathcal O_{V'_\lambda} \\
&\cong& \varinjlim\nolimits_{\lambda\to 1^+}j_{\lambda \ast} R^k \phi_{\lambda \ast}O_{V'_\lambda}\label{2}\\
&\cong& \varinjlim\nolimits_{\lambda\to 1^+}j_{\lambda \ast} j_\lambda^\ast R^k \phi_\ast \mathcal O_{X'^{an}_K}\nonumber\\
&\cong&  \varinjlim\nolimits_{\lambda\to 1^+}j_{\lambda \ast} j_\lambda^\ast (R^k \phi_{K \ast} \mathcal O_{X'_K})^{an}\label{3}\\
&=& j^\dagger(R^k \phi_{K \ast} \mathcal O_{X'_K})^{an}\nonumber.
\end{eqnarray}
Here (\ref{1}) follows from that $\phi$ is quasi-compact and quasi-separated. (\ref{2}) follows from that $j'_\lambda$ and $j_\lambda$ are affine. 
(\ref{3}) follows from \cite[Proposition 3.4.9]{Berk}, where $\phi_K: X'_K\to X_K$ is the canonical morphism.  
By \cite[Theorem 9.2.5]{CLS}, we have $R^k\phi_{K \ast} \mathcal O_{X'_K}=0$ for $k\geq1$ and 
$\phi_{K \ast} \mathcal O_{X'_K}=\mathcal O_{X_K}.$ The first assertion of the lemma follows.
We have 
\begin{eqnarray*}
H^k(X_K^{an}, j^\dagger \mathcal O_{X_K^{an}})
&\cong& H^k(X_K^{an},  \varinjlim\nolimits_{\lambda\to 1^+}j_{\lambda \ast} \mathcal O_{V_\lambda})\\
&\cong& \varinjlim\nolimits_{\lambda\to 1^+} H^k(X_K^{an},  Rj_{\lambda \ast} \mathcal O_{V_\lambda})\\
&\cong&  \varinjlim\nolimits_{\lambda\to 1^+} H^k(V_\lambda, \mathcal O_{V_\lambda}).
\end{eqnarray*}
For $k\geq 1$, we thus get $H^k(X_K^{an}, j^\dagger \mathcal O_{X_K^{an}})=0$. The case $k=0$ 
is given by Lemma \ref{global}.
\end{proof}

Regard $K\langle y \rangle^\dagger$ as a free $K\langle y \rangle^\dagger$-module with 
the connection $$\nabla(1)=\pi \mathrm dy.$$ It defines the \emph{Dwork isocrystal}, which is 
$j_0^\dagger \mathcal O_{\mathbb A_K^{1, an}}$ considered as a free $j_0^\dagger \mathcal O_{\mathbb A_K^{1, an}}$-module 
with the connection $\nabla$, where $j_0: D(0, 1^+)^N\to \mathbb A_K^{1,an}$ is the inclusion. 
Since $F(\mathbf a, \mathbf t)$ lies in  $R[\mathbb Z^n\cap \delta]$, it defines a morphism $X_R=\mathrm{Spec}\, R[\mathbb Z^n\cap \delta]
\to \mathbb A^1_R$. Denote its composite with 
$X'_R\to X_R$ by $F_{\mathbf a}: X'_R\to \mathbb A^1_R$. The base changes of $F_{\mathbf a}$ from $R$ to $K$ and to $\kappa$ are
also denoted by $F_{\mathbf a}$. Let $\mathcal L_{\psi,F_{\mathbf a}}$ be the 
pull-back of $\mathcal L_{\psi}$ by the morphism $F_{\mathbf a}:X_K^{\prime an}\to \mathbb A_K^{1, an}$. 
Then $\mathcal L_{\psi,F_{\mathbf a}}$ is ${j'}^{\dagger} \mathcal O_{X_K^{' an}}$ considered as a free
${j'}^{\dagger} \mathcal O_{X_K^{' an}}$-module with the connection $$\nabla(1)=e^{-\pi F(\mathbf a, \mathbf t)}\circ 
\mathrm d_{\mathbf t} \circ e^{\pi F(\mathbf a, \mathbf t)}(1)
=\sum_{i=1}^n \Big(\sum_{j=1}^N \pi w_{ij}a_j\mathbf t^{\mathbf w_j}\Big)
\frac{\mathrm dt_i}{t_i}.$$ 
Let $\mathbb T^n_K$ be the maximal torus in the toric scheme $X'_K$, and let
$D'=X'_K-\mathbb T_K^n$. As $X'_K$ is a smooth toric scheme and $D'$ is a 
normal crossing divisor, the sheaf ${j'}^\dagger\Omega_{X'^{an}_K}^{k}(\log D')$ 
of forms with logarithmic poles along $D'$ is a free ${j'}^\dagger \mathcal O_{X'^{an}_K}$-module 
with basis $$\frac{\mathrm dt_{i_1}}{t_{i_1}}\wedge\cdots\cdot \wedge \frac{\mathrm dt_{i_k}}{t_{i_k}}\quad (1\leq i_1<\cdots<i_k\leq n)$$ for each $k$. 

\begin{lemma}\label{Tn} We have
$$R\Gamma(X'^{an}_K, j'^\dagger \Omega_{X'^{an}_K}^{\cdot}(\log D')\otimes_{j'^\dagger \mathcal O_{X'^{an}_K}} 
\mathcal L_{\psi,F_{\mathbf a}}) \cong R\Gamma_{rig}(\mathbb T^n_\kappa, \mathcal L_{\psi,F_{\mathbf a}\circ i}),$$
where on the righthand side of the isomorphism, $\mathcal L_{\psi,F_{\mathbf a}\circ i}$ is the pull-back of the Dwork isocrystal 
$\mathcal L_\psi$ by the composite $\mathbb T^n_\kappa\stackrel{i}{\to}X'_\kappa\stackrel{F_{\mathbf a}}{\to}\mathbb A^1_\kappa.$
\end{lemma}

\begin{proof}
Choose a regular fan $\Sigma''$ containing $\Sigma'$ such that $|\Sigma''|=\mathbb R^n.$  The toric 
$R$-scheme $X''_R$ associated to $\Sigma''$ is proper over $R$. Let $\widehat X''_R$ (resp.  $X''_\kappa$ and $X''_K$)
be its formal completion (resp. special fiber, resp. generic fibre),  and let $D''=X''_R-\mathbb T_R^n$. 
Then $D''$ is a divisor with normal crossing in $X''_R$. 
We have open immersions $$\mathbb T^n_\kappa \hookrightarrow X'_\kappa \hookrightarrow  X''_\kappa.$$
Note that $X'^{an}_K$ is a strict neighborhood of $]X'_k[$ in $X''^{an}_K$. 
In \cite[Corollary A.4]{B4}, taking $\mathcal X=\widehat X''_R$, $Z=D''$, $U=\mathbb T^n_\kappa$, $V=X'_\kappa$ and $W=X_K^{\prime an}$, we get
\begin{eqnarray*}
R\Gamma(X'^{an}_K, j'^\dagger \Omega_{X'^{an}_K}^{\cdot}(\text{log}D')\otimes_{j'^\dagger \mathcal O_{X'^{an}_K}} \mathcal L_{\psi,F_{\mathbf a}})
\cong R\Gamma_{\text{rig}}(\mathbb T_\kappa^n, \mathcal L_{\psi,F_{\mathbf a}\circ i}).
\end{eqnarray*}
\end{proof}

\begin{remark} In order to apply \cite[Corollary A.4]{B4}, we require the toric schemes $X'$ and $X''$ to be smooth
and $X''$ proper. So we work with the regular fans $\Sigma'$ and $\Sigma''$. The result 
in Lemma \ref{Tn} is crucial in the next subsection \ref{whyregularrefinement} where we show 
$$R\Gamma_{\mathrm{rig}}(\mathbb T_\kappa^n, \mathcal L_{\psi, F_{\mathbf a}\circ i})\cong C^\cdot(L_0^\dagger).$$
This isomorphism couldn't be obtained by simply working with the fan $\Sigma$. 
\end{remark}

\subsection{Proof of Proposition \ref{newfiber} (i)}\label{whyregularrefinement}
Combining Lemmas \ref{comparison} and \ref{Tn} and the fact that  ${j'}^\dagger\Omega_{X'^{an}_K}^{k}(\log D')$ 
is a free ${j'}^\dagger \mathcal O_{X_K^{\prime an}}$-module, we get
\begin{eqnarray*}
R\Gamma_{\text{rig}}(\mathbb T^n_\kappa, \mathcal L_{\psi,F_{\mathbf a}\circ i})&\cong&R\Gamma\Big(X^{an}_K, R\phi_{\ast}\Big(j'^\dagger \Omega_{X'^{an}_K}^{\cdot}
(\log D')\otimes_{j'^\dagger \mathcal O_{X'^{an}_K}} \mathcal L_{\psi,F_{\mathbf a}}\Big)\Big)\\
&\cong&\Gamma\Big(X^{an}_K, \phi_{\ast}\Big(j'^\dagger \Omega_{X'^{an}_K}^{\cdot}
(\log D')\otimes_{j'^\dagger \mathcal O_{X'^{an}_K}} \mathcal L_{\psi,F_{\mathbf a}}\Big)\Big)\\
&\cong&C^\cdot(L_0^\dagger).
\end{eqnarray*}
In the case where $\gamma\not=0$,
we have $\delta=\mathbb R^n$ by our assumption.  Let $\mathcal K_{\chi_i}$  be Kummer isocrystal, that is, 
$K\langle y,y^{-1}\rangle$ considered as a $K\langle y,y^{-1}\rangle$-module with the connection  
$$\nabla_{i}(1)=\gamma_i\frac{dy}{y}.$$ 
Let $p_i:\mathbb T^n_k\to \mathbb T^1_k$ be the projection to the $i$-th factor. 
Define $\mathcal K_{\chi}=\otimes_i p_i^\ast \mathcal K_{\chi_i}$. The same argument as above shows that 
$$R\Gamma_{\text{rig}}(\mathbb T^n_k,\mathcal K_{\chi}\otimes \mathcal L_{\psi, F_{\mathbf a}})\cong C^\cdot(L_0^\dagger).$$
Rigid cohomology groups of the tensor product of pull-backs of the Kummer isocrystal and the Dwork isocrystal are finite dimensional by \cite[3.10]{B6}. 
So $H^k(C^\cdot(L_0^\dagger))$ are finite dimensional for arbitrary $\mathbf a=(a_1, \ldots, a_N)\in D(0,1^+)^N.$

\medskip
Next we prove the second part of Proposition \ref{newfiber} (i) using the method in \cite[Corollary 1.3]{Bour}.
For any $\mathbf w\in \mathbb Z^n\cap \delta$, define 
$$\omega(\mathbf w)=\inf\{r: \mathbf w\in r\Delta\}.$$
One can show that there exist positive real numbers $C_1$ and $C_2$ such that for any $\mathbf w\in \mathbb Z^n\cap \delta$, we have
$$C_1|\mathbf w|\le \omega(\mathbf w)\le C_2|\mathbf w|.$$
For any real numbers $b>0$ and $c$, define
\begin{eqnarray*}
\mathfrak L(b,c)&=&
\bigcup_c\{\sum_{\mathbf w\in \mathbb Z^n\cap \delta}a_\mathbf w\mathbf t^\mathbf w: \, \mathrm{ord}_p(a_{\mathbf w})\geq b\omega(\mathbf w)+c\},\\
\mathfrak L(b)&=&\bigcup_c\mathfrak L(b,c),\\
\mathfrak B&=&\{\sum_{\mathbf w\in \mathbb Z^n\cap \delta}a_\mathbf w\pi^{\omega(\mathbf w)}\mathbf t^\mathbf w \big| \; a_\mathbf w\to 0 \text{ as } |\mathbf w|\to \infty\}.
\end{eqnarray*}
We have $L_0^\dagger=\bigcup_{b>0} \mathfrak L(b)$. 
Since $\mathrm{ord}_p(\pi)=\frac{1}{p-1}$, we have 
$$\mathfrak B\subset \mathfrak L\Big(\frac{1}{p-1}\Big), \quad \mathfrak L(b)\subset\mathfrak B\hbox { if } b>\frac{1}{p-1}.$$
Let $H(\mathbf t)=\mathbf t^{(1-q)\gamma}\exp\big(\pi F(\mathbf a,\mathbf t)- \pi F(\mathbf a^q, \mathbf t^{q})\big)$. 
By the proof of Proposition \ref{estimation} (i) in 1.1, we have
$H(\mathbf t)\in \mathfrak L(\frac{p-1}{pq}).$ The operator 
$G_{\mathbf a}=\Psi_{\mathbf a}\circ H(\mathbf t)$
can be considered as the composite  
$$\mathfrak L(b)\stackrel{H(\mathbf t)}\to 
\mathfrak L\Big(\min\Big(b,\frac{p-1}{pq}\Big)\Big)\stackrel{\Psi_{\mathbf a}}\to \mathfrak L\Big(\min\Big(qb,\frac{p-1}{p}\Big)\Big).$$
So for any $0<b\leq \frac{p-1}{p}$, 
$G_{\mathbf a}$ defines an endomorphism of $\mathfrak L(b)$, and for any 
$0<b\leq \frac{p-1}{pq}$, $G_{\mathbf a}$ induces a morphism 
$\mathfrak L(b)\to\mathfrak L(qb).$
In the case $p>2$, we have $\frac{p-1}{p}>\frac{1}{p-1}$. So $G_{\mathbf a}$ induces a morphism on $\mathfrak B$: 
$$\mathfrak B\hookrightarrow \mathfrak L\Big(\frac{1}{p-1}\Big)
\stackrel{G_{\mathbf a}}\to \mathfrak L\Big(\min\Big(\frac{q}{p-1},\frac{p-1}{p}\Big)\Big)\hookrightarrow \mathfrak B.$$
Using $\mathfrak L(b)$ and $\mathfrak B$ instead of $L_0^\dagger$, we can construct twisted 
logarithmic de Rham complexes $C^\cdot(\mathfrak L(b))$ ($0<b\leq \frac{p-1}{p}$)
and $C^\cdot(\mathfrak B)$. $G_{\mathbf a}$ acts on $C^\cdot(\mathfrak L(b))$ for all $0<b\leq\frac{p-1}{p}$.
$G_{\mathbf a}$ acts on $C^\cdot(\mathfrak B)$ if $p>2$. 

Consider the case where $p>2$. 
Suppose $k\not =n$. Since $H^k(C^\cdot(L_0^\dagger))$ is finite dimensional and $\bigcup_{b>0}\mathfrak L(b)=L_0^\dagger$, the canonical homomorphism
$i_0: H^k(C^\cdot(\mathfrak L(b)))\to H^k(C^\cdot(L_0^\dagger))$ induced by the inclusion 
$C^\cdot(\mathfrak L(b))\hookrightarrow C^\cdot(L_0^\dagger)$
is surjective for sufficiently small $b>0$. Choose a large positive integer $l$ so that $i_0$ is surjective for any 
$b\leq \frac{p-1}{pq^l}$. Choose $b$ so that $\frac{1}{q^l(p-1)}< b\leq \frac{p-1}{pq^l}$. 
Then $G_{\mathbf a}^l$ induces a homomorphism
$H^k(C^\cdot(\mathfrak L(b)))\to H^k(C^\cdot(\mathfrak L(q^lb)))$ and 
$C^\cdot(\mathfrak L(q^l b))$ is a subcomplex of $C^\cdot(\mathfrak B)$.
We have a commutative diagram
$$\xymatrix{
H^k(C^\cdot(\mathcal L(b)))\ar[r]^{i_0} \ar[d]^{G^l_{\mathbf a}}&H^k(C^\cdot(L_0^\dagger))\ar[d]^{G^l_{\mathbf a}}\\
H^k(C^\cdot(\mathcal L(q^lb)))\ar[r]^{i_l}\ar[d]^i&H^k(C^\cdot(L_0^\dagger))\\
H^k(C^\cdot(\mathcal B))\ar[ur]^{\tilde{i}}   }.$$
Here $i_0,i_l,i, \tilde{i}$ are induced by inclusions of complexes. 
Since $i_0$ is surjective and $G_{\mathbf a}:H^k(C^\cdot(L_0^\dagger))\to 
H^k(C^\cdot(L_0^\dagger))$ is an isomorphism, $\tilde{i}$ is surjective. If $\mathbf a\in U$, then $H^k(C^\cdot(\mathfrak B))=0$ for 
$k\not=n$ by \cite[3.7]{A2}. So $H^k(C^\cdot(L_0^\dagger))=0$ for $k\not=n$. Then $\text{dim}\, H^n(C^\cdot(L_0^\dagger))$ is determined by 
the degree of $L$-function by Theorem \ref{arithmetic}. So we have 
$$\text{dim}\, H^n(C^\cdot(L_0^\dagger))=\text{dim}\, H^n(C^\cdot(\mathcal B))=n!\text{vol}(\Delta)$$ by \cite[3.8, 3.9]{A2}.
The proof for the case $p=2$ is similar. We replace $\pi$ satisfying $\pi+\frac{\pi^p}{p}=0$ by $\pi'$ satisfying
$\sum_{i=0}^\infty \frac{\pi'^{p^i}}{p^i}=0$, and replace the function $\exp(\pi(z-z^p))=\exp\Big(\pi z+ \frac{(\pi z)^p}{p}\Big)$
by $E(\pi' z)$, where $E=\exp\Big(\sum_{i=0}^\infty \frac{z^{p^i}}{p^i}\Big)$ is the Artin-Hasse exponential. These replacements 
give rise to a proof which works for all $p$. See \cite{LPG} for details. 

\subsection{Arithmetic $D$-modules}\label{d-module}

The reader may consult \cite[\S 2]{B4}, \cite{B} and \cite[\S 1]{Overhol} for details.
Let $f:\mathscr X'\to\mathscr X$ be a morphism of smooth formal $R$-schemes, $a$ a uniformizer of $R$, $f_i: X'_i\to X_i$ 
the reduction of $f$ modulo $a^{i+1}$, $f_0: X'_0\to X_0$ the induced morphism 
on special fibers, and $T$ (resp. $T'$) a divisor on $X_0$ (resp. $X'_0$) such that $f_0^{-1}(T)\subset T'$. Let $\mathscr U\subset \mathscr X$ 
be an affine open subset and let $h\in \Gamma(\mathscr U,\mathcal O_{\mathscr X})$ so that the divisor $T$ of $X_0$ is given 
by $h\equiv 0 \mod a$. Let $U_i$ and $h_i\in \Gamma(U_i,\mathcal O_{U_i})$ 
be the reduction of $\mathscr U$ and $h$ modulo $a^{i+1}$.  For any $m\geq 0$, define
\begin{eqnarray*}
&&\mathcal B_{X_i}^{(m)}(T)|_{U_i}=\mathcal O_{U_i}[t]/(h_i^{p^{m+1}}t-p),\\
&& \widehat{\mathcal B}_{\mathscr X}^{(m)}(T)|_{\mathscr U}
=\varprojlim_i\mathcal B_{X_i}^{(m)}(T)|_{U_i}=\mathcal O_{\mathscr U}\{t\}/(h^{p^{m+1}}t-p).
\end{eqnarray*}
$\widehat{\mathcal B}_{\mathscr X}^{(m)}(T)$ is a $p$-adically complete $\mathcal O_{\mathscr X}$-algebra 
depending only on $\mathscr X$ and $T$. Define the sheaf of functions on $\mathscr X$ with overconvergent singularities along $T$ by
$$\mathcal O_{\mathscr X, \mathbb Q}(^\dagger T)=\varinjlim_{m}\widehat{\mathcal B}_{\mathscr X}^{(m)}(T)\otimes_{\mathbb Z}\mathbb Q.$$
Denote by $\mathcal D_{X_i}^{(m)}$ the sheaf differential operators of level $m$ on $X_i$. If $x_1,\ldots, x_N$ are local coordinates on $U_i$ and 
$\partial_j=\partial/\partial x_j$ for $1\le j\le N$, we have
$$\mathcal D_{X_i}^{(m)}(U_i)=\{\sum_{\mathbf k}\mathbf{q}_{\mathbf k}^{(m)}! a_{\mathbf k}
\partial^{[\mathbf k]}:\;  a_{\mathbf k}\in \Gamma(U_i, \mathcal O_{X_i}), a_{\mathbf k}\not =0 \text{ for only finitely many } \mathbf k\},$$
where $\partial^{[\mathbf k]}=\frac{1}{\mathbf k!}\partial_1^{k_1}\ldots\partial_N^{k_N}$, and $\mathbf{q}_{\mathbf k}^{(m)}$ is determined by the 
expression $\mathbf k=p^m\mathbf{q}_{\mathbf k}^{(m)}
+\mathbf{r}_{\mathbf k}^{(m)}$ with $0\le \mathbf r_{\mathbf k, j}^{(m)}<p^m$ for all $j$.
The sheaf of differential operators of level $m$ on $\mathscr X$ is defined by
$$\widehat{\mathcal D}_{\mathscr X}^{(m)}=\varprojlim_i
\mathcal D_{X_i}^{(m)}, \quad \widehat{\mathcal D}_{\mathscr X, \mathbb Q}^{(m)}=\widehat{\mathcal D}_{\mathscr X}^{(m)}\otimes_{\mathbb Z} \mathbb Q.$$ We have
$$\widehat{\mathcal D}_{\mathscr X, \mathbb Q}^{(m)}(\mathscr U)
=\{\sum_{\mathbf k}\mathbf{q}_{\mathbf k}^{(m)}!a_{\mathbf k}\mathbf{\partial}^{[\mathbf k]}: \; a_{\mathbf k}\in 
\Gamma(\mathscr U, \mathcal O_{\mathscr X, \mathbb Q}), a_{\mathbf k}\to 0 \text{ as } \vert\mathbf{k}\vert\to \infty\},$$ 
where $\mathcal O_{\mathscr X, \mathbb Q}=\mathcal O_{\mathscr X}\otimes_{\mathbb Z}\mathbb Q$. 
Define the sheaf of differential operators overconvergent along $T$ as
$$\mathcal D_{\mathscr X, \mathbb Q}^\dagger(^\dagger T):=
\varinjlim_m\widehat{\mathcal D}_{\mathscr X, \mathbb Q}^{(m)}(T)
:=\varinjlim_m\widehat{\mathcal B}_{\mathscr X}^{(m)}(T)\widehat{\otimes}_{\mathcal O_{\mathscr X}}
\widehat{\mathcal D}_{\mathscr X, \mathbb Q}^{(m)}.$$ We have
$$\mathcal D_{\mathscr X, \mathbb Q}^\dagger(^\dagger T)(\mathscr U)
=\{\sum_{j, \mathbf k}\frac{a_{j, \mathbf k}}{h^{j+1}}\partial^{[\mathbf k]}:\; a_{j, \mathbf k}\in \Gamma(\mathscr U, \mathcal O_{\mathscr X, \mathbb Q}), 
\;  \vert a_{j, \mathbf k}\vert s^{j+|\mathbf k|}  \hbox { are bounded for some } s>1\}.$$
Define
$$\widehat{\mathcal D}_{\mathscr X'\to\mathscr X}^{(m)}(T', T):=
\widehat{\mathcal B}_{\mathscr X'}^{(m)}(T')\widehat{\otimes}_{f^{-1}\mathcal O_{\mathscr X}}f^{-1}\widehat{\mathcal D}_{\mathscr X}^{(m)}.$$
It is a  $(\widehat{\mathcal D}_{\mathscr X'}^{(m)}(T'), f^{-1}\widehat{\mathcal D}_{\mathscr X}^{(m)}(T))$-bimodule. Define
$$\mathcal D_{\mathscr X'\to\mathscr X, \mathbb Q}^{\dagger}(^\dagger T', T)
=\big(\varinjlim_m\widehat{\mathcal D}_{\mathscr X'\to\mathscr X}^{(m)}(T', T)\big)\otimes_{\mathbb Z}\mathbb Q.$$
It is a $(\mathcal D_{\mathscr X', \mathbb Q}^{\dagger}(^\dagger T'), f^{-1}\mathcal D_{\mathscr X, \mathbb Q}^{\dagger}(^\dagger T))$-bimodule.
For any object $\mathcal E$ in the derived category $D_{coh}^b(\mathcal D_{\mathscr X, \mathbb Q}^\dagger(^\dagger T))$ of
$\mathcal D_{\mathscr X, \mathbb Q}^\dagger(^\dagger T)$-modules with coherent cohomology, its extraordinary inverse image $f^{!}\mathcal E$ is defined by
$$f^!(\mathcal E):=\mathcal D_{\mathscr X'\to\mathscr X, \mathbb Q}^{\dagger}(^\dagger T', T)
\otimes_{f^{-1}\mathcal D_{\mathscr X, \mathbb Q}^\dagger(^\dagger T)}^Lf^{-1}\mathcal E[d_{\mathscr X'}-d_{\mathscr X}],$$
where $d_{\mathscr X}$ (resp. $d_{\mathscr X'}$) is the dimension of $\mathscr X$ (resp. $\mathscr X'$). 
Let $T_1$ be another divisor of $X_0$ containing $T$. We have a functor 
\begin{eqnarray*}
(^\dagger T_1, T): 
D_{\text{coh}}^b(\mathcal D_{\mathscr X, \mathbb Q}^\dagger(^\dagger T))\to D_{\text{coh}}^b
(\mathcal D_{\mathscr X, \mathbb Q}^\dagger(^\dagger T_1)),\quad
(^\dagger T_1, T)(\mathcal E):=\mathcal D_{\mathscr X, \mathbb Q}^\dagger
(^\dagger T_1)\otimes^L_{\mathcal D_{\mathscr X, \mathbb Q}^\dagger(^\dagger T)}\mathcal E.
\end{eqnarray*}
Here we could omit the symbol $L$ since the morphism 
$\mathcal D_{\mathscr X, \mathbb Q}^\dagger(^\dagger T)\to \mathcal D_{\mathscr X, \mathbb Q}^\dagger(^\dagger T_1)$ is flat. (cf. \cite[4.3.10-11]{B}.)

In the following, we work with $\mathscr X=\mathscr P^N$ and the divisor 
$\infty=\mathscr P^N-\mathscr A^N$. 
For simplicity, when there is no confusion on the divisors $T=\infty$ and $T'$, we use the notation 
$\mathcal O_{\mathscr X, \mathbb Q}(\infty)$, $\mathcal D_{\mathscr X, \mathbb Q}^{\dagger}(\infty)$,  $\mathcal D_{\mathscr X', \mathbb Q}^{\dagger}(\infty)$, 
$\widehat{\mathcal D}_{\mathscr X'\to\mathscr X}^{(m)}(\infty)$, $\mathcal D_{\mathscr X'\to\mathscr X, \mathbb Q}^{\dagger}(\infty)$ 
instead of $\mathcal O_{\mathscr X, \mathbb Q}(^\dagger T)$,
$\mathcal D_{\mathscr X, \mathbb Q}^{\dagger}(^\dagger T)$, $\mathcal D_{\mathscr X', \mathbb Q}^{\dagger}(^\dagger T')$, 
$\widehat{\mathcal D}_{\mathscr X'\to\mathscr X}^{(m)}(T',T),$ $\mathcal D_{\mathscr X'\to\mathscr X, \mathbb Q}^{\dagger}(^\dagger T',T)$ 
and use $(^\dagger T_1)$ instead of $(^\dagger T_1, T)$. 

\subsection{Proof of Propositions \ref{overholonomic} and \ref{newfiber} (ii)}\label{fibre1} \label{proof_of_fiber}
For a point $\mathbf a$ in the formal affine space 
$\mathscr A^N$, let $i_{\mathbf a}: \mathbf a\to \mathscr P^N$ be the 
closed immersion and let $\mathcal J$ be its ideal sheaf. By \cite[Th\'eor\`eme 2.1.4]{Caro1}, it suffices
to show $H^k\Big(i_{\mathbf a}^!(C^\cdot(\mathcal L^\dagger))\Big)$ is finite dimensional for each $k$ and each point $\mathbf a$. 
By Proposition \ref{newfiber} (i) which we have shown, we only have to check
$i_{\mathbf a}^! \big(C^\cdot(\mathcal L^\dagger)\big) \cong C^\cdot(L_0^\dagger)[-N].$ 
By \cite[Lemma 2.2.3]{Caro}, we have $$\mathcal D_{\mathbf a\to \mathscr P^N, \mathbb Q}^{\dagger}= 
i^{-1}_{\mathbf a}(\mathcal D_{\mathscr P^N, \mathbb Q}^\dagger/\mathcal J\mathcal D_{\mathscr P^N, \mathbb Q}^\dagger).$$
Applying \cite[Proposition 1.1.10]{Caro} to $\mathcal D_{\mathscr P^N, \mathbb Q}^{\dagger}(\infty)$, we have
\begin{eqnarray*}
\mathcal D_{\mathbf a\to \mathscr P^N, \mathbb Q}^{\dagger}(\infty)&\cong&
\mathcal D_{\mathbf a\to \mathscr P^N, \mathbb Q}^{\dagger}(\infty)
\otimes^L_{i_{\mathbf a}^{-1}\mathcal D_{\mathscr P^N, \mathbb Q}^{\dagger}(\infty)}i_{\mathbf a}^{-1}\mathcal D_{\mathscr P^N, \mathbb Q}^{\dagger}(\infty)\\
&\cong&\mathcal D_{\mathbf a\to \mathscr P^N, \mathbb Q}^{\dagger}
\otimes^L_{i_{\mathbf a}^{-1}\mathcal D_{\mathscr P^N, \mathbb Q}^{\dagger}}i_{\mathbf a}^{-1}\mathcal D_{\mathscr P^N, \mathbb Q}^{\dagger}(\infty)\\
&\cong&i^{-1}_{\mathbf a}
(\mathcal D_{\mathscr P^N, \mathbb Q}^\dagger/\mathcal J\mathcal D_{\mathscr P^N, \mathbb Q}^\dagger)
\otimes^L_{i_{\mathbf a}^{-1}\mathcal D_{\mathscr P^N, \mathbb Q}^{\dagger}}i_{\mathbf a}^{-1}\mathcal D_{\mathscr P^N, \mathbb Q}^{\dagger}(\infty)\\
&\cong& i^{-1}_{\mathbf a}\Big(\mathcal D_{\mathscr P^N, \mathbb Q}^\dagger/\mathcal J\mathcal D_{\mathscr P^N, \mathbb Q}^\dagger
\otimes^L_{\mathcal D_{\mathscr P^N, \mathbb Q}^{\dagger}}\mathcal D_{\mathscr P^N, \mathbb Q}^{\dagger}(\infty)\Big)\\
&\cong& i^{-1}_{\mathbf a}\Big( \mathcal D_{\mathscr P^N, \mathbb Q}^\dagger(\infty)/\mathcal J\mathcal D_{\mathscr P^N, \mathbb Q}^\dagger(\infty)\Big)\\
&\cong& i^{-1}_{\mathbf a}\Big( \mathcal O_{\mathscr P^N, \mathbb Q}(\infty)/\mathcal J\mathcal O_{\mathscr P^N, \mathbb Q}(\infty)
\otimes_{\mathcal O_{\mathscr P^N, \mathbb Q}(\infty)}^L  \mathcal D_{\mathscr P^N, \mathbb Q}^\dagger(\infty)\Big),
\end{eqnarray*}
where for the last two isomorphisms, we use the fact that $\mathcal D_{\mathscr P^N, \mathbb Q}^{\dagger}(\infty)$ is flat over 
$\mathcal D_{\mathscr P^N, \mathbb Q}^{\dagger}$ (\cite[4.3.11]{B}), and $\mathcal D_{\mathscr P^N, \mathbb Q}^{\dagger}(\infty)$ is flat over 
$\mathcal O_{\mathscr P^N, \mathbb Q}(\infty)$ (\cite[3.5.2]{B}).
So we have 
\begin{eqnarray*}
i_{\mathbf a}^! \big(C^\cdot(\mathcal L^\dagger)\big)&=& \mathcal D_{\mathbf a
\to \mathscr P^N, \mathbb Q}^{\dagger}(\infty)\otimes^L_{i_{\mathbf a}^{-1}\mathcal D_{\mathscr P^N, \mathbb Q}^\dagger(\infty)}i_{\mathbf a}^{-1}C^\cdot(\mathcal L^\dagger)[-N]\\
&\cong& i^{-1}_{\mathbf a}\Big( \mathcal O_{\mathscr P^N, \mathbb Q}(\infty)/\mathcal J\mathcal O_{\mathscr P^N, \mathbb Q}(\infty)
\otimes_{\mathcal O_{\mathscr P^N, \mathbb Q}(\infty)}^L  \mathcal D_{\mathscr P^N, \mathbb Q}^\dagger(\infty)\Big)
\otimes^L_{i_{\mathbf a}^{-1}\mathcal D_{\mathscr P^N, \mathbb Q}^\dagger(\infty)}i_{\mathbf a}^{-1}C^\cdot(\mathcal L^\dagger)[-N]\\
&\cong& i^{-1}_{\mathbf a}\Big( \mathcal O_{\mathscr P^N, \mathbb Q}(\infty)/\mathcal J\mathcal O_{\mathscr P^N, \mathbb Q}(\infty)
\otimes_{\mathcal O_{\mathscr P^N, \mathbb Q}(\infty)}^L  
\mathcal D_{\mathscr P^N, \mathbb Q}^\dagger(\infty)\otimes^L_{\mathcal D_{\mathscr P^N, \mathbb Q}^\dagger(\infty)}C^\cdot(\mathcal L^\dagger)\Big)[-N]\\
&\cong& i^{-1}_{\mathbf a}\Big( \mathcal O_{\mathscr P^N, \mathbb Q}(\infty)/\mathcal J\mathcal O_{\mathscr P^N, \mathbb Q}(\infty)
\otimes_{\mathcal O_{\mathscr P^N, \mathbb Q}(\infty)}^L  C^\cdot(\mathcal L^\dagger)\Big)[-N]\\
&\cong&K'\otimes^L_{K\langle\mathbf x\rangle^\dagger}C^\cdot(L^\dagger)[-N]\\
&\cong&C^\cdot(L_0^\dagger)[-N].
\end{eqnarray*}
Here $K'$ denotes the residue field of $\mathbf a$, and the last isomorphism follows from Lemma \ref{flat}. This finishes the proof.

\subsection{} Recall that 
\begin{eqnarray*}
D^\dagger&=&\Gamma(\mathscr P^N, \mathcal D^\dagger_{\mathscr P^N, \mathbb Q}(\infty)) \\
&=&\bigcup_{r>1,\;s>1}
\{\sum_{\mathbf v\in\mathbb Z_{\geq 0}^N} f_{\mathbf v}(\mathbf x)\frac{\partial^{\mathbf v}}{\pi^{\vert \mathbf v\vert}}:\;
f_{\mathbf v}(\mathbf x)\in K\{ r^{-1}\mathbf x\},\; 
\Vert f_{\mathbf v}(\mathbf x)\Vert_r s^{\vert \mathbf v\vert} \hbox { are bounded}\},\\
L^{\dagger}&=&\bigcup_{r>1,\;s>1}
\{\sum_{\mathbf w\in \mathbb Z^n\cap \delta} a_{\mathbf w}(\mathbf x)\mathbf t^{\mathbf w}:\; a_{\mathbf w}(\mathbf x)\in K\{r^{-1}\mathbf x\}, \; 
\Vert a_{\mathbf w}(\mathbf x)\Vert_r s^{\vert \mathbf w\vert} \hbox{ are bounded}\},\\
L^{\dagger\prime}&=&\bigcup_{r>1,\;s>1}
\{\sum_{\mathbf w\in C(A)} a_{\mathbf w}(\mathbf x)\mathbf t^{\mathbf w}:\; a_{\mathbf w}(\mathbf x)\in K\{r^{-1}\mathbf x\}, \; 
\Vert a_{\mathbf w}(\mathbf x)\Vert_r s^{\vert \mathbf w\vert} \hbox{ are bounded}\},
\end{eqnarray*}
where $C(A)=\{k_1\mathbf w_1+\cdots +k_N \mathbf w_N:\; k_i\in\mathbb Z_{\geq 0}\}$. 
Consider their algebraic counterpart 
\begin{eqnarray*}
D&=&\{\sum_{\mathbf v\in\mathbb Z_{\geq 0}^N} f_{\mathbf v}(\mathbf x)\partial^{\mathbf v}:\;
f_{\mathbf v}(\mathbf x)\in K[\mathbf x],\; 
  f_{\mathbf v}(\mathbf x)\neq 0 \hbox{ for finitely many } \mathbf v\},\\
  L&=&\{\sum_{\mathbf w\in \mathbb Z^n\cap \delta} a_{\mathbf w}(\mathbf x)\mathbf t^{\mathbf w}:\; a_{\mathbf w}(\mathbf x)\in K[\mathbf x], \; 
a_{\mathbf w}(\mathbf x)\neq 0 \hbox{ for finitely many } \mathbf w\},\\
L'&=& \{\sum_{\mathbf w\in C(A)} a_{\mathbf w}(\mathbf x)\mathbf t^{\mathbf w}:\;
a_{\mathbf w}(\mathbf x)\in K[\mathbf x],\; 
  a_{\mathbf w}(\mathbf x)\neq 0 \hbox{ for finitely many } \mathbf w\}.
\end{eqnarray*}
Then $L'$, $L$ and $L/\sum_{i=1}^n F_{i,\gamma}L$ are $D$-module. 
By the proof of \cite[theorem 4.4]{A1}, 
we have $$D/\sum_{\lambda\in \Lambda}D \Box_\lambda\cong  L^\prime.$$ Combined with Proposition \ref{grobner}, we have 
$$D^\dagger\otimes_{D}L'\cong L^{\dagger\prime}.$$
In \ref{co} where we prove Proposition \ref{coherent}, we construct an exact sequence
$$\bigoplus_{B'} L^{\dagger\prime}\rightarrow \bigoplus_B L^{\dagger\prime} \rightarrow L^\dagger\rightarrow 0,$$
where the direct sums are taken over certain finite sets $B'$ and $B$.
The same proof shows that we have an exact sequence
$$\bigoplus_{B'} L^{\prime}\rightarrow \bigoplus_B L^{\prime} \rightarrow L\rightarrow 0.$$
We have a commutative diagram
\begin{equation}\label{LLL}
\xymatrix{
\bigoplus_{B'} L^{\prime}\ar[r]\ar[d]& \bigoplus_{B} L^{\prime}\ar[r]\ar[d] & L\ar[r] \ar[d] & 0\\
\bigoplus_{B'} L^{\dagger\prime}\ar[r]& \bigoplus_{B} L^{\dagger\prime} \ar[r] &  L^\dagger \ar[r]  & 0.}
\end{equation}
It follow that $$D^\dagger\otimes_{D}L\cong L^{\dagger}.$$
By \cite[Proposition 3.1.1]{B}, $D$ is coherent. So $L'$ is a coherent $D$-module. This implies that $L$ 
and $L/\sum_{i=1}^n F_{i,\gamma}L$ are coherent.  Let $\mathcal D_{\mathscr P^N}$ be the sheaf of 
the usual differential operators on the formal scheme $\mathscr P^N$, and let  
$$\mathcal D_{\mathscr P^N, \mathbb Q}= \mathcal D_{\mathscr P^N}\otimes_{\mathbb Z}\mathbb Q, \quad 
\mathcal D_{\mathscr P^N, \mathbb Q}(\infty)=\mathcal O_{\mathscr P^N,\mathbb Q}(\infty)
\otimes_{\mathcal O_{\mathscr P^N, \mathbb Q}}\mathcal D_{\mathscr P^N, \mathbb Q}.$$ 
They are coherent. By \cite[1.2.1]{Caro1}, the functor $\Gamma(\mathscr P^N, -)$ induces an equivalence of between the category of 
coherent $\mathcal D_{\mathscr P^N, \mathbb Q}(\infty)$-modules and the category of coherent 
$\Gamma(\mathscr P^N, \mathcal D_{\mathscr P^N, \mathbb Q}(\infty))$-modules. Let 
$$D^{an}=
\{\sum_{\mathbf v\in\mathbb Z_{\geq 0}^N} f_{\mathbf v}(\mathbf x)\partial^{\mathbf v}:\;
f_{\mathbf v}(\mathbf x)\in K\langle\mathbf x\rangle^\dagger,\; 
  f_{\mathbf v}(\mathbf x)\neq 0 \hbox{ for finitely many } \mathbf v\}.
$$
We have $\Gamma(\mathscr P^N, \mathcal D_{\mathscr P^N, \mathbb Q}(\infty))\cong D^{an}$. 
Denote the coherent $\mathcal D_{\mathscr P^N, \mathbb Q}(\infty)$-module corresponding
to $D^{an}\otimes_D L$ by 
$\mathcal L$. 
We then have 
$$
 \mathcal D^\dagger_{\mathscr P^N, \mathbb Q}(\infty)\otimes_{\mathcal D_{\mathscr P^N, \mathbb Q}(\infty)}
 \mathcal L\cong  \mathcal L^{\dagger}.$$
It follows that 
\begin{eqnarray}\label{importantiso}
\mathcal D^\dagger_{\mathscr P^N, \mathbb Q}(\infty)\otimes_{\mathcal D_{\mathscr P^N, \mathbb Q}(\infty)}\Big(\mathcal L\Big/\sum_i F_{i, \gamma}\mathcal L\Big)
\cong  \mathcal L^\dagger/\sum_i F_{i, \gamma}\mathcal L^\dagger.
\end{eqnarray}

Fix $\mathbf a \in U$. Then 
$F(\mathbf a, \mathbf x)$ is non-degenerate. Fix a monomial basis $\{\mathbf t^{\mathbf u_1},\ldots, \mathbf t^{\mathbf u_d}\}$ of 
$L_0/\sum_{i=1}^n F_{i,\gamma,\mathbf a}L_0$, where $L_0=K'[\mathbf t^{\mathbf u} :\, \mathbf u\in \mathbb Z^n\cap \delta]$ and $d=n!\text{vol}(\Delta)$. 
By \cite[Proposition 6.6]{A2}, there exists a polynomial $b(\mathbf x)$ depending on $\mathbf a$ such that for any $\mathbf a'\in \mathbb A^N_K$ 
satisfying $b(\mathbf a')\ne 0$,  $F(\mathbf a',t)$ is non-degenerate and $\{\mathbf t^{\mathbf u_1},\ldots, \mathbf t^{\mathbf u_d}\}$ form a basis
of $L_0/\sum_{i=1}^n F_{i,\gamma,\mathbf a'}L_0$. 

\begin{lemma}\label{free} $(L/\sum_{i=1}^n F_{i,\gamma}L)_b$ is a free $K[\mathbf x]_b$ module with basis
$\{\mathbf t^{\mathbf u_1},\ldots, \mathbf t^{\mathbf u_d}\}$. 
\end{lemma}

\begin{proof} Let $M$ be a positive integer so that $\omega(\mathbf u)\in \frac{1}{M} {\mathbb Z}$ for any 
$\mathbf u\in \mathbb Z^n\cap \delta$. Let $L_{m/M}$ be the $K[\mathbf x]$-submodule of $L$ 
generated by all $\mathbf t^{\mathbf u}$ with $\omega (\mathbf u)\le m/M$, and let 
$L^{(m/M)}$ be the $K[\mathbf x]$-submodule generated by $\mathbf t^{\mathbf u}$ with $\omega (\mathbf u)=m/M.$ We have
$b(\mathbf x)=\prod_{m=0}^{M(d+1)} b_{m/M}(\mathbf x),$ where $b_{m/M}(\mathbf x)$ are defined in \cite[\S 6]{A2}.
By \cite[Proposition 6.3]{A2}, for any $h\in L^{(m/M)}$, we have 
\begin{eqnarray*}
h&=&\sum_{\omega(\mathbf u_j)=m/M} A_j \mathbf t^{\mathbf u_j}+\zeta+\sum_{i=1}^n t_i\frac{\partial}{\partial t_i}F(\mathbf x, \mathbf t) \eta_i \\
&=& \sum_{\omega(\mathbf u_j)=m/M} A_j \mathbf t^{\mathbf u_j}+\zeta'+\sum_{i=1}^n F_{i,\gamma} \eta_i'.
\end{eqnarray*}
with $A_j\in K[\mathbf x]_b$, $\zeta, \zeta' \in (L_{(m-1)/M})_b$ and $\eta_i, \eta_i' \in (L^{(m/M-1)})_b$.
Applying the same procedure to $\zeta'$ and after finitely many steps, we see 
$\mathbf t^{\mathbf u_1},\ldots, \mathbf t^{\mathbf u_d}$ generate $(L/\sum_{i=1}^n F_{i,\gamma} L)_b$.
Suppose $f_1,\ldots, f_d\in K[\mathbf x]_b$ and 
$$f_1\mathbf t^{u_1}+\cdots+f_d\mathbf t^{u_d}=\sum_{i=1}^n F_{i,\gamma}\eta_i$$
for some $\eta_i\in L.$ Then for any $\mathbf a'\in \mathbb A^N$ with $b(\mathbf a')\ne 0$, we have
$$f_1(\mathbf a')\mathbf t^{u_1}+\cdots+f_d(\mathbf a') \mathbf t^{u_d}=\sum_{i=1}^n F_{i,\gamma,\mathbf a'}\eta_i(\mathbf a').$$
As $\{\mathbf t^{\mathbf u_1},\ldots, \mathbf t^{\mathbf u_d}\}$ is a basis of $L_0/\sum_{i=1}^n F_{i,\gamma,\mathbf a'}L_0,$ 
we have $f _1(\mathbf a')=\cdots=f _d(\mathbf a')=0$. Thus $f_1=\cdots=f_d=0$ in 
$K[\mathbf x]_b$. 
\end{proof}

\begin{remark}\label{remark} Let $\bar L_0=\kappa [\mathbf x][\mathbf t^{\mathbf u}:\, u\in \mathbb Z\cap \delta]$, and let
$\bar L_{0, m/M}$ and $\bar L_0^{(m/M)}$ be defined similarly as in proof of Lemma \ref{free}. 
The matrix of the homomorphism $\phi_m: (L^{(m/M-1)})^n\to L^{(m/M)}$ defined in \cite[\S 6]{A2} is a lifting of 
a similar homomorphism $\bar{\phi}_m: (\bar L^{(m/M-1)})^n\to \bar L^{(m/M)}$. 
If $\mathbf a$ is a lifting of $\bar{\mathbf a}\in U_0,$ then we have $\mathbf a\in U$. By the same argument in \cite{A2}, we 
can choose $\bar b(\mathbf x)\in \kappa [\mathbf x]$ so that $b$ is a lifting of $\bar b$, and 
$(\bar L_0/\sum_{i=1}^n F_{i,\gamma,\bar{\mathbf a}}\bar L_0)_{\bar b}$ is free with basis
$\{\mathbf t^{\mathbf u_1},\ldots, \mathbf t^{\mathbf u_d}\}$. 
\end{remark}

\begin{lemma}\label{fre} $\Big(\mathcal L/\sum_{i=1}^n F_{i,\gamma}\mathcal L\Big)|_{U_0}$ is a locally free $\mathcal O_{\mathscr P^N,\mathbb Q}|_{U_0}$-module.
\end{lemma} 

\begin{proof} Let $H=L/\sum_{i=1}^n F_{i,\gamma}L$ 
and $\mathcal H=\mathcal L/\sum_{i=1}^n F_{i,\gamma}\mathcal L$. Then $\mathcal H$ is the coherent $\mathcal D_{\mathscr P^N, \mathbb Q}(\infty)$-module 
corresponding
to the coherent $D^{an}$-module $D^{an}\otimes_D H$. By \cite[1.2.1]{Caro1}, $\mathcal H$ is the sheaf associated to the presheaf 
\begin{eqnarray*}
V&\mapsto& \mathcal D_{\mathscr P^N, \mathbb Q}(\infty)(V)\otimes_{D^{an}} D^{an} \otimes_D H\\
&\cong& \mathcal D_{\mathscr P^N, \mathbb Q}(\infty)(V)\otimes_D H
\end{eqnarray*} 
for any open subset $V$ of $\mathscr P^N$. 
For any closed point $\bar{\mathbf  a}$ in $U_0$, choose $\bar b(\mathbf x)$ and $b(\mathbf x)$ as in Remark \ref{remark} and Lemma 
\ref{free}. Let 
$U'_0$ be the open subset of $U_0$ consisting of those points $\bar{\mathbf a}'
\in U_0$ such that 
$\bar b(\mathbf a')\not=0$, and let $U'=\mathrm{sp}^{-1}(U'_0)$. Then 
$\mathcal H|_{U'_0}$ is the sheaf associated to the presheaf 
\begin{eqnarray*}
V&\mapsto& \mathcal D_{\mathscr P^N, \mathbb Q}(\infty)(V) \otimes_D H\\
&\cong&  \mathcal D_{\mathscr P^N, \mathbb Q}(\infty)(V)  \otimes_{\Gamma(U'_0,  \mathcal D_{\mathscr P^N, \mathbb Q}(\infty))}
\Gamma(U'_0,  \mathcal D_{\mathscr P^N, \mathbb Q}(\infty)) \otimes_D H
\end{eqnarray*} 
for any open subset $V$ of $U'_0$. We have 
$$\Gamma(U'_0,  \mathcal D_{\mathscr P^N, \mathbb Q}(\infty)) \cong \mathcal O_{\mathbb P^{N, an}_K}(U') \otimes_{K[\mathbf x]} D.$$
So $$ \Gamma(U'_0,  \mathcal D_{\mathscr P^N, \mathbb Q}(\infty)) \otimes_D H\cong  \mathcal O_{\mathbb P^{N, an}_K}(U') \otimes_{K[\mathbf x]} H.$$
By Lemma \ref{free}, $\mathcal O_{\mathbb P^{N,an}_K}(U') \otimes_{K[\mathbf x]} H$ is a free $\mathcal O_{\mathbb P^{N,an}_K}(U')$-module. Thus 
$\mathcal H|_{U'_0}$ is a free $\mathcal O_{\mathscr P^N,\mathbb Q}|_{U'_0}$-module. 
As $\bar {\mathbf a}$ goes over closed points of $U_0$, $U'_0$ form an open covering of $U_0$. So 
$\mathcal H|_{U_0}$ is a locally free $\mathcal O_{\mathscr P^N,\mathbb Q}|_{U_0}$-module. 
\end{proof}

On the other hand, by Proposition \ref{overholonomic} and \cite[1.3.4]{Caro2},  there exists a Zariski open $V_0\subset \mathscr A^N$ such that 
$H^n(C^\cdot(\mathcal L^{\dagger}))\cong \mathcal L^{\dagger}/ \sum_{i=1}^n F_{i,\gamma} \mathcal L^{\dagger}$ 
defines a convergent isocrystal on $V_0$.

\begin{lemma}\label{isocrys} We have an isomorphism
$(\mathcal L/\sum_{i=1}^n F_{i,\gamma}\mathcal L)|_{U_0\cap V_0} \stackrel\cong\rightarrow (\mathcal 
L^\dagger/\sum_{i=1}^n F_{i,\gamma}\mathcal L^\dagger)|_{U_0\cap V_0}$.
\end{lemma}

\begin{proof} Let $\phi: (\mathcal L/\sum_{i=1}^n F_{i,\gamma}\mathcal L)|_{U_0\cap V_0} \rightarrow (\mathcal 
L^\dagger/\sum_{i=1}^n F_{i,\gamma}\mathcal L^\dagger)|_{U_0\cap V_0}$ be the homomorphism induced by the inclusion $L\hookrightarrow L^\dagger$.  
Both $(\mathcal L/\sum_{i=1}^n F_{i,\gamma}\mathcal L)|_{U_0\cap V_0}$ and 
$(\mathcal L^\dagger/\sum_{i=1}^n F_{i,\gamma}\mathcal L^\dagger)|_{U_0\cap V_0}$ are 
locally free $\mathcal O_{\mathcal P^N, \mathbb Q}|_{U_0\cap V_0}$-modules of the same rank $d=n!\text{vol}(\Delta)$ by Lemma \ref{fre} 
and Proposition \ref{newfiber}. So we only need to show $\phi$ is surjective.
Cover $U_0\cap V_0$ by finitely many affine open subsets $W_i$ so that $(\mathcal 
L/\sum_{i=1}^n F_{i,\gamma}\mathcal L)|_{W_i}$ and $(\mathcal 
L^\dagger/\sum_{i=1}^n F_{i,\gamma}\mathcal L^\dagger)|_{W_i}$ are free $\mathcal O_{\mathscr P^N,\mathbb Q}|_{W_i}$-modules.
In 2.2-2.4, we construct an epimorphism $D^\dagger\to L^{\dagger\prime}$
whose kernel is $\sum_{\lambda\in \Lambda} D^\dagger\Box_\lambda$, and we construct
an epimorphism $\bigoplus_{B} L^{\dagger\prime}\to L^\dagger$. We thus get an epimorphism 
$$\bigoplus_{B} D^{\dagger}\to L^\dagger/\sum_{i=1}^n F_{i,\gamma} L^\dagger$$ and hence an epimorphism of sheaves
$$f^\dagger: \bigoplus_{B}\mathcal D^\dagger_{\mathscr P^N, \mathbb Q}(\infty)\to 
\mathcal L^\dagger/\sum_{i=1}^n F_{i,\gamma}\mathcal L^\dagger.$$
Similarly, we can construct an epimorphism of sheaves 
$$\bigoplus_{B} \mathcal D_{\mathscr P^N, \mathbb Q}(\infty) \to\mathcal L/\sum_{i=1}^n F_{i,\gamma}\mathcal L.$$
Note that  $$\mathscr D_{\mathscr P^N, \mathbb Q}(\infty)|_{\mathscr A^N}\cong \mathscr D_{\mathscr P^N, \mathbb Q}|_{\mathscr A^N}.$$
The following diagram commutes: 
$$\begin{array}{ccc}
\bigoplus_{B}{\mathcal D}_{\mathscr P^N, \mathbb Q}(W_i) &\stackrel f \to& 
 (\mathcal L/\sum_{i=1}^n F_{i,\gamma}\mathcal L)(W_i)\\
 \downarrow&&\qquad\downarrow{\scriptstyle \phi(W_i)}\\
\bigoplus_{B} \widehat{\mathcal D}^{(m)}_{\mathscr P^N, \mathbb Q}(W_i) &\stackrel{f^{(m)}}
\to&(\mathcal L^\dagger/\sum_{i=1}^n F_{i,\gamma}\mathcal L^\dagger)(W_i)
\\
\downarrow&&\parallel\\
\bigoplus_{B}\mathcal D^\dagger_{\mathscr P^N, \mathbb Q}(\infty)(W_i) &\stackrel {f^\dagger}\to 
& (\mathcal L^\dagger/\sum_{i=1}^n F_{i,\gamma}\mathcal L^\dagger)({W_i}).
\end{array}$$
We have 
$$\mathcal D^\dagger_{\mathscr P^N, \mathbb Q}(\infty)(W_i)=\bigcup_m \widehat{\mathcal D}^{(m)}_{\mathscr P^N, \mathbb Q}(W_i).$$
By \cite[Lemma 4.1.2]{B}, the Banach norm on $(\mathcal L^\dagger/\sum_{i=1}^n F_{i,\gamma}\mathcal L^\dagger)({W_i})$ inherited from 
$\mathcal O_{\mathscr P^N, \mathbb Q}(W_i)$ and $\widehat{\mathcal D}^{(m)}_{\mathscr P^N, \mathbb Q}(W_i)$ are equivalent.
Since $\Gamma(W_i, \mathcal O_{\mathscr P^N, \mathbb Q}(W_i))$ is a Noetherian Banach $K$-algebra, 
the image of $$\phi(W_i):  (\mathcal L/\sum_{i=1}^n F_{i,\gamma}\mathcal L)({W_i})\to
(\mathcal L^\dagger/\sum_{i=1}^n F_{i,\gamma}\mathcal L^\dagger)({W_i})$$ is closed by  \cite[Proposition 3.7.2/2]{BGR}. 
Let's prove the image is dense,  which implies that $\phi$ is surjective. 
For any element $\xi\in (\mathcal L^\dagger/\sum_{i=1}^n F_{i,\gamma}\mathcal L^\dagger)({W_i})$, we can find a large $m$ and 
$\eta\in \bigoplus_{B}\widehat{\mathcal D}^{(m)}_{\mathscr P^N, \mathbb Q}(W_i)$ such that $\xi=f^{(m)}(\eta)$. Since 
$\bigoplus_{B} {\mathcal D}_{\mathscr P^N, \mathbb Q}(W_i)$ is dense in 
$\bigoplus_{B}\widehat{\mathcal D}^{(m)}_{\mathscr P^N, \mathbb Q}(W_i)$, 
we can find a sequence $\{\eta_n\}\subset \bigoplus_{B}{\mathcal D}_{\mathscr P^N, \mathbb Q}(W_i)$ with limit $\eta$. Then 
the sequence $\{\phi(f(\eta_n)\}$ has limit $\xi$. 
\end{proof}

\subsection{Proof of Theorem \ref{convergent}} By Lemma \ref{fre} and \cite[Proposition 1.20]{O}, 
$(\mathcal L/\sum_{i=1}^n F_{i,\gamma} \mathcal L)|_{U_0}$ is 
a crystal. Since $(\mathcal L^\dagger/\sum_{i=1}^n F_{i,\gamma}\mathcal L^\dagger)|_{U_0\cap V_0}$ 
is a convergent isocrystal,
$(\mathcal L/\sum_{i=1}^n F_{i,\gamma} \mathcal L)|_{U_0\cap V_0}$ is also a convergent isocrystal by Lemma \ref{isocrys}. 
By \cite[Theorem 2.16]{O}, $(\mathcal L/\sum_{i=1}^n F_{i,\gamma}\mathcal L)|_{U_0}$ is a convergent isocrystal.
By \cite[Proposition 3.1.4]{B5}, we have an isomorphism
$$(\mathcal L/\sum_{i=1}^n F_{i,\gamma}\mathcal L)|_{U_0} \stackrel\cong\to \mathcal D^\dagger_{\mathscr P^N, \mathbb Q}(\infty)|_{U_0}\otimes_{\mathcal D_{\mathscr P^N, \mathbb Q}(\infty)|_{U_0}}(\mathcal L/\sum_{i=1}^n F_{i,\gamma}\mathcal L)|_{U_0}.$$ Combined with the isomorphism (\ref{importantiso}), we get an isomorphism
$$\Big(\mathcal L/\sum_{i=1}^n F_{i,\gamma}\mathcal L\Big)|_{U_0}\stackrel\cong\to \Big(\mathcal L^\dagger/\sum_{i=1}^n F_{i,\gamma}\mathcal L^\dagger\Big)|_{U_0}.$$
Together with Lemma \ref{fre}, we deduce that $\Big(\mathcal L^\dagger/\sum_{i=1}^n F_{i,\gamma}\mathcal L^\dagger\Big)|_{U_0}$ is 
$\mathcal O_{\mathscr P^N, \mathbb Q}|_{U_0}$-coherent.
So
$H^n(C^\cdot(\mathcal L^{\dagger}))|_{U_0}$  is $\mathcal O_{\mathscr P^N, \mathbb Q}|_{U_0}$-coherent. 
For any divisor $T$ in $\mathbb P_{\kappa}^N$ containing $\mathbb P_{\kappa}^N\backslash U_0,$ we have
$$(^\dagger T)\big(H^n(C^\cdot(\mathcal L^{\dagger}))\big)=\mathcal D^\dagger_{\mathscr P^N,\mathbb Q}(^\dagger T) 
\otimes_{\mathcal D^\dagger_{\mathscr P^N, \mathbb Q}(\infty)}H^n(C^\cdot(\mathcal L^{\dagger})).$$
Hence $(^\dagger T)\big(\mathcal H^n(C^\cdot(\mathcal L^{\dagger\cdot}))\big)$ is a coherent $\mathcal D^\dagger_{\mathscr P^N,\mathbb Q}(^\dagger T)$-module 
whose restriction to $\mathscr P^N\backslash T$ is $\Big(\mathcal O_{\mathscr P^N,\mathbb Q}|_{\mathscr P^N\backslash T}\Big)$-coherent. 
By \cite[Th\'eor\`eme 2.2.12]{Caro}, 
$(^\dagger T)\big(H^n(C^\cdot(\mathcal L^{\dagger}))\big)$ is a convergent $F$-isocrystal on $\mathscr P^N\backslash T$ overconvergent along $T.$
This proves Theorem \ref{convergent} (ii).

Consider the complexes 
\begin{eqnarray*}
{\mathcal C}^\cdot:&=&\big[0\to C^0(\mathcal L^\dagger)\to\cdots\to C^n(\mathcal L^\dagger)\to H^n\big(C^\cdot(\mathcal L^\dagger)\big)\to 0\big],\\
C^\cdot:&=&\big[0\to C^0(L^\dagger)\to\cdots\to C^n(L^\dagger)\to H^n\big(C^\cdot(L^\dagger)\big)\to 0\big],\\
C^\cdot_0:&=&\big[0\to C^0(L_0^\dagger)\to\cdots\to C^n(L_0^\dagger)\to H^n\big(C^\cdot(L_0^\dagger)\big)\to 0\big].
\end{eqnarray*}
For any point $\mathbf a$ in $U$,  as in \ref {proof_of_fiber}, using the right-exactness of the tensor functor and 
the flatness of $H^n\big(C^\cdot(\mathcal L^\dagger)\big)|_{U_0}$ over $\mathcal O_{\mathscr P^N,\mathbb Q}|_{U_0}$ proved in Theorem \ref{convergent} (ii), 
we have 
\begin{eqnarray*}
i_{\mathbf a}^! \mathcal C^\cdot &=& \mathcal D_{\mathbf a
\to \mathscr P^N, \mathbb Q}^{\dagger}(\infty)\otimes^L_{i_{\mathbf a}^{-1}\mathcal D_{\mathscr P^N, \mathbb Q}^\dagger(\infty)}i_{\mathbf a}^{-1}
\mathcal C^\cdot[-N]\\
&\cong& i^{-1}_{\mathbf a}\Big( \mathcal O_{\mathscr P^N, \mathbb Q}(\infty)/\mathcal J\mathcal O_{\mathscr P^N, \mathbb Q}(\infty)
\otimes_{\mathcal O_{\mathscr P^N, \mathbb Q}(\infty)}^L  \mathcal D_{\mathscr P^N, \mathbb Q}^\dagger(\infty)
\otimes^L_{\mathcal D_{\mathscr P^N, \mathbb Q}(\infty)}\mathcal C^\cdot\Big)[-N]\\
&\cong& i^{-1}_{\mathbf a}\Big( \mathcal O_{\mathscr P^N, \mathbb Q}(\infty)/\mathcal J\mathcal O_{\mathscr P^N, \mathbb Q}(\infty)
\otimes_{\mathcal O_{\mathscr P^N, \mathbb Q}(\infty)}^L  \mathcal C^\cdot\Big)[-N]\\
&\cong&K'\otimes^L_{K\langle\mathbf x\rangle^\dagger}C^\cdot [-N]\\
&\cong&K'\otimes_{K\langle\mathbf x\rangle^\dagger}C^\cdot [-N]\\
&\cong&C_0^\cdot[-N].
\end{eqnarray*}
By Proposition \ref{newfiber}, $C^\cdot_0$ is acyclic. So $i_{\mathbf a}^! \mathcal C^\cdot$ is acyclic 
for any point 
$\mathbf a$ of $U_0$. By \cite[Lemma 1.3.11]{AC},  this implies that $\mathcal C^\cdot|_{U_0}$ is acyclic. 
So $H^k(C(\mathcal L^\dagger))|_{U_0}=0$ for $k\ne n.$ This proves Theorem \ref{convergent} (i).

\section{Open questions}

In this final section, we state a few open questions on 
the arithmetic and geometry of the overconvergent $F$-isocrystal 
$\mathrm{Hyp}_{U_0}$. 

For any geometric point $\mathbf a$ of $U$ such that its specialization is rational over some 
finite extension $\mathbb{F}_{q^k}$, the $q^k$-adic 
Newton polygon of the action $q^{nk}G_{\mathbf a}$ 
on the fiber $\mathrm{Hyp}(\mathbf a)=i_{\mathbf a}^! \mathrm{Hyp}_{U_0}[N]$ 
is independent of the choice of $k$. 
This polygon is called the Newton polygon of the fibre $\mathrm{Hyp}(\mathbf a)$ and is denoted by $\mathrm{NP(Hyp}(\mathbf a))$. It is an important geometric and 
arithmetic invariant depending 
on the point $\mathbf a$. 

A basic question is how the Newton polygon $\mathrm{NP(Hyp}(\mathbf a))$ 
varies as $\mathbf a$ varies over $U$. The well-known Grothendieck 
specialization theorem 
says that the Newton polygon goes up under specialization. In particular,  
there is a generic Newton polygon as $a$ varies over $U$. The generic 
Newton polygon is the lowest possible Newton polygon as $\mathbf a$ varies over $U$. It is attained for all 
points $\mathbf a$ whose specialization lies in $U_0-T_0$, where $T_0$ is a hypersurface in $U_0$. 
This generic Newton polygon depends not only on the data $\mathbf w_1, \ldots, \mathbf w_N$ and $\gamma_1, \ldots, \gamma_n$, but also on $p$. 
Denote the generic Newton polygon by $\mathrm{GNP}(p)$. The Newton polygon $\mathrm{NP(Hyp}(\mathbf a))$ is equal to $\mathrm{GNP}(p)$ for all $\mathbf a \in U-T$. 

\begin{question} Explicitly determine the generic Newton polygon $\mathrm{GNP}(p)$.
\end{question} 

A good lower bound is obtained by Adolphson-Sperber \cite{AS}. 
For simplicity, 
we assume that $\gamma=0$ below. In this case, 
the lower bound depends only the convex hull $\Delta$ of $\{0, \mathbf w_1, \ldots, \mathbf w_N\}$ 
in $\mathbb{R}^n$ and is independent of $p$. It is called the Hodge polygon and is denoted by 
$\mathrm{HP}(\Delta)$. It has slopes 
$0, 1/D, 2/D, ..., (nD)/D$ with multiplicities $h_0(\Delta), 
h_1(\Delta), ..., h_{nD}(\Delta)$, 
where $D$ is some explicitly determined 
positive integer depending only on $\Delta$, and 
the Hodge numbers $h_i(\Delta)$ are also explicitly determined 
by $\Delta$. 
We have
$$\sum_{i=0}^{nD} h_i(\Delta)=n!{\rm vol}(\Delta).$$
The theorem of 
Adolphson-Sperber \cite{AS} says for every point $a \in U$, 
$$\mathrm{NP(Hyp}(\mathbf a)) \geq \mathrm{HP}(\Delta).$$
In particular, 
$$ \mathrm{GNP}(p) \geq \mathrm{HP}(\Delta).$$
If this inequality becomes an equality, we say that $\mathrm{Hyp}$ 
is generically ordinary at $p$. 
In such an ordinary case, $\mathrm{GNP}(p)$ 
is explicitly determined. Since $\mathrm{GNP}(p)$ depends on $p$ 
but its lower bound $\mathrm{HP}(\Delta)$ does not depend on $p$, one cannot expect that 
$ \mathrm{GNP}(p) = \mathrm{HP}(\Delta)$ in general. In fact, 
a simple ramification argument shows that a necessary 
condition for $\mathrm{Hyp}_{U_0}$ 
to be generically ordinary at $p$ is $p \equiv 1 \mod D$. 
Adolphson-Sperber \cite{AS} conjectured that the converse is also true. It is true in many interesting cases,  
but the full form is false in general (\cite{W1}, \cite{W4}). The ideal congruence condition 
$p \equiv 1 \mod D$ needs to be replaced by a slightly weaker  congruence 
condition $p \equiv 1 \mod D^*$ for some effectively computable 
positive integer $D^*$. That is, a slightly weaker version of the 
Adolpson-Sperber conjecture is true. 
The optimal (smallest) $D^*$ can be quite subtle. 

\medskip
Suppose that $\mathrm{GNP}(p)$ has slopes
$s_0 < s_1< \ldots$ with multiplicities $h_0, h_1, \ldots$. 
By Katz's isogeny theorem \cite{K1} and the Newton-Hodge 
decomposition, there is a 
decreasing filtration of $\mathrm{Hyp}_{U_0}|_{U_0-T_0} $ by convergent sub-$F$-isocrystals 
$$\mathrm{Hyp}_{U_0}|_{U_0-T_0} = M_0 \supset M_1 \supset M_2 \supset ...\supset 0,$$
such that each $M_i/M_{i-1}$
is pure of slope $s_i$ and rank $h_i$. Twisting
by $q^{-s_i}$, the convergent $F$-isocrystal $M_i/M_{i-1}$
becomes a unit root $F$-isocrytal $H_i$ on $U_0-T_0$,
which is convergent but may not be overconvergent.
This unit root crystal $H_i$ corresponds to a lisse $p$-adic
\'etale sheaf of rank $h_i$ on $U_0-T_0$ and hence a $p$-adic representation of $\pi_1(U_0-T_0)$. It is a 
transcendental $p$-adic object and has no
$\ell$-adic counterpart. As explained above, 
in the cases $p \equiv 1 \mod D^*$, it is known that 
the slopes are $\{0, 1/D, \cdots, (nD)/D  \}$ 
and the rank $h_i$ is equal to the Hodge number 
$h_i(\Delta)$ introduced by Adolphson-Sperber \cite{AS} for all $0\leq i \leq nD$. 

\begin{question} 
Determine the image of the $p$-adic representation of $\pi_1(U_0-T_0)\to \mathrm{GL}(h_i)$
corresponding to $H_i$.
\end{question}

In the classical case of Kloosterman sums over the $n$-dimensional torus, 
one has $D=D^*=1$, the slopes are $\{0,1,\cdots, n\}$ and 
the ranks are $h_0=h_1=\cdots=h_n=1$. The question on images of the
rank one $p$-adic representations corresponding to $H_i$ is 
asked by Katz \cite{K2}. The 
image lies inside $\mathrm{GL}(h_i, R)$, where $R$ is a finite ramified 
extension of $\mathbb{Z}_p$. This creates a major new difficulty. 
For instance, the image will be far from being of finite index 
in $\mathrm{GL}(h_i, R)$. 
In the geometric case of the universal family of ordinary elliptic 
curves or abelian varieties, this image is the full group $\mathrm{GL}(h_i, \mathbb{Z}_p)$ by 
Igusa \cite{Ig} and Chai \cite{Ch}. 

\medskip
A further problem is to study the $L$-function of the 
$p$-adic representation corresponding to $H_i$. By Dwork's conjecture proven in Wan (\cite{W2}, \cite{W3}), the unit root $L$-function
$L(H_i, t)$ of $H_i$
is a $p$-adic meromorphic function in $t$. 
It is not rational in $t$ in general. See Liu-Wan \cite{LW} for  
the geometric example of universal ordinary elliptic curves, 
and its close connection to overconvergent 
$p$-adic modular forms. The $L$-function $L(H_i, t)$ will have infinitely 
many zeros and infinitely many poles. 

\begin{question} Determine the order of $L(H_i, t)$ as a $p$-adic meromorphic function, that is, determine 
approximately how
many zeros and poles in a large disk $|t|_p <r$ as $r$
goes to infinity.
\end{question}

The answer depends at least on the rank $h_i$, the dimension of 
$U$, and possibly on finer geometry of $\mathrm{Hyp}_{U_0}|_{U_0-T_0}$.  

\medskip
The above questions are arithmetic in nature. A geometric problem is the following. 

\begin{question} Explicitly determine the $p$-adic horizontal sections of 
the connection for the unit root 
$F$-isocrystal $H_i$. 
\end{question}

For $i=0$, $H_0$ always has rank $1$ and nonzero horizontal sections (unique up to constant) of 
$H_0$ is of the same form as the integral representation 
solution (\ref{intrepgkz}) of the GKZ hypergeometric system, as shown by Adolphson-Sperber \cite{AS2}. 
It would be interesting to understand the higher slope case when $i>0$.

\bigskip

\centerline{\bf Statement and Declaration}

\medskip

On behalf of all authors, the corresponding author states that there is no conflict of interest. 

All data included in this study are available upon request by contact with the corresponding author.

\end{document}